\documentclass[a4paper,11pt, twoside, reqno]{amsart}
\usepackage[inner=2cm,outer=2cm,top=2.5cm,headsep=1cm, footskip=1cm,bottom=3cm]{geometry}
\usepackage[english]{babel}
\usepackage{amsfonts,latexsym,rawfonts,amsmath,amssymb,amsthm,amscd, mathrsfs}
\usepackage{enumerate}
\usepackage{stackrel}
\usepackage{comment}
\DeclareMathAlphabet{\mathpzc}{OT1}{pzc}{m}{it}

\usepackage{mdwlist}
\usepackage{color}
\usepackage[dvipsnames]{xcolor}
\usepackage{array} 
\input xy
\xyoption{all}

\usepackage{tikz}
\usepackage{tikz-cd}
\usetikzlibrary{positioning,arrows}

\newtheorem{prop}{Proposition}[section]
\newtheorem{teo}[prop]{Theorem}
\newtheorem{lem}[prop]{Lemma}
\newtheorem{cor}[prop]{Corollary}

\theoremstyle{definition}
\newtheorem{defi}[prop]{Definition}
\newtheorem{example}[prop]{Example}

\newtheorem{rmk}[prop]{Remark}
\newtheorem{notation}[prop]{Notation}

\newcommand{\NN}{\mathbb{N}}

\newcommand{\bbS}{\mathbb{S}}

\newcommand{\ZZ}{\mathbb{Z}}

\newcommand{\Aa}{\mathcal{A}}

\newcommand{\Cc}{\mathcal{C}}
\newcommand{\Dd}{\mathcal{D}}
\newcommand{\Ee}{\mathcal{E}}

\newcommand{\Pp}{\mathcal{P}}
\newcommand{\Qq}{\mathcal{Q}}
\newcommand{\Ss}{\mathcal{S}}

\newcommand{\Ho}{\mathrm{Ho}}
\newcommand{\Tot}{\mathrm{Tot}}

\newcommand{\Hom}{\mathrm{Hom}}

\newcommand{\Ker}{\mathrm{Ker}}
\newcommand{\Img}{\mathrm{Im}}

\newcommand{\ccat}{\underline{\mathscr{C}}} 
\newcommand{\ddcat}{\underline{\mathscr{D}}} 

\newcommand{\vcat}{\mathscr{V}}

\newcommand{\cpx}{\mathrm{C}_R} 			
\newcommand{\bimod}{\mathrm{bgMod}_R} 		
\newcommand{\vbic}{\mathrm{vbC}_R} 		
\newcommand{\tc}{\mathrm{tC}_R}            		
\newcommand{\bimodinf}{\mathrm{bgMod}_R^\infty} 	

\newcommand{\fm}{\mathrm{fMod}_R} 			
\newcommand{\sfm}{\mathrm{sfMod}_R}			
\newcommand{\fc}{\mathrm{fC}_R} 				
\newcommand{\sfc}{\mathrm{sfC}_R} 				
\newcommand{\btc}{\mathrm{tC}_R^b} 			
\newcommand{\bvbic}{\mathrm{vbC}_R^b} 			
\newcommand{\bbimod}{\mathrm{bgMod}_R^b} 		
\newcommand{\bfm}{\mathrm{fMod}^b_R} 			
\newcommand{\bsfm}{\mathrm{sfMod}_R^b} 		
\newcommand{\bfc}{\mathrm{fC}_R^b}			
\newcommand{\bsfc}{\mathrm{sfC}_R^b} 			

\newcommand{\rtc}{r\text{-}\mathrm{tC}_R}            

\newcommand{\Ainf}[1]{\mathrm{A}_\infty{(#1)}}		
\newcommand{\dAinf}[1]{d\mathrm{A}_\infty{(#1)}}	
\newcommand{\Ainfintc}{A_\infty^{\mathrm{tC}}(R)} 		
\newcommand{\bAinfintc}{A_\infty^{\mathrm{tC}^b}} 	
\newcommand{\fAinf}{\mathrm{f}A_\infty} 			
\newcommand{\sfAinf}{\mathrm{sf}A_\infty} 		
\newcommand{\bsfAinf}{\mathrm{sf}A_\infty^{b}} 		
\newcommand{\bfAinf}{\mathrm{f}A_\infty^{b}} 		

\newcommand{\Tc}{\underline{\mathrm{tC}_R}} 		
\newcommand{\Vbic}{\underline{\mathrm{vbC}_R}} 		
\newcommand{\Cpx}{\underline{\mathrm{C}_R}} 		
\newcommand{\Bimod}{\underline{\mathrm{bgMod}_R}} 	

\newcommand{\wbimod}{\underline{\mathpzc{bgMod}_R}} 

\newcommand{\wtc}{\underline{\mathpzc{tC}_R}}      	
\newcommand{\wbtc}{\underline{\mathpzc{tC}_R^b}} 	
\newcommand{\wfc}{\underline{\mathpzc{fC}_R}} 		
\newcommand{\wsfc}{\underline{\mathpzc{sfC}_R}} 	
\newcommand{\wfm}{\underline{\mathpzc{fMod}_R}} 	
\newcommand{\wsfm}{\underline{\mathpzc{sfMod}_R}} 	
\newcommand{\wbsfm}{\underline{\mathpzc{sfMod}_R^b}} 
\newcommand{\wbsfc}{\underline{\mathpzc{sfC}_R^b}} 

\newcommand{\wtot}{\mathfrak{Tot}} 

\newcommand{\End}{\underline{\mathpzc{End}}} 

\newcommand{\As}{{\mathcal{A}s}}
\newcommand{\dAs}{\mathrm{d}\As}

\newcommand{\simr}[1]{\begin{array}{c}\vspace{-.3cm} \simeq\\ \vspace{-.4cm}\text{\tiny{$#1$}}\vspace{.36cm} \end{array}}
\newcommand{\simre}[1]{\begin{array}{c}\vspace{-.3cm} \sim\\ \vspace{-.4cm}\text{\tiny{$#1$}}\vspace{.36cm} \end{array}}
\newcommand{\antishriek}{\text{!`}}

\newcommand{\pb}{\ar@{}[dr]|{\mbox{\LARGE{$\lrcorner$}}}}
\newcommand{\lra}{\longrightarrow}

\newcommand{\xrightarrowdbl}[2][]{%
  \xrightarrow[#1]{#2}\mathrel{\mkern-14mu}\rightarrow
}

\title{Derived $A$-infinity algebras and their homotopies}

\author{Joana Cirici}
\address[J. Cirici]{
Fachbereich Mathematik und Informatik\\ 
Universit\"{a}t M\"{u}nster\\
Einsteinstra{\ss}e 62\\
48149 M\"{u}nster, Germany}
\email{cirici@uni-muenster.de}

\author{Daniela Egas Santander}
\address[D. Egas Santander]{
Fachbereich Mathematik und Informatik\\
Freie Universit\"{a}t Berlin\\  Arnimallee 3\\ 
14195 Berlin, Germany}
\email{daniela.egas@fu-berlin.de}

\author{Muriel Livernet}
\address[M. Livernet]{
Institut de math\'ematiques de Jussieu-Paris Rive Gauche,                                                                  
UMR 7586 du CNRS, 
Universit\'e Paris Diderot, 
75013 Paris, France }
\email{livernet@math.univ-paris-diderot.fr}

\author{Sarah Whitehouse}
\address[S. Whitehouse]{
School of Mathematics and Statistics\\ 
University of Sheffield\\ S3 7RH\\ England}
\email{s.whitehouse@sheffield.ac.uk }

\thanks{Cirici would like to acknowledge financial support from
the German Research Foundation (SPP-1786) and partial support from
the Spanish Ministry of Economy and Competitiveness (MTM2013-42178-P).
Egas Santander would like to thank the Berlin Mathematical School and CRC 647 for financial support. Livernet would like to acknowledge the project ``Focal" from idex Sorbonne Paris Cit\'e.
}


\begin{document}

\begin{abstract}
The notion of a derived $A$-infinity algebra, considered by Sagave, is a 
generalization of the classical notion of  $A$-infinity algebra, relevant to the case where one works
over a commutative ring rather than a field.
We initiate a study of the homotopy theory of these algebras, by introducing a hierarchy
of notions of homotopy between the morphisms of such algebras. We define $r$-homotopy, for 
non-negative integers $r$,
in such a way that $r$-homotopy equivalences underlie $E_r$-quasi-isomorphisms, defined via an 
associated spectral sequence. We study the special case of twisted complexes (also known as multicomplexes) 
first since it is of independent interest and this simpler case clearly exemplifies the structure we study.
We also give two new interpretations of derived $A$-infinity algebras
as  $A$-infinity algebras in twisted complexes and as $A$-infinity algebras in split filtered cochain complexes. 
\end{abstract}

\maketitle
\tableofcontents

\newpage

\section{Introduction}
The homotopy invariant version of an associative algebra, known as an $A_{\infty}$-algebra,
has become an important idea in many areas of mathematics, including algebra, geometry and mathematical physics.
An introduction to these structures and a discussion of various applications can be found in Keller's survey~\cite{Keller}. 
Kadeishvili's work on minimal models used the existence of $A_{\infty}$-structures in order to classify differential graded algebras over a field up to quasi-isomorphism~\cite{Kadeishvili}. 

The homotopy theory of these algebras has been studied
by several authors, including Prout{\'e}, Grandis and Lef\`evre-Hasegawa~\cite{Proute},~\cite{Grandis},~\cite{LefHas}.
Grandis gave a notion of homotopy between morphisms of $A_\infty$-algebras
via a functorial path construction.  Working over a field, Lef\`evre-Hasegawa establishes the structure of a ``model category without limits'' 
on the category of $A_\infty$-algebras.

In order to formulate a generalization of Kadeishvili's work
over a general commutative ground ring, Sagave considered the notion of
derived $A_\infty$-algebra~\cite{Sag10}. This is related to the notion of 
$D^{(s)}_\infty$-differential $A_\infty$-algebra
considered by Lapin in \cite{Lapin2002}.
Sagave establishes the existence of minimal models for differential
graded algebras (dgas)  by showing that the structure of a derived $A_\infty$-algebra arises on some projective resolution of the
homology of a differential graded algebra. Furthermore, the dga can be recovered up to quasi-isomorphism from
this data.

In~\cite{LRW} an operadic description of derived $A_\infty$-algebras was developed, working with
 non-symmetric operads in the category $\vbic$ of bicomplexes with zero horizontal differential. 
There is an operad $\dAs$ in this category encoding bidgas, which are simply monoids in bicomplexes. 
Derived $A_\infty$-algebras are 
precisely algebras over the operad \[dA_\infty=(\dAs)_\infty= \Omega((\dAs)^{\antishriek}).\] 
Here $(\dAs)^{\antishriek}$ is the Koszul dual cooperad of the operad  $\dAs$, and $\Omega$ 
denotes the cobar construction. Further development of the operadic theory of these algebras was carried
out in~\cite{ALRWZ}. The recent PhD thesis of Maes~\cite{Maes} studies derived $P_\infty$-algebras, replacing the
associative operad $\As$ with a suitable operad $P$.

A derived $A_\infty$-algebra has an underlying \emph{twisted complex}, 
also known as a \textit{multicomplex} or \textit{$D_\infty$-module}. Twisted complexes arise as a
natural generalization of the notion of double complex
by considering a family of ``differentials'' indexed over the non-negative integers.
These objects were first considered by Wall~\cite{Wall} in his work on resolutions for extensions of 
groups and subsequently they have arisen in the work of many authors.
They were studied by
Gughenheim and May~\cite{GugMay} in their approach to differential homological algebra.
Meyer~\cite{Meyer} introduced homotopies between morphisms of twisted complexes and
 proved an acyclic models theorem for these objects.
More recently, twisted complexes have proven to be an important tool in homological perturbation theory
(see for example \cite{Lapin2001}, \cite{Huebschmann}).
Saneblidze~\cite{Saneblidze} introduced projective twisted complexes and showed that every 
(possibly unbounded) chain complex over an abelian category $\Aa$ is weakly equivalent to a projective multicomplex,
provided that $\Aa$ has enough projectives and countable coproducts. 
This result provides a good description of the derived category of $\Aa$.
The work of Sagave on minimal models for differential graded algebras can be thought of as a 
multiplicative enhancement
of Saneblidze's result~\cite{Saneblidze}.

\medskip

In this paper,
we initiate a study of the homotopy theory of derived $A_\infty$-algebras, by introducing a hierarchy
of notions of homotopy between the morphisms of such algebras. We define $r$-homotopy for 
$r\geq 0$ and consider a related notion of $E_r$-quasi-isomorphism. 
Denoting the set of $r$-homotopy equivalences by $\Ss_r$ and the
set of $E_r$-quasi-isomorphisms by $\Ee_r$,
we have the following inclusions.
\begin{center}
\begin{tabular}{lclcccclclc}
&$\Ss_0$ &$\subset$ &$\Ss_1$ &$\subset$ &$\dots$ &$\subset$ &$\Ss_r$ &$\subset$ &$\Ss_{r+1}$&$\subset \dots$\\ 
&$\cap$ &&$\cap$ &&&&$\cap$ &&$\cap$ \\
&$\Ee_0$ &$\subset$ &$\Ee_1$ &$\subset$ &$\dots$ &$\subset$ &$\Ee_r$ &$\subset$ &$\Ee_{r+1}$&$\subset \dots$\\ 
\end{tabular}
\end{center}

We treat the  special case of twisted complexes first, since it is of independent interest and the theory is 
simpler in this case. Every twisted complex has an associated spectral sequence, defined via the 
column filtration of its total complex. In fact, the totalization functor 
gives rise to an isomorphism of categories 
between the category of twisted complexes and the full subcategory of filtered complexes whose objects have split
filtrations (Theorem~\ref{equiv_twisted_split}).
The class of $E_r$-quasi-isomorphisms is given by those morphisms of twisted complexes 
inducing a quasi-isomorphism on the $r$-th stage of their associated spectral sequences. 
The notion of $r$-homotopy for twisted complexes that we consider corresponds to 
the notion of homotopy of order $r$ introduced in~\cite{CaEil} and further developed in~\cite{CG2} 
in the context of filtered complexes. We study the localized category of twisted complexes
with respect to $r$-homotopies (Theorem~\ref{QuotientcatTC}).  
We present several equivalent formulations of $r$-homotopy: via a functorial
path, via explicit formulas and an operadic approach (Theorem~\ref{teo:coderh}). 

A substantial part of the paper is devoted to developing new interpretations of derived $A_\infty$-algebras.
As well as being interesting in themselves, these new viewpoints are used in establishing properties of homotopy
and they provide a key idea
for the proof of the equivalence of the
various formulations of $r$-homotopy in the derived $A_\infty$ case. 
Firstly, we show that derived $A_\infty$-algebras can be interpreted as $A_\infty$-algebras in twisted complexes
(Theorem~\ref{dA_Atc}).
This formulation has the potential to be a very useful tool for the future development of 
different aspects of the theory of derived $A_\infty$-algebras.
Secondly, under suitable boundedness conditions, we show that derived $A_\infty$-algebras can be viewed as split filtered $A_\infty$-algebras
(Theorem~\ref{dA-filtered-bounded-tc}). 
This result allows one to transfer known constructions in the category of $A_\infty$-algebras
to the category of derived $A_\infty$-algebras, by checking compatibility with filtrations.

The context for these new interpretations is the 
theory of operadic algebras for monoidal categories over a base, as developed by Fresse~\cite{Fresse}. 
We endow the categories of twisted complexes and filtered complexes with a monoidal structure over the category of vertical bicomplexes and use this to enrich them over vertical bicomplexes.  
This allows us to formulate these algebra structures by means of an enriched endomorphism operad.  
The totalization functor extends to the enriched setting and gives an isomorphism between the vertical bicomplexes-enriched categories of twisted complexes and split filtered complexes (Theorem~\ref{prop:weirdtot}).  
We use this isomorphism to show that, under certain boundedness conditions, the different intepretations of derived $A_\infty$-algebras are equivalent.

We then turn to $r$-homotopy for derived $A_\infty$-algebras. Here,
the class of $E_r$-quasi-isomorphisms is defined by lifting
$E_r$-quasi-isomorphisms of the underlying twisted complexes.
The notion of $r$-homotopy that we consider arises as a combination of the notion of $r$-homotopy for twisted complexes 
and the classical notion of homotopy between morphisms of $A_\infty$-algebras.
We again present different approaches: via a functorial path construction, via explicit formulas and in 
operadic terms. For the operadic description we
formulate the general notion of a $(g,f)$-coderivation for morphisms $g,f$ of cofree coalgebras over
a (non-symmetric) operad. In Theorem~\ref{r-homotopy-agree}, we show that the different approaches are equivalent.

The results of this paper set up the foundations for a homotopy theory of derived $A_\infty$-algebras,
contextualizing the ad-hoc notion of homotopy introduced by Sagave in a very particular case
and generalizing the work of Grandis on functorial paths for $A_\infty$-algebras.
We expect that
both the new descriptions of derived $A_\infty$-algebras and 
the properties of homotopies developed here 
will allow us to endow the category of derived $A_\infty$-algebras  with 
 the structure of a model category without limits in the future,
 with weak equivalences being $E_r$-quasi-isomorphisms.

\medskip

The paper is organized as follows. Section~\ref{sec:prelim} covers background material introducing some of 
the categories we work with. The notion of $r$-homotopy for twisted complexes is covered in Section~\ref{sec:twcx}.
Our new interpretations of derived $A_\infty$-algebras are presented in Section~\ref{sec:newder}. Finally,
Section~\ref{sec:der} studies $r$-homotopy for derived $A_\infty$-algebras.

\subsection*{Acknowledgements}
The authors would like to thank the Banff International Research Station and the organizers of the Women in Topology workshop in April 2016
for bringing us together to work on this paper. 
We would also like to thank Gabriel Drummond-Cole, Fernando Muro, Emily Riehl, Agust\'{\i} Roig and Sinan Yalin
for helpful conversations. Further thanks go to the anonymous referee for his/her careful report and helpful suggestions.

\section{Preliminaries}
\label{sec:prelim}

We first set up some notation which we will use throughout this paper.
\begin{notation}Throughout this paper
$R$ will denote a commutative ring with unit. Unless stated otherwise, all tensor products will be taken over $R$.
Let $\Cc$ be a category and let $A,B$ be arbitrary objects in $\Cc$. 
We denote by $\Hom_{\Cc}(A,B)$ the set of morphisms from $A$ to $B$ in $\Cc$.  
If $(\Cc,\otimes,1)$ is symmetric monoidal closed, then we denote its internal hom-object by $[A,B]\in\Cc$ in which case we have by definition a bijection
\[\Hom_{\Cc}(A\otimes B,C)\cong \Hom_{\Cc}(A,[B,C])\]
which is natural in $A, B$ and $C$.  
\end{notation}

\subsection{Filtered modules and filtered cochain complexes}

We collect some algebraic preliminaries about filtered $R$-modules and filtered complexes of $R$-modules.
Our complexes are cochain complexes and our filtrations are increasing.

\begin{defi}\label{def:fm_obj}
A \textit{filtered $R$-module} $(A,F)$ is given by a family of $R$-modules $\{F_pA\}_{p\in\ZZ}$ 
indexed by the integers such that $F_{p-1}A\subseteq F_pA$ for all $p\in\ZZ$ and $A=\cup F_pA$.
A \textit{morphism of filtered modules} is a morphism $f:A\to B$ of $R$-modules
which is \textit{compatible with filtrations}:
$f(F_pA)\subset F_pB$ for all $p\in\ZZ$.
\end{defi}

\begin{defi}The tensor product of two filtered $R$-modules $(A,F)$ and $(B,F)$ is a filtered $R$-module, with
\[F_p(A\otimes B):=\sum_{i+j=p} \Img(F_iA\otimes F_jB\lra A\otimes B).\]
This makes the category of  filtered $R$-modules into a symmetric monoidal category,
where the unit is given by $R$ with the trivial filtration
$0=F_{-1}R\subset F_0R=R$. 
\end{defi}

Denote by $\cpx$ the category of cochain complexes of $R$-modules. 
The standard tensor product endows
this with a symmetric monoidal structure, with unit $R$ concentrated in degree zero.

\begin{defi}\label{def:fc_obj}A \textit{filtered complex} $(K,d,F)$ is a complex $(K,d)\in\cpx$ together with
a filtration $F$ of each $R$-module $K^n$ such that 
$d(F_pK^n)\subset F_pK^{n+1}$ for all $p,n\in\ZZ$.
\end{defi}
We denote by $\fc$ the category of filtered complexes of $R$-modules. 
Its morphisms are given by morphisms of complexes $f:K\to L$ compatible with filtrations: $f(F_pK)\subset F_pL$ for all $p\in\ZZ$.
It is a symmetric monoidal category, with the filtration on the tensor product defined as above. The symmetry isomorphisms are inherited from the standard ones on cochain complexes.
\medskip

We next recall the notion of homotopy of order $r$ due to Cartan and Eilenberg.
\begin{defi}~\cite[p321]{CaEil}
Let $f,g:(K,F)\lra (L,F)$ be two morphisms of filtered complexes and let $r\geq 0$ be an integer.
A \textit{homotopy of order $r$} from $f$ to $g$ is a map of graded $R$-modules $H:K^*\to L^{*-1}$
such that $dH+Hd=g-f$ and $H(F_pK^n)\subset F_{p+r}L^{n-1}$ for all $p,n\in\ZZ.$
\end{defi}

\begin{rmk}
Every filtered complex $(K,d,F)$ has an associated spectral sequence $\{E_r(K),\delta_r\}$
and this assignment is functorial. Given a homotopy of order $r$ between morphisms of filtered complexes
$f,g:(K,F)\to (L,F)$,
the induced morphisms at the $k>r$ stages of the spectral sequences coincide: 
$f_k^*=g_k^*:E_k(K)\to E_k(L)$ for all $k>r$ (see~\cite[Proposition XV.3.1]{CaEil}). 
This result indicates how the notion of homotopy of order $r$ 
is  suitable for studying the \textit{$r$-derived category} defined by inverting
those morphisms of filtered complexes which induce an isomorphism at
the $E_{r+1}$-stage of the associated spectral sequences (see~\cite{Paranjape},~\cite{CG2}).
\end{rmk}

\subsection{Bigraded modules, vertical bicomplexes, sign conventions}
\label{subsection:bgmods}
We consider $(\ZZ,\ZZ)$-bigraded $R$-modules $A=\{A_i^j\}$,
where elements of $A_i^j$ are said to have bidegree $(i,j)$. 
We sometimes
refer to $i$ as the \textit{horizontal degree} and $j$ the \textit{vertical degree}.
The \textit{total degree} of an element $a\in A_i^j$ is $|a|=j-i$.
A morphism of bidegree $(p,q)$ maps $A_i^j$ to $A_{i+p}^{j+q}$.
The \textit{tensor product} of two bigraded $R$-modules $A$ and $B$ is the bigraded $R$-module $A\otimes B$ 
given by 
	\[
		(A\otimes B)_i^j:=\bigoplus_{p,q} A_p^q\otimes B_{i-p}^{j-q}.
	\]

We denote by $\bimod$ the category whose objects are bigraded $R$-modules
and whose morphisms are morphisms of bigraded $R$-modules of bidegree $(0,0)$.
 It is symmetric monoidal with the above
tensor product.

We introduce the following scalar product notation for bidegrees: for $x,y$ of bidegree $(x_1,x_2)$, $(y_1,y_2)$ 
respectively, we let $\langle x,y \rangle=x_1y_1+x_2y_2$.

The symmetry isomorphism 
\[\tau_{A\otimes B}^{\bimod}: A\otimes B \to B\otimes A\] is
given by
\[a\otimes b \mapsto (-1)^{\langle a, b \rangle}b\otimes a.\]

We follow the \textit{Koszul sign rule}:
if $f:A\rightarrow B$ and $g:C\rightarrow D$ are bigraded morphisms,
then the morphism $f\otimes g: A\otimes C\rightarrow B\otimes D$ is defined by
	\[
	(f\otimes g)(a\otimes c):=(-1)^{\langle g, a\rangle} f(a)\otimes g(c).
	\]

\begin{defi}
\label{def:vbicx}
A \emph{vertical bicomplex} is a bigraded $R$-module $A$ equipped with a vertical differential
$d^A: A \longrightarrow A$
of bidegree $(0,1)$. A \emph{morphism of vertical bicomplexes} is a  morphism of bigraded modules 
of bidegree $(0,0)$ commuting with the vertical differential.
\end{defi}

We denote by $\vbic$ the category of vertical bicomplexes.
The tensor product of two vertical bicomplexes $A$ and $B$ is given by endowing the tensor product of
underlying bigraded modules with vertical differential
   $d^{A\otimes B}:=d^A\otimes 1+1\otimes d^B:(A\otimes B)_u^v\to (A\otimes B)_u^{v+1}$.
This makes $\vbic$ into a symmetric monoidal category.

The symmetric monoidal categories $(\cpx,\otimes,R)$, $(\bimod,\otimes,R)$ and $(\vbic,\otimes,R)$ are related by  embeddings 
$\cpx \hookrightarrow \vbic$ and $\bimod \hookrightarrow \vbic$
which are monoidal and full.

\begin{defi}
Let $A, B$ be bigraded modules.  We define $[A,B]_*^*$ to be the bigraded module of morphisms of bigraded modules $A\to B$. 
Furthermore, if $A, B$ are vertical bicomplexes, and $f\in [A,B]_u^v$, we define 
\[
	\delta (f):=d^Bf-(-1)^{v}f d^A.
	\]
\end{defi}

\begin{lem}
If $A, B$ are vertical bicomplexes, then $([A,B]_*^*,\delta)$ is a vertical bicomplex.
\end{lem}
\begin{proof}
A direct computation gives that $\delta^2=0$.
\end{proof}

This gives an internal hom on $\vbic$, making it symmetric monoidal closed. 
It restricts to give the standard internal hom on the categories $\bimod$ and $\cpx$.

We denote by $\Vbic$, $\Cpx$, and $\Bimod$, the categories of vertical bicomplexes, complexes, and bigraded modules respectively, enriched over themselves via their symmetric monoidal closed structure.
\smallskip

We will use a standard (vertical) shift $S$ of bigraded modules, following Sagave's conventions,
as in the first part of~\cite[Section 4]{Sag10}. So $S$ is the shift of bidegree $(0,1)$;
it is an endofunctor on the category of bigraded $R$-modules with morphisms of arbitrary bidegree, where $S(A)_i^j=A_i^{j+1}$ and on morphisms
$Sf=(-1)^v f$, if $f$ has bidegree $(u,v)$.

We write $\sigma$ for the corresponding natural transformation from $S$ to the identity;
this means that
\[f\sigma_A = \sigma_B S(f)=(-1)^v \sigma_B f.\]
For every bigraded $R$-module $A$, $\sigma_A$ is an isomorphism of bidegree $(0,1)$.
Then
$\Psi_k$ is the induced isomorphism from morphisms on a $k$-fold tensor power:
	\[
	\Psi_k:\Hom(A^{\otimes k}, B)\to \Hom((SA)^{\otimes k}, SB),
	\]
 where $\sigma_{B} \Psi_k(f)=(-1)^{\langle  \Psi_k(f), \sigma \rangle} f\sigma_{A}^{\otimes k}$.
If the bidegree of $f$ is $(u,v)$, then the bidegree of $\Psi_k(f)$ is $(u,v+k-1)$, and then since that of $\sigma$ is $(0,1)$, we have ${\langle  \Psi_k(f), \sigma \rangle}=v+k-1$.

\section{Twisted complexes and $r$-homotopy}
\label{sec:twcx}
In this section, we recall some key properties of the category of twisted complexes, also called multicomplexes in the literature.
We define the totalization 
functor from twisted complexes to filtered cochain complexes and show that 
it induces an isomorphism of categories onto its image. We then introduce $r$-homotopies
between morphisms of twisted complexes, consider
their interplay with spectral sequences and study the localized category $\tc[\Ss_r^{-1}]$.

\subsection{The category of twisted complexes (or multicomplexes)}

\begin{defi}\label{def:tc_objects}
A \textit{twisted complex $(A,d_m)$} is a bigraded $R$-module $A=\{A_i^j\}$
together with a family of morphisms $\{d_m:A\lra A\}_{m\geq 0}$ of bidegree $(-m,-m+1)$
such that for all $m\geq 0$,
\begin{equation}\label{condition_twisted_complex}
\sum_{i+j=m} (-1)^id_id_j=0. \tag{$A_{m1}$}
\end{equation}
\end{defi}

\begin{defi}\label{def:tc_mor}
 A \textit{morphism of twisted complexes} $f:(A,d_m^A)\to (B,d_m^B)$ is given by a family of morphisms
of $R$-modules
 $\{f_m:A\lra B\}_{m\geq 0}$ of bidegree $(-m,-m)$
such that for all $m\geq 0$,
\begin{equation}
 \sum_{i+j=m}d_i^Bf_j=\sum_{i+j=m}(-1)^if_id_j^A. \tag{$B_{m1}$}
\end{equation}
\end{defi}
The composition of morphisms is given by $(g\circ f)_m:=\sum_{i+j=m} g_i f_j$.
A morphism $f=\{f_m\}_{m\geq 0}$ is said to be \textit{strict} if $f_i=0$ for all $i>0$.
The identity morphism $1_A:A\to A$ is the strict morphism given by ${(1_{A})}_0(x)=x$.
A morphism $f=\{f_i\}$ is an isomorphism if and only if $f_0$ is an isomorphism of bigraded $R$-modules.
Indeed, an inverse of $f$ is obtained from an inverse of $f_0$ by solving a triangular system.

Denote by $\tc$ the category of twisted complexes.
Also, denote by  $\bimodinf$ the full subcategory of $\tc$ whose objects are twisted complexes with trivial structure i.e., $d_m=0$ for all $m\geq 0$.  

The following construction endows $\tc$ with a symmetric monoidal structure.

\begin{lem}
The category $(\tc,\otimes,R)$ is symmetric monoidal, where the monoidal structure is given by the bifunctor 
\[\otimes:\tc \times \tc \to \tc\]
which on objects is given by
$((A,d_m^A),(B,d_m^B))\mapsto (A\otimes B, d^A_m\otimes 1 + 1\otimes d^B_m)$
and on morphisms is given by
$(f, g)\mapsto f\otimes g$, where 
$(f\otimes g)_m:=\sum_{i+j=m} f_i\otimes g_{j}$.  In particular, by the Koszul sign rule we have that  $(f_i\otimes g_j)(a\otimes b)=(-1)^{<g_j,a>} f_i(a)\otimes g_j(b)$. 
The symmetry isomorphism is given by the strict morphism of twisted complexes
\[\tau_{A\otimes B}^{\tc}: A\otimes B \to B\otimes A\]
defined by
\[a\otimes b \mapsto (-1)^{\langle a, b \rangle}b\otimes a.\]

This functor describes a symmetric monoidal structure on $\bimodinf$ by restriction.
\end{lem}

\begin{proof}
We check that $(A\otimes B, \partial_m=d^A_m\otimes 1 + 1\otimes d^B_m)$ is a twisted complex:
 for all $m\geq 0$ we have
	\[
		\sum_{i+j=m} (-1)^i\partial_i\partial_j=\sum_{i+j=m}(-1)^i 
			(d^A_id ^A_j\otimes 1+1\otimes d^B_i d^B_j+d^A_i\otimes d^B_j+(-1)^{ij+(1-i)(1-j)} 
			d^A_j\otimes d^B_i)=0.
	\]
Similarly, one checks that $f\otimes g$ is a morphism of twisted complexes. 
It only remains to see that this construction is functorial.  A direct computation shows that 
	\[		((f\otimes g) \circ (f'\otimes g'))_m = (f\circ f'\otimes g \circ g')_m.\qedhere \]
\end{proof}

We extend the internal hom on bigraded modules to twisted complexes.  

\begin{lem}
Let $A, B$ be twisted complexes.
For $f\in [A,B]_u^v$, setting
	\[
	(d_if):=(-1)^{i(u+v)}d_i^Bf-(-1)^{v}f d_i^A,
	\]
for $i\geq 0$, endows $[A,B]_*^*$ with the structure of a twisted  complex.
\end{lem}

\begin{proof}
It is a matter of calculation that $\sum_i (-1)^i (d_i d_{m-i}f)=0$, for all $m\geq 0$.
Thus the maps $d_i: [A,B]_u^v\to [A,B]_{u-i}^{v-i+1}$ make $[A,B]_*^*$ into a twisted complex.
\end{proof}

This construction gives an internal hom and
we denote by $\Tc$ the category of twisted complexes enriched over 
itself via this symmetric  monoidal closed structure.

\begin{defi}Let $r\geq 0$ be a non-negative integer.
An \textit{$r$-bigraded complex} is a twisted complex $(A,d_m)$ such that $d_m=0$ for all $m\neq r$.
\end{defi}
Note that in this case, we have $d_r d_r=0$.
For $r=0$, this coincides with the notion of vertical bicomplex.
Denote by $\rtc$ the full subcategory of $\tc$ whose objects are $r$-bigraded complexes.
The prototypical example of such an object is given by the $r$-th term of the spectral sequence associated with a 
twisted complex, as we will see in Section \ref{section:SS_of_tc}.

\subsection{Total cochain complex of a twisted complex}
\hspace{0.5cm}
\begin{defi}
The \textit{total graded $R$-module $\Tot(A)$} of a bigraded $R$-module $A=\{A_i^j\}$ is 
given by
\[
\Tot(A)^n:=\prod_{i\leq 0} A_i^{n+i}\oplus\bigoplus_{i>0} A_i^{n+i}.\]
The \textit{column filtration of $\Tot(A)$} is the filtration given by
$F_p\Tot(A)^n:=\prod_{i\leq p} A_i^{n+i}$ for all $p,n\in\ZZ$.
\end{defi}

We will show that, via the totalization functor, the category of twisted complexes
is isomorphic to a full subcategory of that of filtered complexes.
We use the following.

\begin{defi}\label{def:sfc}
A filtered complex $(K,d,F)$ is said to be \textit{split} if $K=\Tot(A)$ 
is the total graded module of a bigraded $R$-module $A=\{A_i^j\}$ and $F$ is
the column filtration of $\Tot(A)$.  We denote by $\sfc$ the full subcategory of $\fc$ whose objects are split filtered complexes.

Given a twisted complex $(A,d_m)$, define a map $d:\Tot(A)\to \Tot(A)$ of degree $1$ by letting
\[d(a)_j:=\sum_{m\geq 0} (-1)^{mn}d_m(a_{j+m}),\text{ for }a=(a_i)_{i\in\ZZ}\in \Tot(A)^n,\]
where $a_i\in A_i^{n+i}$ denotes the $i$-th component of $a$, and $d(a)_j$ denotes the $j$-th component of $d(a)$.
Note that, for a given $j\in \ZZ$ there is a sufficiently large $m\geq 0$ 
such that $a_{j+m'}=0$ for all $m'\geq m$. Hence $d(a)_j$ is given by a finite sum.
Also, for $j$ sufficiently large, 
one has $a_{j+m}=0$ for all $m\geq 0$, which implies $d(a)_j=0$.

Given a morphism $f:(A,d_m)\to (B,d_m)$ of twisted complexes,
let $\Tot(f):\Tot(A)\to \Tot(B)$ be the map of degree $0$ defined by
\[(\Tot(f)(a))_j:=\sum_{m\geq 0} (-1)^{mn}f_m(a_{j+m}),\text{ for }a=(a_i)_{i\in\ZZ}\in \Tot(A)^n.\]\end{defi}

\begin{teo}\label{equiv_twisted_split}
The assignments $(A,d_m)\mapsto (\Tot(A),d,F)$, where $F$ is the column filtration of $\Tot(A)$,
and $f\mapsto \Tot(f)$
define a functor $\Tot:\tc\lra \fc$
which is an isomorphism of categories when restricted to its image $\sfc$.
\end{teo}

\begin{proof}
Let $(A,d_m)$ be a twisted complex and let $a=(a_i)_{i\in\ZZ}\in\Tot(A)^n$.
To see that  $(\Tot(A),d)$ is a complex it suffices to note that:
\[
(dd(a))_j=\sum_{p\geq 0}\sum_{m\geq 0}(-1)^{p(n+1)+mn} d_p(d_m(a_{j+m+p}))=\sum_{l\geq 0}(-1)^{ln}\sum_{\substack{m,p\geq 0,\\ m+p=l}}(-1)^{p}d_pd_m(a_{j+l})=0.\]
One easily verifies that $F_{p-1}\Tot(A)^n\subset F_p\Tot(A)^n$ and that $d(F_{p}\Tot(A)^n)\subset F_p\Tot(A)^{n+1}$.

Let $f:(A,d_m^A)\to (B,d_m^B)$ be a morphism of twisted complexes. If $a=(a_i)\in \Tot(A)^n$ then
\begin{multline*}
(\Tot(f)\circ d (a))_j=\sum_{m\geq 0}\sum_{p+q=m}(-1)^{p(n+1)}(-1)^{qn} f_pd_q(a_{j+m})=\\
\sum_{m\geq 0}(-1)^{mn}\sum_{p+q=m}d_qf_p(a_{j+m})=\sum_{m\geq 0}\sum_{p+q=m}(-1)^{(q+p)n}d_qf_p(a_{j+m})=(d\circ \Tot(f)(a))_j.
\end{multline*}
Note that $\Tot(f)$ is compatible with the filtration $F$ and that $\Tot(fg)=\Tot(f)\Tot(g)$.
This proves that $\Tot$ is a functor with values in the category of split filtered complexes.

We next define a functor $\Tot^{-1}:\sfc\rightarrow \tc$ inverse to the restriction of $\Tot$ onto its image. Let $(\Tot(A),d,F)$ be a split filtered complex,
where $A=\{A_i^j\}$ is a bigraded $R$-module.
For all $m\geq 0$, let $d_m:A\to A$ be the morphism of bidegree $(-m,-m+1)$ defined by
$d_m(a)=(-1)^{nm}d(a)_{i-m}$, where $a\in A_i^{n+i}$ and $d(a)_k$ denotes the $k$-th component of $d(a)$, which lies in $A_k^{n+1+k}$.
Since $d$ is compatible with the filtration $F$, we have $d_i=0$ for $i<0$.
Then $(A,d_m)$ is a twisted complex and its filtered total complex is $(\Tot(A),d,F)$.
Lastly, let $f:(\Tot(A),d,F)\to (\Tot(B),d,F)$ be a morphism of split filtered complexes. For all $m\geq 0$,
let $f_m:A\to B$ be the morphism of bidegree $(-m,-m)$ defined by
$f_m(a)=(-1)^{nm}f(a)_{i-m}$, where $a\in A_i^{n+i}$ and $f(a)_{k}$ denotes the $k$-th component of $f(a)$, which lies in $B_k^{n+k}$.
Since $f$ is compatible with the filtration $F$, we have that $f_i=0$ for $i<0$.
Then the family $\{f_m\}_{m\geq 0}$ is a morphism of twisted complexes whose total morphism is $f$.
It is straightforward to see that the above constructions define an inverse functor to the restriction of $\Tot$.
\end{proof}

\begin{rmk}
Strict morphisms of twisted complexes correspond, via the above isomorphism of categories,
to strict morphisms of split filtered complexes, that is, morphisms preserving the splittings.
\end{rmk}

We will also consider the following bounded versions of our categories, since the totalization functor has better properties when restricted to these.

\begin{defi}\label{bddcats}
We let 
$\btc$, $\bvbic$, $\bbimod$ be the full subcategories of $(\NN,\ZZ)$-graded twisted complexes, vertical bicomplexes and bigraded modules respectively. We let $\bfm$, $\bsfm$, $\bfc$, $\bsfc$ be the full subcategories of (split) non-negatively filtered modules, respectively complexes,
i.e. the full subcategories of objects $(K, F)$ such that $F_p K^n=0$ for all $p<0$.  We refer to all of these as the \textit{bounded subcategories} of $\tc$, $\vbic$, $\bimod$, $\fm$, $\sfm$,  $\fc$ and $\sfc$
respectively.
\end{defi}

In the following proposition, we show that the monoidal structures of twisted complexes and filtered complexes are compatible under the totalization functor.   

\begin{prop}\label{prop:lax}
The functors 
$\Tot:\bimod\to \fm$ and $\Tot:\tc\to \fc$
are lax symmetric monoidal, with structure maps
	\[
		\epsilon: R\rightarrow \Tot(R)
			\quad \quad \text{and} \quad \quad		
		\mu_{A,B}: \Tot(A)\otimes \Tot(B)\rightarrow \Tot(A\otimes B),
	\]
given by $\epsilon=1_R$ and for $a=(a_i)_i  \in\Tot(A)^{n_1}$ and $b=(b_j)_j\in\Tot(B)^{n_2}$,
	\[
		(\mu_{A,B}(a\otimes b))_k:=\sum_{k_1+k_2=k}(-1)^{k_1n_2} a_{k_1}\otimes b_{k_2} .
	\]
When restricted to the bounded case,
the functors 
$\Tot:\bbimod\to \bfm$ and
$\Tot:\btc\to \bfc$
are strong symmetric monoidal functors.
\end{prop}
\begin{proof}
Clearly $\epsilon$ is a map of filtered complexes and a direct computation shows that the same is true for $\mu_{A,B}$.  We now show that $\mu_{A,B}$ respects the symmetric structure i.e., that Diagram (1) commutes.
\[\begin{array}{cc}
\begin{tikzpicture}[scale=0.01]
  \node (a) {$\Tot(A)\otimes \Tot(B)$};
  \node[right=1cm of a] (b) {$\Tot(A\otimes B)$};
  \node[below= 1.5cm of a] (c) {$\Tot(B)\otimes \Tot(A)$};
  \node[below= 1.5cm of b] (d) {$\Tot(B\otimes A)$};
  \draw[->]
   (a) edge node[above] {$\mu_{A,B}$}  (b)
   (a) edge node[left] {$\tau^{\fc}_{A,B}$} (c)
   (c) edge node[below] {$\mu_{B,A}$}  (d)
   (b) edge node[right] {$\Tot(\tau^{\tc}_{A,B})$} (d);
\end{tikzpicture} 
& 
\begin{tikzpicture}[scale=0.01]
  \node (a) {$\Tot(A)\otimes \Tot(B)$};
  \node[right=1cm of a] (b) {$\Tot(A\otimes B)$};
  \node[below= 1.5cm of a] (c) {$\Tot(A')\otimes \Tot(B')$};
  \node[below= 1.5cm of b] (d) {$\Tot(A'\otimes B')$};
  \draw[->]
   (a) edge node[above] {$\mu_{A,B}$}  (b)
   (a) edge node[left] {$\Tot(f)\otimes\Tot(g)$} (c)
   (c) edge node[below] {$\mu_{A',B'}$}  (d)
   (b) edge node[right] {$\Tot(f\otimes g)$} (d);
\end{tikzpicture}\\
(1) & (2)
\end{array}\]
Let $a\otimes b \in \Tot(A)^{n_1}\otimes \Tot(B)^{n_2}$, with $n_1+n_2=n$.  Then
\begin{align*}
(
\Tot(\tau_{A,B}^{\tc})\mu_{A,B}(a\otimes b))_k
&=\sum_{k_1+k_2=k} (-1)^{k_1n_2+k_1k_2+(k_1+n_1)(k_2+n_2)} b_{k_2}\otimes a_{k_1}\\
&= \sum_{k_1+k_2=k} (-1)^{n_1n_2+k_2n_1} b_{k_2}\otimes a_{k_1}\\
&=(\mu_{B,A}\tau_{A,B}^{\fc}(a\otimes b))_k.
\end{align*}

The commutativity of Diagram (2) is obtained from the following computation.
Let $a\otimes b \in \Tot(A)^{n_1}\otimes \Tot(B)^{n_2}$, with $n_1+n_2=n$.  Calculating one composite we get
\begin{align*}
	\left(\Tot(f \otimes g)\circ\mu_{A,B})(a\otimes b)\right)_{j}
		&=\sum_{m\geq 0}(-1)^{mn} (f\otimes g)_m((\mu_{A,B}(a\otimes b))_{j+m}) \\
		&=\sum_{\substack{m_1,m_2\geq 0\\ k_1+k_2=j}}(-1)^{mn+(k_1+m_1)n_2+m_2n_1}  
			f_{m_1}(a_{k_1+m_1})\otimes g_{m_2}(b_{k_2+m_2}),
\end{align*}
where $m=m_1+m_2$.
On the other hand, evaluating the other composite we get
\begin{align*}
	(\mu_{A',B'}\circ\Tot(f)\otimes \Tot(g))(a\otimes b))_{j}
		&=\sum_{k_1+k_2=j}(-1)^{k_1n_2} \Tot(f)(a)_{k_1{}}\otimes\Tot(g)(b)_{k_2 {}}\\
		&=\sum_{\substack{m_1,m_2\geq 0\\ k_1+k_2=j}}(-1)^{k_1n_2+m_1n_1+m_2n_2}  
			f_{m_1}(a_{k_1+m_1})\otimes g_{m_2}(b_{k_2+m_2}),
\end{align*}
showing that the equality holds.

The coherence axioms are left to the reader.  In the bounded case $\Tot$ is strong symmetric monoidal since $\otimes$ distributes over $\oplus$, therefore the natural transformation $\mu$ is a natural isomorphism.
\end{proof}

We summarize the categories we study and their relations in the following commutative diagram. 
\medskip

 \begin{center}
 \begin{tikzpicture}[scale=0.03]
   \node (a) {$\cpx$};
   \node[right=2cm of a] (b) {$\vbic$};
   \node[above=2cm of b] (x) {$\bimod$};
   \node[right= 2cm of b] (c) {$\tc$};
   \node[above=2cm of c] (y) {$\bimodinf$};
   \node[right= 2cm of c] (d) {$\bimodinf$};
   \node[below=2cm of c] (e) {$\fc$};
   \node[left=2cm of e] (g) {$\sfc$};
   \node[below=1.97cm of d] (f) {$\fm$};
    \node[right=2cm of f] (h) {$\sfm$};
   \draw[->]
     (a) edge [right hook->] node[above] {$*$}  (b)
     (b) edge [right hook->]  (c)
     (c) edge node[above]{$U$} (d)
     (c) edge [right hook->] node [left]{$\Tot$} node [right]{$\ast$} (e)
     (d) edge [left hook->] node [right]{$\Tot$} node [left]{$\ast$} (f)
     (e) edge node [below] {$U$}(f)
     (x) edge [right hook->] node[right] {$*$} (b)
     (y) edge [right hook->]  node[right] {$*$} (c)
     (y) edge node[right] {$1_{\bimodinf}$} (d)
     (x) edge [right hook->]  (y)     
     (g) edge [right hook->] node[above] {$*$}  (e)
     (h) edge [left hook->] node[above] {$*$}  (f);
   \draw[dashed,->]
     (c) edge node[above] {$\cong$}  (g)
     (d) edge node[above] {$\cong$}  (h);
 \end{tikzpicture}
\end{center}
All hooked arrows are embeddings; arrows with a $\ast$ are full embeddings.  We embed bigraded modules in vertical bicomplexes and twisted chain complexes by assigning them trivial differentials.  The forgetful functors $U$ forget the differential structure.  All functors are strong symmetric monoidal, except for $\Tot$.  This is strong symmetric monoidal when restricted to the full subcategories of bounded objects; it is only lax symmetric monoidal otherwise.

\subsection{Spectral sequence associated to a twisted complex}\label{section:SS_of_tc}
Every twisted complex $(A,d_m)$ has an associated spectral sequence
\[E_r^{*,*}(A,d_m):=E_r^{*,*}(\Tot(A,d_m)),\]
which is functorial for morphisms of twisted complexes.
Denote by $\delta_r$ the differential of the $r$-th term.
We choose the bigrading in such a way that for all $r\geq 0$, the pair $(E_r(A,d_m),\delta_r)$ is an $r$-bigraded complex,
so we have a functor
$E_r:\tc\lra \rtc.$
With this choice we have
$E_0^{p,q}(A,d_m)=A_p^{q}$ and $\delta_0=d_0$. For $r\geq 1$, we have 
$E_{r}^{p,q}(A,d_m)= H^*(E_{r-1}^{p,q}(A,d_m),\delta_{r-1})$
and the map $\delta_{r}$ depends on the maps $d_m$ for $m\leq r$. 

The map $\delta_r$ is induced by $d_r$ only on those classes
that have a representative $a\in A_i^j$ for which $d_k(a)=0$ for all $k<r$.
In particular, for $r=1$ we have
\[E_1^{p,q}(A,d_m)= H^{q}(A_p^{*},d_0)={\frac{A_p^{q}\cap \Ker(d_0)}{d_0(A_p^{q-1})}} \text{ and }\delta_1=H_{d_0}(d_1).\]

The morphism
of spectral sequences 
\[E_r(f):=E_r(\Tot(f)):E_r^{*,*}(A,d_m^A)\to E_r^{*,*}(B,d_m^B)\]
associated with a morphism of twisted complexes $f:(A,d_m^A)\to (B,d_m^B)$
is given by
$E_0(f)=f_0$ and $E_{r}(f)= H(E_{r-1}(f),\delta_{r-1})$ for $r\geq 1$.
In particular, for $r=1$ we have $E_1(f)= H_{d_0}(f_0)$.
We refer to \cite{Boardman} and \cite{Hurtubise} for further properties
of the spectral sequence associated to a twisted complex.

For the rest of this section, let $r\geq 0$ be an integer. 
We shall consider the following notion of weak equivalence in the category of twisted complexes.

\begin{defi}A morphism of twisted complexes $f:A\to B$ is called an \textit{$E_r$-quasi-isomorphism} if
the morphism $E_r^{*,*}(f):E_r^{*,*}(A)\to E_r^{*,*}(B)$ at the $r$-stage of the associated spectral 
sequence is a quasi-isomorphism of $r$-bigraded complexes (that is, $E_{r+1}^{*,*}(f)$ is an isomorphism).
\end{defi}

Denote by $\Ee_r$ the class of $E_r$-quasi-isomorphisms of $\tc$. This class is closed under composition and 
contains all isomorphisms of $\tc$. Denote by
\[\Ho_r(\tc):=\tc[\Ee_r^{-1}]\] the \textit{$r$-homotopy category of twisted complexes}
defined by inverting $E_r$-quasi-isomorphisms.
Since $\Ee_{r}\subset \Ee_{r+1}$ for all $r\geq 0$, we have a chain of functors
\[\Ho_0(\tc)\lra \Ho_1(\tc)\lra \cdots \lra \Ho_r(\tc)\lra \cdots\]

\begin{rmk}
The class $\Ee_1$ of $E_1$-quasi-isomorphisms corresponds to the class of
\textit{weak multiequivalences} defined by Huebschmann \cite{Huebschmann}
and the class of \textit{$E_2$-equivalences} considered by Sagave \cite{Sag10}.
\end{rmk}

\subsection{$r$-homotopies and $r$-homotopy equivalences}\label{rhomot_twisted}
We next define a collection of functorial paths indexed by an integer $r\geq 0$ on the category of twisted complexes,
giving rise to the corresponding notions of $r$-homotopy. 

\begin{defi}
\label{def:rpathtc}
The \textit{$r$-path of a twisted complex $(A,d_m)$} is the twisted complex given by
\[P_r(A)_{i}^j:=A_i^j\oplus A_{i+r}^{j+r-1}\oplus A_i^j,\]
with the maps $D_m:P_r(A)\to P_r(A)$ of bidegree $(-m,-m+1)$ given by
\[D_r:=\left(
\begin{matrix}
 d_r&0&0\\
 -1&-d_r&1\\
 0&0&d_r
\end{matrix}
\right)\text{ and }
D_m:=\left(
\begin{matrix}
 d_m&0&0\\
 0&(-1)^{m+r+1}d_m&0\\
 0&0&d_m
\end{matrix}
\right)\text{ for }m\neq r.\]
\end{defi}

For all $m\geq 0$ we have
$\sum_{i+j=m}(-1)^iD_iD_j=0$.
Hence $(P_r(A),D_m)$ is indeed a twisted complex.

We have strict morphisms of twisted complexes
\[\xymatrix{A\ar[r]^-{\iota_A}&P_r(A) \ar@<1ex>[r]^{\partial^+_A} \ar@<-1ex>[r]_{\partial^-_A}&A}\,\,\,;\,\,\, \partial^\pm_A\circ \iota_A=1_A,\]
given by $\partial^-_A(x,y,z)=x$, $\partial^+_A(x,y,z)=z$ and $\iota_A(x)=(x,0,x)$.
We will denote by $\partial_A^0:P_r(A)\to A$ the map of bidegree $(r,r-1)$ given by $(x,y,z)\mapsto y$. 
We will often omit the subscripts of these maps when there is no danger of confusion.
These maps make the $r$-path of a twisted complex into a path object in the standard
sense of homotopical algebra (see Lemma $\ref{iotahtp}$ below).

\begin{defi}
The \textit{$r$-path of a morphism} $f:(A,d_m^A)\to (B,d_m^B)$ of twisted complexes
is the morphism of twisted complexes $P_r(f):(P_r(A),D_m^A)\to (P_r(B),D_m^B)$ given by
\[P_r(f)_m:=(f_m, (-1)^{m}f_m,f_m).\]
\end{defi}
The above definitions give rise to a functorial path $P_r:\tc\to \tc$ in the category of twisted complexes.
This gives a natural notion of homotopy.

\begin{defi}
\label{def:rhtwcx}
Let $f,g:A\to B$ be two morphisms of twisted complexes.
An \textit{$r$-homotopy from $f$ to $g$} is given by a morphism 
of twisted complexes $h:A\to P_r(B)$ such that $\partial^-_B\circ h=f$ and $\partial^+_B\circ h=g$.
We use the notation $h:f\simr{r} g$.
\end{defi}

\begin{rmk}\label{comparison_paths_twisted}
Let $\Lambda_r$ be the $r$-bigraded complex generated by $e_-$, $e_+$ in bidegree $(0,0)$ and $u$ in bidegree
$(-r,1-r)$, with the differential $\delta_r(e_-)=-u$, $\delta_r(e_+)=u$ and $\delta_i(e_\pm)=0$ for all $i\not=r$.
Then the assignment
  $(x,y,z)\mapsto e_-\otimes x+u\otimes y+e_+\otimes z$ defines a strict isomorphism of twisted complexes
from the $r$-path $(P_r(A),D_m)$ of a twisted complex $(A,d_m)$ to the twisted complex $(\Lambda_r\otimes A,\partial_m)$
where $\partial_m=\delta_m\otimes 1+1\otimes d_m$. 
\end{rmk}

\begin{prop}\label{equivalent_notion_htp}Let $f,g:(A,d_m^A)\to (B,d_m^B)$ be two morphisms of twisted complexes.
Giving an $r$-homotopy $h:f\simr{r} g$ is equivalent to giving a collection of morphisms $\widehat h_m:A\to B$ of
bidegree $(-m+r,-m+r-1)$ such that for all $m\geq 0$,
\begin{equation}
\sum_{i+j=m} (-1)^{i+r}d_i^B\widehat h_j+(-1)^i\widehat h_id_j^A=
\left\{ 
\begin{array}{ll}
0&\text{ if }m<r,\\
g_{m-r}-f_{m-r}&\text{ if }m\geq r.
\end{array}\right.\tag{$H_{m1}$}
\end{equation}
\end{prop}

\begin{proof}
Let $h:f\simr{r} g$. For every $m\geq 0$ we may write $h_m(x)=(f_m(x),\widehat h_m(x),g_m(x))$,
where $\widehat h_m=\partial^0_Bh_m$.
It is a matter of verification to see that
the family $\{\widehat h_m\}_{m\geq 0}$ satisfies $(H_{m1})$ for all $m\geq 0$.
Conversely, one may check that given a family $\{\widehat h_m\}_{m\geq 0}$ satisfying $(H_{m1})$, then the family
$h_m(x):=(f_m(x),\widehat h_m(x),g_m(x))$ satisfies
\[\sum_{i+j=m}(-1)^ih_id_j^A=\sum_{i+j=m}D_i^Bh_j.\qedhere\] 
\end{proof}

\begin{rmk}For $r=1$ we recover the notion of homotopy 
between morphisms of twisted complexes first introduced by Meyer \cite{Meyer}, also considered by Saneblidze \cite{Saneblidze}
and Huebschmann \cite{Huebschmann}.
Up to signs and forgetting bigradings, our notion of $r$-homotopy is also related to the notion of 
$(r)$-homotopy between morphisms of $D_\infty^{(r)}$-modules
introduced by Lapin in \cite{Lapin2001}.
\end{rmk}

\begin{lem}\label{homotopies_tot}
Let $f,g:(A,d_m^A)\to (B,d_m^B)$ be morphisms of twisted complexes. Giving an $r$-homotopy $h:f\simr{r}g$ is equivalent to giving a
homotopy of order $r$,
$\widehat H:\Tot(A)^*\to \Tot(B)^{*-1}$, from $\Tot(f)$ to $\Tot(g)$, that is,
$\widehat H(F_p\Tot(A))\subset F_{p+r}(\Tot(B))$, where $F$ is the column filtration.
\end{lem}
\begin{proof}
Given an $r$-homotopy $h:A\to P_r(B)$ from $f$ to $g$, we obtain a morphism of filtered complexes
$\Tot(h):\Tot(A)\lra \Tot(P_r(B))$. Since
$F_p(\Tot(P_r(B)))^n =F_p\Tot(B)^n\oplus F_{p+r}\Tot(B)^{n-1}\oplus F_p\Tot(B)^n$,
we may write
$\Tot(h)(a)=(\Tot(f)(a),\widehat{H}(a),\Tot(g)(a))$, where $\widehat H:\Tot(A)^*\to \Tot(B)^{*-1}$ satisfies the desired conditions.

Conversely, given $\widehat H:\Tot(A)^*\to \Tot(B)^{*-1}$
such that $d\widehat H+\widehat H d=\Tot(g)-\Tot(f)$
and $\widehat{H}(F_pA)\subset F_{p+r}B$
we define a 
morphism of filtered complexes $H:\Tot(A)\to \Tot(P_r(B))$
by letting $H(a):=(\Tot(f)(a),\widehat{H}(a),\Tot(g)(a))$.
By Theorem $\ref{equiv_twisted_split}$, there is a morphism $h:A\to P_r(B)$ of twisted complexes such that $\Tot(h)=H$. By construction,
$h$ is an $r$-homotopy from $f$ to $g$.
\end{proof}

\begin{prop}\label{equivrelTC}
The notion of $r$-homotopy defines an equivalence relation on the set of morphisms between two given twisted complexes,
which is compatible with the composition.
\end{prop}
\begin{proof}
The homotopy relation defined by a functorial path is reflexive and compatible with the composition 
(see for example~\cite[Lemma I.2.3]{KampsPorter}).
Symmetry is clear.
We prove transitivity.
Let $h:f\simr{r}f'$ and $h':f'\simr{r}f''$.
Using the equivalent notion of $r$-homotopy of Proposition $\ref{equivalent_notion_htp}$
we get an $r$-homotopy $h''$ by letting $\widehat h''=\widehat h+\widehat h'$.
\end{proof}

\begin{defi}
A morphism of twisted complexes $f:A\to B$ is called an \textit{$r$-homotopy equivalence} if 
there exists a morphism $g:B\to A$ satisfying $f\circ g\simr{r} 1_B$ and $g\circ f\simr{r} 1_A$.
\end{defi}
Denote by $\Ss_r$ the class of $r$-homotopy equivalences of $\tc$. 
This class is closed under composition and contains all isomorphisms.
 
\begin{prop}
For all $r\geq 0$, we have $\Ss_r\subset \Ss_{r+1}$.
\end{prop}
\begin{proof}
Using the equivalent notion of homotopy of Proposition $\ref{equivalent_notion_htp}$, it is straightforward to see that
given an $r$-homotopy $h$ from $f$ to $g$, we 
obtain an $(r+1)$-homotopy $h'$ from $f$ to $g$ by letting $\widehat h'_0=0$ and $\widehat h'_m=\widehat h_{m-1}$ for $m>0$.
\end{proof}

\begin{prop}
For all $r\geq 0$, we have $\Ss_r\subset \Ee_r$.
\end{prop}
\begin{proof}
By Lemma $\ref{homotopies_tot}$, an $r$-homotopy from $f$ to $g$ in $\tc$ gives a chain
homotopy $H$ from $\Tot(f)$ to $\Tot(g)$ satisfying $H(F_p)\subset F_{p+r}$.
By~\cite[Proposition XV.3.1]{CaEil}, we have $E_{r+1}(f)=E_{r+1}(g)$.
\end{proof}

\begin{lem}\label{iotahtp}
Let $(A,d_m)$ be a twisted complex. The strict morphism $\iota_A:(A,d_m)\lra (P_r(A),D_m)$
given by $\iota_A(x)=(x,0,x)$ is an $r$-homotopy equivalence.
\end{lem}
\begin{proof}
Since $\partial_A^-\iota_A=1_{A}$, it suffices to define an $r$-homotopy from $1_{P_r(A)}$ to
$\iota_A\partial_A^-$.
Consider the morphisms $\widehat h_m:P_r(A)\to P_r(A)$ of bidegree $(-m+r,-m+r-1)$ defined by
$\widehat h_0(x,y,z)=(0,0,y)$ and $\widehat h_i=0$ for all $i>0$. 
It only remains to verify condition $(H_{r1})$ of Proposition $\ref{equivalent_notion_htp}$. We have
\begin{align*}
(D_r\widehat h_0+\widehat h_0D_r)(x,y,z)&=D_r(0,0,y)+\widehat h_0(d_rx,-x-d_ry+z,d_rz)\\
&=(0,y,d_ry)+(0,0,-x-d_ry+z)=(0,y,-x+z)\\
&=
(1_{P_r(A)}-\iota_A\partial_A^-)(x,y,z).\qedhere
\end{align*}
\end{proof}

\begin{teo}\label{QuotientcatTC}
The localized category $\tc[\Ss_r^{-1}]$
is canonically isomorphic to the quotient category  $\pi_r(\tc):=\tc/\simr{r}$.
\end{teo}
\begin{proof}
Denote by $\gamma_r:\tc\to \tc[\Ss_r^{-1}]$ the localization functor.
It suffices to show that 
if $h:f\simr{r} g$ then $\gamma_r(f)=\gamma_r(g)$
(see~\cite[ Proposition 1.3.3]{GNPR}).
Consider the following diagram of morphisms of twisted complexes.
\[
\xymatrix{
&B\\
A\ar[ur]^f\ar[r]^h\ar[dr]_g&P_r(B)\ar[u]_{\partial^-_B}\ar[d]^{\partial^+_B}&B\ar[l]_{\iota_B}\ar@{=}[ul]\ar@{=}[dl]\\
&B&
}\]
By Lemma $\ref{iotahtp}$ the morphism $\iota_B$ is an $r$-homotopy equivalence. Hence
the above diagram is a hammock between the $\Ss_r$-zigzags $f$ and $g$ in the sense of \cite{DHKS}. 
This gives $f=g$ in $\tc[\Ss_r^{-1}]$.
\end{proof}

\subsection{The $r$-translation and the $r$-cone}
The $r$-path construction is related to a translation functor depending on $r$ 
(see \cite{CG2} and \cite{CG1} for similar constructions in the categories of filtered complexes and filtered commutative dgas respectively).
Furthermore, the cone obtained via this translation allows one to detect $E_r$-quasi-isomorphisms, as we shall see next.

\begin{defi}The \textit{$r$-translation} of a twisted complex $(A,d_m)$ is the twisted complex $(T_r(A),T_r(d_m))$ given by
$T_r(A)_i^j:=A_{i-r}^{j-r+1}$ and $T_r(d_m):=(-1)^{m+r+1}d_m$.
\end{defi}

\begin{defi}The \textit{$r$-cone} of a morphism $f:(A,d_m^A)\to (B,d_m^B)$ of twisted complexes is the
twisted complex $(C_r(f),D_m)$ given by
$C_r(f)_i^j:=A_{i-r}^{j-r+1}\oplus B_i^j$
with the maps $D_m:C_r(f)\to C_r(f)$ of bidegree $(-m,-m+1)$ given by
\[D_m(a,b):=((-1)^{m+r+1}d_m(a),(-1)^{m+r+1}f_{m-r}(a)+d_m(b)),\]
where we adopt the convention that $f_{<0}=0$.
\end{defi}

We have strict morphisms $(B,d_m^B)\lra (C_r(f),D_m)$
and
$(C_r(f),D_m)\lra (T_r(A),T_r(d_m^A))$
given by $b\mapsto (0,b)$ and $(a,b)\mapsto a$ respectively. These fit into a short exact sequence
\[0\lra (B,d_m^B)\lra (C_r(f),D_m)\lra (T_r(A), T_r(d_m^A))\lra 0.\]

The following is a matter of verification.
\begin{lem}\label{conehomotopy}
Let $w:A\to B$ be a morphism of twisted complexes and $X$ a twisted complex. Giving a morphism
$\tau:C_r(w)\to X$ of twisted complexes is equivalent to giving
a pair $(f,h)$ where $f:B\to X$ is a morphism of twisted complexes and $h:0\simr{r} fw$
is an $r$-homotopy from 0 to $fw$.
\end{lem}
\begin{proof}
Let $\tau:C_r(w)\to X$ be a morphism of twisted complexes. Define
a morphism of twisted complexes
 $f:B\to X$ by letting $f_m(b):=\tau_m(0,b)$.
Let $\widehat h_m:A\to X$ be defined by
$\widehat h_m(a):=(-1)^m\tau_m(a,0)$. 
By Proposition $\ref{equivalent_notion_htp}$ this gives an $r$-homotopy
$h$ from $0$ to $fw$.
Conversely, given $(f,h)$, we let $\tau_m(a,b):=(-1)^m\widehat{h}_m(a)+f_m(b)$.
\end{proof}

\begin{prop}\label{detect}
 Let $r\geq 0$ and let $f:(A,d_m^A)\to (B, d_m^B)$ be a morphism of twisted complexes. We have a long exact sequence
\[\cdots \lra E^{p,q}_{r+1}(A) \lra E^{p,q}_{r+1}(B) \lra E^{p,q}_{r+1}(C_r(f)) \lra E^{p-r,q-r+1}_{r+1}(A) \lra\cdots.\]
In particular, the morphism $f$ is an $E_r$-quasi-isomorphism if and only if
 the $r$-cone of $f$ is $\Ee_r$-acyclic, that is, $E^{*,*}_{r+1}(C_r(f))=0$.
\end{prop}
\begin{proof}
For every $p$ we have a short exact sequence of complexes
\[0\to (E_0^{p,*}(B), d_0^B)\to E_0^{p,*}(C_0(f),D_0)\to (E_0^{p,*}(A),d_0^A)[1]\to 0,\]
which induces a long exact sequence in cohomology. This proves the result for $r=0$.
Assume that $r>0$. For $m<r$ we have $D_m(a,b)=((-1)^{m+r+1}d_m^A(a),d_m^B(b))$,
so the contribution of $f$ to the differential vanishes.
This gives a direct sum decomposition 
$E^{p,q}_r(C_r(f))\cong E_r^{p-r,q-r+1}(A)\oplus E_r^{p,q}(B)$ 
inducing a long exact sequence in cohomology.
\end{proof}

\subsection{Operadic approach}

In this section we recall how to view twisted complexes as algebras over the operad $\Dd_\infty$. We then
study $r$-homotopy from this point of view.

Let $\Dd$ be the operad of dual numbers in vertical bicomplexes.
Here $\Dd=R[\epsilon]/(\epsilon^2)$,  where the bidegree of $\epsilon$ is $(-1,0)$.
This has trivial vertical differential and contains only arity one
operations, so it can be thought of as simply a bigraded $R$-algebra.

The category of twisted complexes $\tc$ is isomorphic to the category of $\Dd_\infty$-algebras in vertical bicomplexes
(see~\cite[Section 3.1]{LRW} or~\cite[10.3.17]{LV} for the singly-graded analogue). Using the so-called
Rosetta Stone~\cite{LV}, this means that twisted complexes can be studied via structure on conilpotent cofree coalgebras
over the Koszul dual cooperad  $\Dd^{\antishriek}$. 

We first recall some details from~\cite[3.4]{ALRWZ} about $\Dd^{\antishriek}$-coalgebras.
We then make explicit how twisted complexes and their morphisms may be encoded via conilpotent cofree coalgebras, before
putting $r$-homotopies into this context.

The Koszul dual $\Dd^{\antishriek}$ of $\Dd$ is again concentrated in
arity one and can be thought of as just an $R$-coalgebra. We have
 $\Dd^{\antishriek}=R[x]$, where $x=S^{-1}\epsilon$, $x$ has bidegree
$(-1,-1)$ and the comultiplication is determined by $\Delta (x^n)=\sum_{i+j=n} x^i\otimes x^j$.

A $\Dd^{\antishriek}$-coalgebra is a (left)-comodule $C$ over this coalgebra and this 
turns out to just be a pair $(C,f)$, where $C$ is an $R$-module
and $f$ is a linear map $f:C\to C$ of bidegree $(1,1)$. 
(Given a coaction $\rho:C\to \Dd^{\antishriek}\otimes C=R[x]\otimes C$, 
write $f_i$ for the projection onto $R x^i \otimes C$; then coassociativity gives
$f_{m+n}=f_mf_n$, so the coaction is determined by $f_1$.) 
A coderivation is a linear map $d:C\to C$ of bidegree $(s,t)$
such that $df=(-1)^{\langle d, f \rangle}fd$, that is  $df=(-1)^{s+t}fd$. In particular, if $d$ has bidegree
$(0,1)$ then it anti-commutes with $f$.

\begin{rmk}
\label{rmk:dx}
As an example, the conilpotent cofree  $\Dd^{\antishriek}$-coalgebra generated by a bigraded module $A$ is given by $\Dd^{\antishriek}(A)=R[x]\otimes A$
with linear map $d_x^A:R[x]\otimes A\rightarrow R[x]\otimes A$ determined by $d_x^A(x^i\otimes a)=x^{i-1}\otimes a$. A map of $\Dd^{\antishriek}$-coalgebras $h:(C,f_C)\rightarrow (D,f_D)$ of bidegree $(u,v)$ is a map of bigraded modules satisfying $f_Dh=(-1)^{u+v}hf_C$.
\end{rmk}

It will be useful to introduce the following basic object.

\begin{defi}
\label{composition_defi}
Let $A, B, C$ be bigraded modules. We denote by $\wbimod(A,B)$ the bigraded module given by 
\[\wbimod(A,B)_u^v:=\prod_{j\geq 0}[A,B]_{u-j}^{v-j}\]
where $[A,B]$ is the inner hom-object of bigraded modules. 
More precisely, $g\in\wbimod(A,B)_u^v$ is given by $g:=(g_0, g_1, g_2, \dots)$, where $g_j:A\to B$ is a map of bigraded modules of bidegree $(u-j, v-j)$.  Moreover, we define a \emph{composition morphism}
\[c:\wbimod(B,C)\otimes \wbimod(A,B)\to \wbimod(A,C)\]
by  
	\[ c(f,g)_m:=\sum_{i+j=m} (-1)^{i|g|}f_ig_j.
	\]
\end{defi}

In the next section, we will develop this much further, in particular defining the enriched category
$\wbimod$.

We explain how structure in the world of $\Dd^{\antishriek}$-coalgebras corresponds to 
the explicit twisted complex notions. We write $U$ for the forgetful functor from $\Dd^{\antishriek}$-coalgebras
to bigraded modules, left adjoint to the cofree coalgebra functor $\Dd^{\antishriek}$.

\begin{prop}\label{prop:corr}Let $A$, $B$ and $C$ be bigraded modules.
\begin{enumerate}
\item We have bijections of underlying sets

 \begin{tabular}{ccccc}
$\Hom_{\Dd^{\antishriek}-{\mathrm{coalg}}}( \Dd^{\antishriek}(A),  \Dd^{\antishriek}(B))_u^v$
&$\longleftrightarrow$ &$[U\Dd^{\antishriek}(A),B]_u^v$
&$\longleftrightarrow$ &$\wbimod(A,B)_u^v$\\
$\widetilde{F}$&$\longleftrightarrow$ &$F$&&\\
&&$F$&$\longleftrightarrow$ &$f=(f_n)$ 
\end{tabular}

Here a map of bidegree $(u,v)$  of bigraded modules $F:U\Dd^{\antishriek}(A)\rightarrow B$ uniquely lifts as a morphism of 
$\Dd^{\antishriek}$-coalgebras $\widetilde  F:\Dd^{\antishriek}(A)\rightarrow \Dd^{\antishriek}(B)$, of bidegree $(u,v)$
with formula
\[\widetilde  F(x^n\otimes a)=\sum_{i\geq 0} (-1)^{i(u+v)} x^i\otimes F(x^{n-i}\otimes a).\]
Furthermore if $B=A$ then $\widetilde  F$ is also a coderivation of $\Dd^{\antishriek}$-coalgebras.
We associate to such a map $F$ the collection of maps $f_n:A\rightarrow B$ given by $f_n(a)=F(x^n\otimes a)$.
\item
If $\widetilde{d^A} :\Dd^{\antishriek}(A)\rightarrow \Dd^{\antishriek}(A) $ is a square-zero 
coderivation of $\Dd^{\antishriek}$-coalgebras of bidegree $(0,1)$, then the corresponding collection
of maps $d_n^A$ makes $A$ into a twisted complex.
\item
If $\widetilde{d^A}$ and 
$\widetilde{d^B}$ are square-zero coderivations of bidegree $(0,1)$ on 
$\Dd^{\antishriek}(A)$ and $\Dd^{\antishriek}(B)$ respectively,
and $\widetilde{F}:\Dd^{\antishriek}(A)\rightarrow \Dd^{\antishriek}(B)$ is a morphism
of  $\Dd^{\antishriek}$-coalgebras of bidegree $(0,0)$
with $\widetilde{d^B}\widetilde{F}= \widetilde{F} \widetilde{d^A}$ then $f=(f_n)$
is a morphism of twisted complexes from $(A, (d_n^A))$ to $(B, (d_n^B))$.

\item
Composition of coalgebra morphisms  of bidegree $(0,0)$, $\widetilde  G:\Dd^{\antishriek}(A)\rightarrow \Dd^{\antishriek}(B)$ and $\widetilde  F:\Dd^{\antishriek}(B)\rightarrow \Dd^{\antishriek}(C)$,
 corresponds to composition of morphisms of
twisted complexes. 
\end{enumerate}
\end{prop}

\begin{proof} 
\begin{enumerate}
\item As above, let $\Delta:\Dd^{\antishriek}\rightarrow \Dd^{\antishriek}\circ \Dd^{\antishriek}$ 
be the co-composition in the cooperad $\Dd^{\antishriek}$ (which sends $x^n$ to $\sum x^i\otimes x^{n-i}$). 
Then one  checks easily that
$\Delta\widetilde{F}=\Dd^{\antishriek}(\widetilde {F})\Delta$, for $\widetilde{F}$ to be a morphism or a coderivation. 

One obtains the map $f_m:A\to B$ by setting $f_m(a)=\pi_{B}\widetilde{F}(x^m\otimes a)$ 
and similarly for  any map from $\Dd^{\antishriek}(A)\rightarrow \Dd^{\antishriek}(B)$.

\item Considering $(\widetilde{d^A})^2(x^m\otimes a)=0$, we read off $\sum_{i+j=m} (-1)^i d_i^A d_j^A=0$.

\item Considering $\widetilde{d^B}\widetilde{F}(x^m\otimes a)=\widetilde{F}\widetilde{d^A}(x^m\otimes a)$, we read off 
	\[
	\sum_{i+j=m}  d_i^Bf_j=\sum_{i+j=m} (-1)^{i} f_id_j^A.
	\]

\item One checks the statement about composition similarly.\qedhere
\end{enumerate}
\end{proof}

\begin{rmk}
The fact that the condition is the same to be a morphism of coalgebras as to be a coderivation arises 
because the cooperad $\Dd^{\antishriek}$ has only unary operations.
\end{rmk}

In order to formulate $r$-homotopy in this context we use a certain kind of shift operation on morphisms. 

\begin{defi}
\label{def:shift}
Let $\bbS:\wbimod(A,B)_u^v\to \wbimod(A,B)_{u+1}^{v+1}$ be the following map of $R$-modules.
For $f=(f_0, f_1, f_2, \dots)\in\wbimod(A,B)_u^v$,  we define $\bbS f\in \wbimod(A,B)_{u+1}^{v+1}$ 
by  $(\bbS f)_n:=f_{n-1}$, for $n\geq 1$ and $(\bbS f)_0:=0$.
That is, $\bbS(f_0, f_1, f_2, \dots):= (0, f_0, f_1, f_2, \dots)$. 

We 
write $\bbS^r$ for the $r$-th iterate of this operation.
\end{defi}

\begin{prop}
\label{prop: SF}
If $f\in \wbimod(A,B)_u^v$  corresponds to $\widetilde{F}\in \Hom_{\Dd^{\antishriek}-{\mathrm{coalg}}}( \Dd^{\antishriek}(A),  \Dd^{\antishriek}(B))_u^v$ under the bijection of Proposition~\ref{prop:corr},
then $\bbS f$ corresponds to $\widetilde{F}d_x^A$.
\end{prop}

\begin{proof}
Let $\widetilde{G}$ be the map  corresponding to $\bbS f$. Then,
for all $n\geq 0$ and all $a\in A$,
	\begin{align*}
	\widetilde{G}(x^n\otimes a)&=\sum_{i=0}^n (-1)^{i|f|} x^i\otimes (\bbS f)_{n-i}(a)
			=\sum_{i=0}^{n-1} (-1)^{i|f|} x^i\otimes f_{n-1-i}(a)\\
			&=\widetilde{F}(x^{n-1}\otimes a)=\widetilde{F}d_x^A(x^n\otimes a).\qedhere
	\end{align*}
\end{proof}

Now we see what $r$-homotopy looks like in this context.

\begin{teo}
\label{teo:coderh}
Let $A, B\in \tc$, with  $\widetilde{d}^A, \widetilde{d}^B$ the square-zero coderivations of bidegree $(0,1)$ on 
$\Dd^{\antishriek}(A)$ and $\Dd^{\antishriek}(B)$  respectively
encoding the twisted complex structures of $A$ and $B$. 
Let $f,g\in \Hom_{\tc}(A, B)$, with corresponding $\Dd^{\antishriek}$-coalgebra maps
$\widetilde{F},\widetilde{G}: \Dd^{\antishriek}(A)\to \Dd^{\antishriek}(B)$ of bidegree $(0,0)$. Then having an $r$-homotopy $h$ between
$f$ and $g$ is equivalent to having a $\Dd^{\antishriek}$-coalgebra map  $\widetilde{H}: \Dd^{\antishriek}(A)\to \Dd^{\antishriek}(B)$
of bidegree $(r,r-1)$
such that 
	\[
	(-1)^r\widetilde{d}^B\widetilde{H}+\widetilde{H}\widetilde{d}^A=\bbS^r \widetilde{G}-\bbS^r \widetilde{F}.
	\]
\end{teo}

\begin{proof}
Considering $((-1)^r\widetilde{d}^B\widetilde{H}+\widetilde{H}\widetilde{d}^A)(x^n\otimes a)=(\bbS^r \widetilde{G}-\bbS^r \widetilde{F})(x^n\otimes a)$, we read off
	\begin{equation}
	\sum_{i+j=m}(-1)^{r+i} d^B_ih_j(a)+(-1)^i h_i d^A_j(a)=
		\begin{cases} 
			g_{m-r}(a)-f_{m-r}(a), &\text{if $m\geq r$}\\ 
			0, & \text{if $m<r$}
		\end{cases}
	\tag{$H_m$}\end{equation}
which is equivalent to the $r$-homotopy condition, by Proposition~\ref{equivalent_notion_htp}.
\end{proof}

\section{New interpretations of derived $A_\infty$-algebras}
\label{sec:newder}
In this section we reinterpret derived $A_\infty$-algebras both as $A_\infty$-algebras 
in twisted chain complexes and as split filtered $A_\infty$-algebras.  First we recall the basic notions 
regarding derived $A_\infty$-algebras.  Next we endow the categories of twisted complexes and filtered
complexes with a monoidal structure over a base, in the sense of Fresse~\cite{Fresse},  and explicitly 
describe the enrichments that these structures induce.  Then we show that the totalization functor and
its properties extend to this enriched setting and use this to prove our main results.

\subsection{The category of derived $A_\infty$-algebras}
We begin by recalling the basic definitions for derived $A_\infty$-algebras, also known as $dA_\infty$-algebras.

\begin{defi}\label{def:dA_obj}
A (non-unital) \textit{$dA_\infty$-algebra} $(A,m_{ij})$ is a $(\ZZ,\ZZ)$-bigraded $R$-module $A=\{A_i^j\}$ equipped with morphisms 
$\{m_{ij}:A^{\otimes j}\lra A\}_{i\geq 0, j\geq 1}$ of bidegree $(-i,2-i-j)$ such that for all $u\geq 0$ and all $v\geq 1$,
\begin{equation}\label{R:dAs}
\sum\limits_{\substack{ u=i+p, v=j+q-1 \\ j=1+r+t}} 
(-1)^{rq+t+pj}m_{ij}({1}^{\otimes r}\otimes m_{pq}\otimes {1}^{\otimes t})=0. \tag{$A_{uv}$}
\end{equation}

\end{defi}

\begin{defi}\label{def:dA_mor}
 A \textit{morphism} $f:(A,m_{ij}^A)\to (B,m_{ij}^B)$ of $dA_\infty$-algebras is given by a family of morphisms
 $\{f_{ij}: A^{\otimes j}\rightarrow B\}_{i\geq 0, j\geq 1}$ of bidegree $(-i,1-i-j)$ such that for all
 $u\geq 0$ and all $v\geq 1$,
\begin{equation}\label{dmapequation}
\sum\limits_{\substack{ u=i+p, v=j+q-1 \\ j=1+r+t}} (-1)^{rq+t+pj}
f_{ij} (1^{\otimes r} \otimes m^A_{pq} \otimes 1^{\otimes t}) = \\
\sum\limits_{\substack{ u=i+p_1+\cdots+p_j,\\ v=q_1+\cdots+q_j }} (-1)^{\sigma} m^{B}_{ij}(f_{p_1q_1}\otimes\cdots\otimes f_{p_jq_j}),
\tag{$B_{uv}$}\end{equation}
where $\sigma=u+\sum\limits_{t=1}^{j}(p_t+q_t)(j+t)+q_t\sum\limits_{w=t+1}^j (p_w+q_w)$.
\end{defi}

\begin{prop}\label{prop:composition}
Let $g=(g_{ij}):A\to B$ and $f=(f_{ij}):B\to C$ be two morphisms of $dA_\infty$-algebras. Then the composite morphism 
$fg:A\to C$ of
$dA_\infty$-algebras has components
\[
(fg)_{uk}=\sum_{i+p=u} \sum_r\sum_{\substack{p_1+\cdots +p_r=p\\q_1+\cdots +q_r=k}}
							(-1)^{\sigma} f_{ir}\left(
			 g_{p_1q_1}\otimes \cdots \otimes g_{p_rq_r}\right),
\]
where $\sigma=\sum\limits_{t=1}^{r} (p_t+q_t)(r+t)+q_t\sum\limits_{w=t+1}^r (p_w+q_w)$. 

\end{prop}
\begin{proof} This is a direct consequence of Equations (4) and (5) in \cite[Theorem 2.8]{LRW}.
\end{proof}

A morphism $f=\{f_{ij}\}$ is said to be \textit{strict} if $f_{ij}=0$ for all $i>0$ and all $j>1$.
The identity morphism $1_A:A\to A$ is the strict morphism given by $(1_A)_{01}(x)=x$.

\begin{lem}
A morphism $f=\{f_{ij}\}$ of $dA_\infty$-algebras is an isomorphism if and only if $f_{01}$ is an isomorphism of bigraded $R$-modules.
\end{lem}
\begin{proof}
If $fg=1$, then $f_{01}g_{01}=1$, so $f_{01}$ is an isomorphism and $g_{01}=f_{01}^{-1}$. Then
the equation giving the $(uk)$ component of the composite has a ``top term''
$f_{01}g_{uk}$, with all other summands involving components $g_{ij}$ with $i<u$ or $j<k$. Thus we
can successively solve for each $g_{uk}$ if and only if $f_{01}$ is an isomorphism. This gives a right
inverse $g$ for $f$ if and only if $f_{01}$ is an isomorphism. The same argument shows that $g$ also
has a right inverse, and this must be $f$, so $f$ and $g$ are two-sided inverses.
\end{proof}

Denote by $\dAinf{R}$ the category of $dA_\infty$-algebras over $R$.

\begin{example}[$A_\infty$-algebras] 
The category $\Ainf{R}$ of $A_\infty$-algebras is a full subcategory of $\dAinf{R}$.
Indeed, if a $dA_\infty$-algebra $(A,m_{ij})$ is concentrated in horizontal degree $0$, 
that is, $A_i^j=0$ and $m_{ij}=0$ for all $i>0$, then $(A,m_{0j})$ is an $A_\infty$-algebra.
\end{example}

\begin{example}[Underlying twisted complex]
There is a forgetful functor $U:\dAinf{R}\lra \tc$ defined by sending a $dA_\infty$-algebra
$(A,m_{ij})$ to the twisted complex $(A,m_{i1})$
and a morphism $f=\{f_{ij}\}$ of $dA_\infty$-algebras to the morphism of twisted complexes given by 
$U(f)=\{f_{i1}\}$.
\end{example}

We define $E_r$-quasi-isomorphism for $dA_\infty$-algebras via their underlying twisted complexes.

\begin{defi}
Let $r\geq 0$.
A morphism $f=\{f_{ij}\}$ of $dA_\infty$-algebras is said to be an \textit{$E_r$-quasi-isomorphism} if
the corresponding map $U(f):=\{f_{i1}\}$ of twisted complexes is an $E_r$-quasi-isomorphism.
\end{defi}

Denote by $\Ee_r$ the class of $E_r$-quasi-isomorphisms of $\dAinf{R}$ and by
$\Ho_r(\dAinf{R}):=\dAinf{R}[\Ee_r^{-1}]$ the \textit{$r$-homotopy category} defined by inverting $E_r$-quasi-isomorphisms.
Note that $\Ee_r=U^{-1}(\Ee_r^{\tc})$. The forgetful functor induces a functor $U:\Ho_r(\dAinf{R})\lra \Ho_r(\tc)$.

\subsection{Monoidal categories over a base}
In the following sections, all our notions of enriched category theory follow \cite{Borceux} and \cite{Riehl}.
Our new descriptions of $dA_\infty$-algebras will use certain enriched categories coming
from monoidal categories over a base as defined 
in~\cite{Fresse}.  We recall this notion first.

\begin{defi} 
\label{monoidal_over_base}
Let $(\vcat,\otimes,1)$ be a symmetric monoidal category and let $(\Cc,\otimes,1)$ be a monoidal category.  We say that $\Cc$ is a \emph{monoidal category over $\vcat$} if we have an \emph{external tensor product} 
$\ast:\vcat \times \Cc \to \Cc$ such that we have natural isomorphisms:
\begin{itemize}
\item $1\ast X \cong X$ for all $X\in \Cc$,
\item $(C\otimes D)\ast X \cong C\ast (D\ast X)$ for all $C,D\in\vcat$ and $X\in\Cc$,
\item $C\ast(X\otimes Y) \cong (C\ast X)\otimes Y \cong X\otimes (C\ast Y)$ for all $C\in \vcat$ and $X,Y\in \Cc$.
\end{itemize}
\end{defi}

\begin{rmk}
\label{monoidal_over_enrichment}
If we have, in addition, a bifunctor
$\ccat(-,-):\Cc^{\mathrm{op}}\times \Cc \to \vcat$
such that we have natural bijections
\begin{equation}
\label{adjoint_exteranl_formal}
\Hom_{\Cc}(C\ast X, Y)\cong \Hom_{\vcat}(C,\ccat(X,Y))
\end{equation}
(for example, if $\ast$ preserves colimits on the left and certain smallness conditions hold) 
we get a $\vcat$-enriched category $\ccat$ with the same objects as $\Cc$ and with hom-objects given by $\ccat(-,-)$.
The unit morphism
$u_A: 1 \to \ccat(A,A)$
corresponds to the identity map in $\Cc$ under the adjunction and the composition morphism is given by the adjoint of the composite 
\[(\ccat(B,C)\otimes \ccat(A,B)) \ast A 
\stackrel{\cong}{\longrightarrow}
\ccat(B,C)\ast (\ccat(A,B)\ast A) 
\stackrel{id \ast ev_{AB}}{\xrightarrow{\hspace*{1cm}} }
\ccat(B,C)\ast B
\stackrel{ev_{BC}}{\xrightarrow{\hspace*{1cm}} }
C,\]
where $ev_{AB}$ is the adjoint of the identity $\ccat(A,B)\to \ccat(A,B)$.
Note that by construction, the underlying category of $\ccat$ is $\Cc$.  Furthermore, $\ccat$ is a monoidal $\vcat$-enriched category, namely we have an enriched functor
\[\underline{\otimes}:\ccat\times\ccat\to\ccat\]
where $\ccat\times\ccat$ is the enriched category with objects $\mathrm{Ob}(\ccat)\times \mathrm{Ob}(\ccat)$ and hom-objects
\[\ccat\times\ccat((X, Y), (W,Z)):=\ccat(X,W)\otimes \ccat(Y, Z).\]
In particular we get maps in $\vcat$ 
\[\ccat(X,W)\otimes \ccat(Y,Z)\to \ccat(X\otimes Y, W\otimes Z),\]
given by the adjoint of the composite 
\[(\ccat(X,W)\otimes \ccat(Y,Z))\ast (X\otimes Y)
\stackrel{\cong}{\xrightarrow{\hspace*{0.7cm}} }
(\ccat(X,W)\ast X)\otimes (\ccat(Y,Z)\ast Y)
\stackrel{ev_{XW}\otimes ev_{YZ}}{\xrightarrow{\hspace*{1.5cm}} }
W\otimes Z.
\]
\end{rmk}

We will assume the setup above holds throughout the paper.
It is assumed in \cite{Fresse} and holds in a fairly general setting.
In particular, it holds in all the cases we study here.
One of its useful features is that 
constructions on the level of ordinary categories which respect the external monoidal structure extend to the enriched setting.

\begin{defi}
Let $\Cc$ and $\Dd$ be monoidal categories over $\vcat$.  A \emph{lax functor over $\vcat$} consists of a functor $F:\Cc\to \Dd$ together with a natural transformation 
\[\nu_F:-\ast_\Dd F(-)\Rightarrow F(-\ast_\Cc -)\]
which is associative and unital with respect to the monoidal structures over $\vcat$ of $\Cc$ and $\Dd$. 
(See \cite[Proposition 10.1.5]{Riehl} for explicit diagrams stating the coherence axioms.)
If $\nu_F$ is a natural isomorphism we say $F$ is a \emph{functor over $\vcat$} (or \emph{preserves external tensor products}). 

Let $F,G:\Cc\to \Dd$ be lax functors over $\vcat$.  A \emph{natural transformation over $\vcat$} is a natural transformation $\mu:F \Rightarrow G$ such that for any $C\in\vcat$ and for any $X\in\Cc$ we have 
\[\nu_G\circ (1*_{\Dd} \mu_X)=\mu_{C*_{\Cc} X}\circ \nu_F.\]

A \emph{(lax) monoidal functor over $\vcat$} is a triple $(F,\epsilon,\mu)$, where $F:\Cc\to \Dd$ is a lax functor over $\vcat$, $\epsilon:1_{\Dd} \to F(1_{\Cc})$ is a morphism in $\Dd$ and  
\[\mu: F(-)\otimes F(-) \Rightarrow F(-\otimes -)\] 
is a natural transformation over $\vcat$ satisfying the standard unit and associativity conditions.  If $\nu_F$ and $\mu$ are natural isomorphisms then we say that $F$ is \emph{monoidal over $\vcat$}.
\end{defi}

When restricted to the case of functors over $\vcat$, the first part of the following statement is Proposition 1.1.15 in \cite{Fresse}.  The second part is implicit in the same text.  However, both results and their proofs extend to the case of lax functors over $\vcat$ as we describe below.

\begin{prop}
\label{prop:functor_extends_to_enriched}
Let $F, G:\Cc\to \Dd$ be lax functors over $\vcat$.
Then $F$ and $G$ extend to $\vcat$-enriched functors 
\[\underline{F}, \underline{G}:\ccat \to \ddcat\]
where $\ccat$ and $\ddcat$ denote the $\vcat$-enriched categories corresponding to $\Cc$ and $\Dd$ as described in Remark \ref{monoidal_over_enrichment}.
Moreover, any  natural transformation $\mu:F\Rightarrow G$ over $\vcat$ also extends to a $\vcat$-enriched natural transformation 
\[\underline{\mu}:\underline{F}\Rightarrow \underline{G}.\]  
In particular, if $F$ is (lax) monoidal over $\vcat$, then $\underline{F}$ is (lax) monoidal in the enriched sense, where the monoidal structure of $\ccat\times {\ccat}$ is the one described in Remark \ref{monoidal_over_enrichment}.
\end{prop}

\begin{proof}
For any $X,Y\in\Cc$ the functor $F$ extends to an enriched functor $\underline{F}$ where $\underline{F}(X):=F(X)$ 
and  where the morphism on hom-objects
\[\ccat(X,Y)\to \ddcat(F(X),F(Y))\]
is given by the adjoint of the composite
\[\ccat(X,Y)\ast_\Dd F(X)
\stackrel{\nu_F}{\longrightarrow}
F(\ccat(X,Y)\ast_\Cc X)
\stackrel{F(ev_{XY})}{\xrightarrow{\hspace*{1.5cm}}} 
F(Y).\]
The coherence axioms and the fact that $F$ is the underlying functor of $\underline{F}$ follow formally.  See for example the proof of 
Proposition 10.1.5 in \cite{Riehl}.

To show the second statement recall that a $\vcat$-enriched natural transformation $\underline{\mu}:\underline{F}\Rightarrow \underline{G}$ is given by maps
\[\underline{\mu}_X: 1 \to \ddcat(FX, GX),\]  
such that a naturality conditions holds  
(see \cite[Definition 3.5.8]{Riehl}).  
In our setup, we set $\underline{\mu}_X$ to be the adjoint of $\mu_X:FX\to GX$.  
Since $F$ and $G$ are lax over $\vcat$ and $\mu$ is a natural transformation over $\vcat$,  for any $C\in \vcat$ and $X,Y\in \Cc$, Diagram (\ref{diag_nat_sets}) commutes. Here $\eta_F(f):=F(f)\circ \nu_F$ and $\eta_G$ is defined in the same way.  
\begin{equation}\label{diag_nat_sets}
\begin{tikzpicture}[scale=0.01]
  \node (a) {$\Hom_{\Cc}(C\ast X,Y)$};
  \node[right=1.8cm of a] (b) {$\Hom_{\Dd}(C\ast FX,FY)$};
  \node[below= 1.5cm of a] (c) {$\Hom_{\Dd}(C\ast GX,GY)$};
  \node[below= 1.5cm of b] (d) {$\Hom_{\Dd}(C\ast FX,GY)$};
  \draw[->]
   (a) edge  node[above] {$\eta_F$} (b)
   (a) edge  node[left] {$\eta_G$} (c)
   (c) edge node[below] {$-\circ \mu_X$} (d)
   (b) edge  node[right] {$\mu_Y \circ -$} (d);
\end{tikzpicture}
\end{equation}
Then by adjunction Diagram (\ref{diag_nat_enriched}) commutes, showing that $\underline{\mu}$ is a $\vcat$-enriched natural transformation.
\begin{equation}\label{diag_nat_enriched}
\begin{tikzpicture}
  \node (a) {$\Hom_\vcat(C,\ccat(X,Y))$};
  \node[right=3cm of a] (b) {$\Hom_\vcat(C,\ddcat(FX,FY))$};
  \node[below= 1.5cm of a] (c) {$\Hom_\vcat(C,\ddcat(GX,GY))$};
  \node[below= 1.5cm of b] (d) {$\Hom_\vcat(C,\ddcat(FX,GY))$};
  \draw[->]
   (a) edge  node[above] {$\Hom_\vcat (C,F_{XY})$} (b)
   (a) edge  node[left] {$\Hom_\vcat (C,G_{XY})$}(c)
   (c) edge  node[below] {$\Hom_\vcat (C,\mu^*)$}(d)
   (b) edge  node[right] {$\Hom_\vcat (C,\mu_*)$}(d);
   \end{tikzpicture}
\end{equation}

The last statement of the proposition follows from the above together with two formal facts.  Firstly, if $\Cc$ is monoidal over $\vcat$ then so is $\Cc\times\Cc$.  Secondly, if $(F,\epsilon,\mu)$ is a lax monoidal functor over $\vcat$, then $F(-)\otimes F(-)$ and $F(-\otimes -)$ are lax functors over $\vcat$ and $\mu$ is a natural transformation over $\vcat$.
\end{proof}

The monoidal structures of vertical bicomplexes and twisted complexes are compatible.  More precisely, we can use the tensor product in twisted complexes to define monoidal structures over a base by restriction. 

\begin{lem}
The category $\tc$ is a monoidal category over $\vbic$ and the category $\bimodinf$ is a monoidal category over $\bimod$.  In both cases, the external tensor product is given by restricting the tensor product in twisted complexes.  We use the notation $\otimes$ instead of $\ast$ since the external tensor product coincides with the internal tensor product in $\tc$.  
\end{lem} 
\begin{proof}
The axioms of Definition~\ref{monoidal_over_base} hold because $\otimes$ is a symmetric monoidal structure on $\tc$ and the vertical arrows in the following diagram are monoidal embeddings.
\begin{equation*}
 \begin{tikzpicture}
   \node (a) {$\otimes:\tc \times \tc $};
   \node[right=2cm of a] (b) {$\tc$};
   \node[below= 1cm of a] (c) {$\otimes:\vbic \times \tc $};
   \node[below= 1cm of b] (d) {$\tc$};
   \node[below=1cm of c] (e) {$\otimes: \bimod \times \bimodinf$};
    \node[below=1cm of d] (f) {$\bimodinf$};
   \draw[->]
     (a) edge  (b)
     (c) edge (d)
     (e) edge (f)
     (c) edge [right hook->]  (a)
     (d) edge [left hook->]  (b)
     (e) edge [right hook->]  (c)
     (f) edge [left hook->]  (d);
 \end{tikzpicture}
\qedhere
\end{equation*}
\end{proof}

We can also endow filtered complexes with a monoidal structure over $\vbic$. We use the following. 

\begin{defi}
The \emph{totalization with compact support} of a vertical bicomplex $A$ is the filtered complex given by
\[\Tot_c(A)^n:=\bigoplus_{i\in\ZZ} A_i^{n+i}\]
with the column filtration and with differential as for the totalization functor.  Given a morphism of vertical bicomplexes $f:A\to B$ we get a morphism of filtered complexes $\Tot_c(f):\Tot_c(A)\to \Tot_c(B)$ constructed analogously to $\Tot(f)$.
\end{defi}

\begin{rmk}
Note $\Tot_c$ is well-defined since vertical bicomplexes have only one differential and the category $\vbic$ has only strict morphisms.  Moreover, for any $A$ we have a natural map $\Tot_c(A)\to \Tot(A)$ which is the identity if $A$ is bounded.
\end{rmk}

\begin{lem}
The category $\fc$ is monoidal over $\vbic$ with external tensor product given by
\[\begin{array}{cccl}
\ast: &\vbic \times \fc &\to      & \fc\\
      &(A,K)            &\mapsto  & A\ast K := \Tot_c(A)\otimes K.
\end{array}
\]
On morphisms it is given by the assignment $(f,g)\mapsto \Tot_c(f)\otimes g$.  This induces by restriction a monoidal structure on $\fm$ over $\bimod$.
\end{lem}
\begin{proof}
The assignments define a bifunctor since $\Tot_c:\vbic \to \fm$ is a functor.  Furthermore, the axioms of Definition~\ref{monoidal_over_base} hold since $\otimes$ is a symmetric monoidal product in $\fc$ and $\Tot_c$ is strong symmetric monoidal since $\otimes$ distributes over $\oplus$.  Finally, since $\fm\hookrightarrow\fc$ and $\bimod \hookrightarrow \vbic$ are full embeddings, this construction induces by restriction a bifunctor $\ast:\bimod \times \fm \to \fm$ which gives a monoidal structure on $\fm$ over $\bimod$.
\end{proof}

\subsection{Enrichments from monoidal structures over a base}
We now explain how to give enrichments to the categories of bigraded modules, filtered modules, twisted complexes and filtered complexes, using their monoidal structure over a base.
This will give the $\vbic$-enriched categories $\wtc$ and $\wfc$ and $\bimod$-enriched categories $\wbimod$ and $\wfm$. We emphasize that $\wbimod$ is different from its standard enrichment coming from its symmetric monoidal closed structure.

Recall from Definition~\ref{composition_defi} that for bigraded modules $A,B,C$, we have already defined a bigraded module
$\wbimod(A,B)$ and a composition
	\[
	c:\wbimod(B,C)\otimes \wbimod(A,B)\to \wbimod(A,C).
	\]

\begin{lem}
\label{associative_comp}The composition morphism respects the identity and is associative.
\end{lem}
\begin{proof}
For  $f,g,h$,
\[c(c(f,g),h))_m=\sum_{i+j+k=m} (-1)^{i(|g|+|h|)+j|h|} f_ig_jh_k=c(f,c(g,h)).\qedhere\]
\end{proof}

\begin{defi}
\label{def:delta}
Let $(A,d_i^A), (B,d_i^B)$ be twisted  complexes, $f\in\wbimod(A,B)_u^v$ and consider $d^A:=(d_i^A)_i\in\wbimod(A,A)_0^1$ and $d^B:=(d_i^B)_i\in \wbimod(B,B)_0^1$.
We define
\begin{equation}
\label{differential_wtc_def}
\delta (f):= c(d^B, f) -(-1)^{<f,d^A>} c(f, d^A)\in \wbimod(A,B)_u^{v+1},
\end{equation} 
where $<f,d^A>$ is the scalar product for the bidegrees (as in subsection~\ref{subsection:bgmods}) and $c$ is the composition morphism described in Definition \ref{composition_defi}.
More precisely,
\[ (\delta (f))_m :=
	\sum_{i+j=m} (-1)^{i|f|}d_i^Bf_j- (-1)^{v+i} f_id_j^A.\]
	\end{defi}

\begin{lem}
\label{differential_morphism_tc}
The following equations hold
\begin{align*}
c(d^A,d^A)&=0, \\
\delta^2 &=0,
\end{align*}
\begin{equation}
\label{comp_differential}
\delta(c(f, g))= c(\delta(f), g)+(-1)^v c(f, \delta(g)),
\end{equation}
where the bidegree of $f$ is $(u,v)$. 
Furthermore, $f\in \wbimod(A,B)$ is a map of twisted  complexes if and only if $\delta(f)=0$.  In particular, $f$ is a morphism in $\tc$ if and only if the bidegree of $f$ is 
$(0,0)$ and $\delta(f)=0$. 
 Moreover, for $f,g$ morphisms in $\tc$, we have that $c(f,g)=f\circ g$, where the latter denotes composition in $\tc$. 
\end{lem}

\begin{proof}
One has 
\[c(d^A,d^A)_m=\sum_{i+j=m} (-1)^id^A_id^A_j=0,\]
so that
\begin{align*}
\delta^2(f)=&c(d^B, \delta (f)) -(-1)^{<\delta (f),d^A>} c(\delta (f), d^A)\\
=&c(d^B,c(d^B,f))-(-1)^{<f,d^A>} c(d^B,c(f,d^A))+(-1)^{<f,d^A>} c(c(d^B,f),d^A)-c(c(f,d^A),d^A)\\
=&0.
\end{align*}
The last equation follows from the associativity of $c$.
\end{proof}

Since $\delta$ is of bidegree $(0,1)$, Lemma \ref{differential_morphism_tc} allows us to make the following definition.
\begin{defi}
\label{weird_tc_defi}
For $A, B$ twisted  complexes, we define $\wtc(A,B)$ to be the vertical bicomplex $\wtc(A,B):=(\wbimod(A,B),\delta)$.
\end{defi}

\begin{prop}
\label{external_hom_tc}
If $B, C$ are twisted  complexes, then the construction of the vertical bicomplex $\wtc(B,C):=(\wbimod(B,C),\delta)$ extends to a bifunctor 
\[\wtc(-,-):\tc^{\mathrm{op}}\times \tc \to \vbic,\]
where for $f:C\to C'$ in $\tc$ we set
\[
\begin{array}{ccccccccc}
\wtc(B,f):\wtc(B,&C) & \to &\wtc(B,C') & \quad\text{and}\quad\quad &\wtc(f,B):\wtc(C',& B) & \to &\wtc(C,B)\\
                 & g & \mapsto & c(f,g) &      &                  & g & \mapsto & c(g,f).
\end{array}
\]
Moreover, the functor $-\otimes B:\vbic\to \tc$ is left adjoint to the functor $\wtc(B, -):\tc \to \vbic$, i.e.,  for all $A\in \vbic,  B,C \in\tc$ we have natural bijections
\begin{equation}
\label{adjoint_external}
\Hom_{\tc}(A\otimes B, C)\cong \Hom_{\vbic}(A,\wtc(B,C)),
\end{equation}
given by $f\mapsto \tilde f$ where for $a\in A_u^v, \tilde f(a)_m$ is given by
$b\mapsto (-1)^{m|a|} f_m(a\otimes b).$
\end{prop}

\begin{proof}
If $f$ is a map in $\tc$ then it is of bidegree $(0,0)$ and $\delta(f) = 0$ and thus by (\ref{comp_differential}) 
\[\delta(c(f, g))= c(\delta(f), g)+(-1)^vc(f, \delta(g))= c(f, \delta(g)),\]
showing that $\wtc(B,f)$ is a map of vertical bicomplexes.  A similar argument shows that $\wtc(f,B)$ is a map of vertical bicomplexes. 
Finally, the fact that $\wtc(-,-)$ is a bifunctor follows directly from Lemmas \ref{associative_comp} and \ref{differential_morphism_tc}. 
Now, to see the adjointness property we describe a map $\Hom_{\tc}(A\otimes B, C) \to \Hom_{\vbic}(A,\wtc(B,C))$, which sends a map of twisted complexes $f=(f_m):A\otimes B\rightarrow C$ , to the map
\begin{alignat*}{2}
 \tilde f: A & \to && \wtc(B,C)\\
a & \mapsto && \{\tilde f(a)_m:b\mapsto (-1)^{m|a|} f_m(a\otimes b)\}_{m\geq 0}.
\end{alignat*}
It is clear that $\tilde{f}$ is a bidegree $(0,0)$ map of bimodules.  
To show that it is a map of vertical bicomplexes we will show that $\delta\tilde f=\tilde f d^A$. Let $a\in A_u^v$, then $\tilde f(a)$ has bidegree $(u,v)$
so that
\[\delta(\tilde f(a)))_m=\sum_{i+j=m}(-1)^{i|a|}d_i^C (\tilde f(a))_j-(-1)^{v+i} (\tilde f(a))_id_j^B.\]
Applying this to $b\in B_{u'}^{v'}$ one gets
\begin{align*}
(\delta(\tilde f(a)))_m(b)&=\sum_{i+j=m}(-1)^{i|a|+j|a|}d_i^C f_j(a\otimes b)-(-1)^{v+i+i|a|} f_i(a\otimes d_j^B(b))\\
&=\sum_{i+j=m} (-1)^{m|a|+j}f_j(d_i^A(a)\otimes b)+\sum_{i+j=m} (-1)^{m|a|+j+i|a|+v}f_j(a\otimes d_i^B(b))\\
&\qquad-\sum_{i+j=m}(-1)^{v+j+j|a|} f_j(a\otimes d_i^B(b))).
\end{align*}
Since $A$ is a vertical bicomplex we have that $d^A_i(a)=0$ for $i>0$.  Thus,
\[(\delta(\tilde f(a)))_m=(-1)^{m|a|+m}f_m(d^A(a)\otimes b)=\tilde f(d^A(a))_m(b).\]
The inverse map is constructed in a similar fashion. 
\end{proof}

\begin{defi}
We define $\wbimod(-,-)$ to be the restriction of the bifunctor $\wtc(-,-)$ to the full subcategories $\bimod^{\infty,\mathrm{op}}\times \bimodinf$.
\end{defi}

\begin{prop}\label{external_hom_bimod}
The image of $\wbimod(-,-)$ factors through $\bimod$.  Therefore it defines a bifunctor 
\[\wbimod(-,-):\bimod^{\infty,\mathrm{op}}\times \bimod^\infty \to \bimod\]
and $-\otimes B:\bimod\to \bimodinf$ is left adjoint to the functor $\wbimod(B, -):\bimodinf \to \bimod$, i.e.,  for all $A\in \bimod,  B,C \in\bimodinf$ we have natural bijections
\[\Hom_{\bimodinf}(A\otimes B, C)\cong \Hom_{\bimod}(A,\wbimod(B,C)).\]
\end{prop}

\begin{proof}
This follows directly from the facts that if $B,C\in\bimodinf$ then $\wtc(B,C)$ has trivial 
differential and that the functors $\bimod\hookrightarrow \vbic$ and $\bimodinf \hookrightarrow \tc$ are full embeddings.
\end{proof}

This construction gives us our different enrichments of twisted  complexes and bigraded modules, which we describe now.

\begin{defi}\label{def:wtc}
The $\vbic$-enriched category of twisted  complexes $\wtc$ is the enriched category given by the following data.
\begin{enumerate}
\item The objects of $\wtc$ are twisted  complexes.
\item For $A,B$ twisted  complexes the hom-object is the vertical bicomplex $\wtc(A,B)$ 
\item The composition morphism $c:\wtc(B,C)\otimes \wtc(A,B)\to \wtc(A,C)$ is given by Definition \ref{composition_defi}.
\item The unit morphism $R \to \wtc(A,A)$ is given by the morphism of vertical bicomplexes sending $1\in R$ to $1_A:A\to A$, the strict morphism of twisted complexes given by the identity of $A$.
\end{enumerate}
\end{defi}

\begin{defi}
\label{weird_bimod}
We denote by $\wbimod$ the $\bimod$-enriched category of bigraded modules given by the following data.
\begin{enumerate}
\item The objects of $\wbimod$ are bigraded modules.
\item For $A,B$ bigraded modules the hom-object is the bigraded module $\wbimod(A,B)$.
\item The composition morphism $c:\wbimod(B,C)\otimes \wbimod(A,B)\to \wbimod(A,C)$ is given by Definition \ref{composition_defi}.
\item The unit morphism $R \to \wbimod(A,A)$ is given by the morphism of bigraded modules that sends $1\in R$ to $1_A:A\to A$, the strict morphism given by the identity of $A$. 
\end{enumerate}
\end{defi}

\begin{lem}\label{L:wtc}
The enriched categories
$\wtc$ and $\wbimod$ are well-defined and their enrichments are the ones induced by the external tensor products $\otimes:\vbic \times \tc \to \tc$ and $\otimes:\bimod \times \bimodinf \to \bimodinf$. Therefore, these are also  monoidal enriched categories and their underlying categories are $\tc$ and $\bimodinf$ respectively.
\end{lem}
\begin{proof}
This follows directly from Propositions \ref{external_hom_tc} and \ref{external_hom_bimod}.  Notice in particular that in the case of $\wtc$, the fact that the composition morphism $c$ is a map of vertical bicomplexes is equivalent to equation (\ref{comp_differential}) in Lemma~\ref{differential_morphism_tc}. 
\end{proof}

\begin{rmk}
There is an interpretation of $\bimodinf$ and $\wbimod$ via a standard categorical construction, the co-Kleisli
category for a comonad. Recall that $\Dd^\antishriek$ denotes the cofree $\Dd^\antishriek$-coalgebra functor
from $\bimod$ to the category of
$\Dd^{\antishriek}$-coalgebras, with left adjoint the forgetful functor $U$. Then  $\bimodinf$ is the co-Kleisli category of $\bimod$ for the comonad $U\Dd^\antishriek$. And this construction enriches. Namely, recall that
 $\Bimod$ denotes the  category of bigraded modules enriched over itself via its symmetric monoidal closed structure.
Then
$\wbimod$ is the enriched co-Kleisli category of $\Bimod$ for the enriched comonad $\underline{U\Dd^\antishriek}$.  To see this, note that the objects are the same, Proposition \ref{prop:corr} gives the isomorphism on morphisms or hom-objects and this isomorphism respects the composition.
\end{rmk}

We can describe the monoidal structures of $\wtc$ and $\wbimod$ explicitly.

\begin{lem}
\label{lem: mon structure on wtc}
The monoidal structure of $\wtc$ is given by the following map of vertical bicomplexes.

\begin{center}
\def\arraystretch{1.5}
\begin{tabular}{cll}
$\widehat{\otimes}: \wtc(A,B)\otimes \wtc(A',B')$&$\longrightarrow$ &$\wtc(A\otimes A',B\otimes B')$\\
$(f , g)$&$\mapsto$ &$(f\widehat{\otimes} g)_m:=\sum_{i+j=m} (-1)^{ij} f_i\otimes g_j$\\
\end{tabular}
\end{center}

The monoidal structure of $\wbimod$ is given by the restriction of this map.
\end{lem}

\begin{proof} 
The same argument as in the proof of Proposition \ref{external_hom_tc} shows that this is indeed a map of vertical bicomplexes.  To see that this is indeed the enriched monoidal structure induced by the monoidal structure of $\tc$ over $\vbic$, one can follow closely Remark \ref{monoidal_over_enrichment}.
\end{proof}

Next we introduce enriched structures on filtered modules and filtered complexes.
We enrich filtered modules over bigraded modules in the following way.

\begin{defi}\label{def:wfm}
The \textit{$\bimod$-enriched category of filtered modules} $\wfm$ is the enriched category given by the following data.
\begin{enumerate}
\item The objects of $\wfm$ are filtered modules.
\item For filtered modules $(K,F)$ and $(L,F)$, the bigraded module $\wfm(K,L)$ is given by
\[\wfm(K,L)_u^v:=\left\lbrace 
f:K\to L\,|\, f(F_qK^m)\subset F_{q+u}L^{m+v-u}, \forall m,q\in\ZZ
\right\rbrace. \]
\item The composition morphism is given by $c(f,g)=(-1)^{u|g|} fg$, where $f$ has bidegree $(u,v)$.
\item The unit morphism is given by the map $R\to \wfm(K,K)$ given by $1\to 1_K$.
\end{enumerate}
We denote by $\wsfm$ the full subcategory of $\wfm$ whose objects are split filtered modules.
\end{defi}

\begin{lem}\label{lem: enriched wfm}The above definition gives a 
well-defined $\bimod$-enriched category, $\wfm$.
\end{lem}
\begin{proof}
Let $f\in\wfm(L,M)_u^v$ and $g\in\wfm(K,L)_{u'}^{v'}$, where $(K,F)$, $(L,F)$ and $(M,F)$ are filtered modules.  For all $q\in\ZZ$ one has 
\[c(f,g)(F_qK^m)\subset f(F_{q+u'}L^{m+v'-u'})\subset F_{q+u+u'}M^{m+v-u+v'-u'},\] 
so $c(f,g)\in\wfm(K,M)_{u+u'}^{v+v'}$ showing that the composition morphism is a map in $\bimod$ and $1_K\in\wfm(K,K)_0^0$.
It is a short computation to show that the associativity and unit axiom hold.
\end{proof}

\begin{rmk}
Notice that morphisms $f\in\wfm(K,L)_0^v$ correspond precisely to degree $v$ morphisms of filtered modules which respect the filtration.
\end{rmk}

We will enrich filtered complexes over vertical bicomplexes analogously.

\begin{defi}\label{def: diff in wfm}
Let $(K,d^K,F)$ and $(L,d^L,F)$ be filtered complexes.  We define $\wfc(K,L)$ to be the vertical bicomplex whose underlying bigraded module is $\wfm(K,L)$ with vertical differential
\[\delta(f) := c(d^L, f )-(-1)^{<f,d^K>} c(f,d^K)=d^Lf-(-1)^{v+u} fd^K=d^Lf-(-1)^{|f|} fd^K\]
for $f\in\wfm(K,L)_u^v$.
\end{defi}

Note that $d^K\in\wfm(K,K)_0^1$ and $d^L\in\wfm(L,L)_0^1$. So $\delta(f)$ has bidegree $(u,v+1)$.
Also $\delta^2=0$ and thus $\wfc(K,L)$ is indeed a vertical bicomplex.

\begin{rmk}
Notice that $f\in \wfc(K,L)_u^v $ is a map of complexes if and only if $\delta(f)=0$.  In particular, $f$ is a morphism in $\fc$ if and only if $f\in \wfc(K,L)^0_0$ and $\delta(f)=0$.
\end{rmk}

\begin{defi}\label{def:wfc}
The $\vbic$-enriched category of filtered complexes $\wfc$ is the enriched category given by the following data.
\begin{enumerate}
\item The objects of $\wfc$ are filtered complexes.
\item For $K,L$ filtered complexes the hom-object is the vertical bicomplex $\wfc(K,L)$.
\item The composition morphism is given as in $\wfm$ in Definition \ref{def:wfm}. 
\item The unit morphism is given by the map $R\to \wfc(K,K)$ given by $1\mapsto 1_K$.
\end{enumerate}
We denote by $\wsfc$ the full subcategory of $\wfc$ whose objects are split filtered complexes.
\end{defi}

\begin{lem}
The above definition gives a well-defined $\vbic$-enriched category, $\wfc$.
\end{lem}
\begin{proof}
To see that the composition and unit maps are maps of vertical bicomplexes note that 
for $(K,d^K,F)$ and $(L,d^L,F)$ filtered complexes, $f\in\wfm(K,L)_u^v$ and $g\in\wfm(L,M)_{u'}^{v'}$
\[\delta(c(f,g))=c(\delta(f),g)+(-1)^{{|f|}} c(f,\delta(g)).\]
The associativity and unit axiom hold, because they hold in $\wfm$.
\end{proof}

\begin{lem}\label{lem:enrich_tc_fm_by_monoidal_over}
The enrichment of filtered complexes and filtered modules is the one induced by the external tensor products $\ast:\vbic \times \fc \to \fc$ and $\ast:\bimod \times \fm \to \fm$. Therefore, the enriched categories $\wfc$ and $\wfm$ are also  monoidal enriched categories and their underlying categories are $\fc$ and $\fm$ respectively.
\end{lem}
\begin{proof}
Since $\wfc$ is a well-defined $\vbic$-enriched category we have a bifunctor 
\[\wfc(-,-):\fc^{\mathrm{op}} \times \fc \lra \vbic.\]
It is left to show that we have natural bijections
$\Hom_{\fc}(A*K, L) {\cong} \Hom_{\vbic}(A,\wfc(K,L)).$
In one direction we have a map
\begin{alignat*}{2}
\Hom_{\fc}(\Tot_c(A)\otimes K,L)& {\longrightarrow} && \Hom_{\vbic}(A,\wfc(K,L))\\
 f                              &\mapsto            &&\tilde f: a\mapsto (k\mapsto f(a\otimes k)).
 \end{alignat*}
In the inverse direction we have a map
\begin{alignat*}{2}
\Hom_{\vbic}(A,\wfc(K,L))& \longrightarrow             & \Hom_{\fc}(\Tot_c(A)\otimes K,L)& \\
\tilde{g}:a\mapsto g_a \quad              &\mapsto         & g: (a_i)\otimes k \mapsto \sum g_{a_i}(k).
 \end{alignat*}
These constructions are inverse to each other and natural. 
\end{proof}

We can define the monoidal structure of $\wfc$ explicitly.
\begin{lem}\label{lem:monidal_wfc} 
The monoidal structure of $\wfm$ is given by the following map of vertical bicomplexes.
\[
\begin{array}{cccc}
\widehat{\otimes}:&\wfc(K,L)\otimes \wfc(K',L') &\to     & \wfc(K\otimes K', L\otimes L'),\\
             &(f,g)                        &\mapsto & f\widehat{\otimes} g:= (-1)^{u|g|} f\otimes g
\end{array}
\]
where $f$ has bidegree $(u,v)$.
\end{lem}
\begin{proof}
A direct computation using the Koszul rule gives
\[\begin{array}{ccl}
c(f,f')\widehat{\otimes} c(g,g')& = & (-1)^{u_f|f'|+u_g|g'|+(u_f+u_{f'})(|g|+|g'|)} ff'\otimes gg'\\
									    & = & (-1)^{<g,f'>}c(f\widehat{\otimes} g,f'\widehat{\otimes} g'),
\end{array}\]
which shows that the construction is functorial.
To see that this is indeed the enriched monoidal structure induced by the monoidal structure of $\fc$ over $\vbic$, one can follow closely Remark \ref{monoidal_over_enrichment}.
\end{proof}

\subsection{Enriched totalization}
The totalization functor and its properties extend to the enriched setting.  In this section we describe this structure explicitly.

\begin{lem}\label{lem:tot_lax_over_base}
The totalization functors 
\[
\begin{array}{ccc}
 \Tot:\bimod \to \fm & \qquad\qquad\text{and }\qquad\qquad & \Tot:\tc \to \fc
\end{array}\]
are lax monoidal functors over $\bimod$ and $\vbic$ respectively.  When restricted to the bounded case they are monoidal functors over $\bimod$ and $\vbic$ respectively.
\end{lem}
\begin{proof}
For the case of $\tc$, we have a natural transformation 
$\iota:\Tot_c(-)\Rightarrow \Tot(-)$
given by the inclusion. 
From Proposition \ref{prop:lax} we also have a natural transformation 
\[\mu:\Tot(-)\otimes \Tot(-)\Rightarrow \Tot(- \otimes -).\]
Thus, we have a natural transformation $\nu_{\Tot}$ given by the composite 
\[\nu_{\Tot}:-\ast \Tot(-):=
\Tot_c(-)\otimes \Tot (-)
\stackrel{\iota\otimes 1}{\Rightarrow}
\Tot(-)\otimes \Tot (-)
\stackrel{\mu}{\Rightarrow}
\Tot(- \otimes -).
\]
The coherence conditions follow from the coherence conditions for $\mu$ together with the fact that $\iota$ is the inclusion.  The case of $\bimod$ follows by restriction.
In the bounded case these are all natural isomorphisms.
\end{proof}

In order to describe the enriched totalization functors we first extend the definition of $\Tot$ to morphisms of any bidegree.

\begin{defi}
\label{def:tot_on_weirdmorphism}
Let $A,B$ be bigraded modules and $f\in\wbimod(A,B)_u^v$ we define 
\[\Tot(f)\in \wfm(\Tot(A),\Tot(B))_u^v\]  
to be given on any $a\in \Tot(A)^n$ by
\[  (\Tot(f)(a))_{j+u}:=\sum_{m\geq 0} (-1)^{(m+u)n} f_m(a_{j+m})\in B_{j+u}^{j+n+v}\subset \Tot(B)^{n+v-u}.\]
Let $K=\Tot(A)$, $L=\Tot(B)$ and $g\in \wfm(K,L)_u^v$ we define 
\[f:=\Tot^{-1}(g) \in\wbimod(A,B)_u^v\] 
to be $f:=(f_0, f_1, \ldots)$ where $f_i$ is given on each $A_j^{m+j}$ by the composite

	\begin{align*}
		f_i: A_j^{m+j} \hookrightarrow 
		\prod_{k\leq j} A_k^{m+k} = F_j(\Tot(A)^m)
		\stackrel{g}{\longrightarrow}
		&F_{j+u}(\Tot(B)^{m+v-u})\\
		&=\prod_{l\leq j+u} B_l^{m+v-u+l}
		\xrightarrowdbl{\times (-1)^{(i+u)m}}
		B_{j+u-i}^{m+j+v-i},
     	\end{align*}
where the last map is a projection and multiplication with the indicated sign.
\end{defi}

\begin{teo}
\label{prop:weirdtot}
Let $A, B$ be twisted complexes.  
The assignments $\wtot(A):=\Tot(A)$ and
\begin{center}
\def\arraystretch{1.5}
\begin{tabular}{cll}
$\wtot_{A,B}:\wtc (A,B)$&$\longrightarrow$ &$\wfc(\Tot(A),\Tot(B))$\\
$f$ &$\mapsto$ &$\Tot(f)$
\end{tabular}
\end{center}
define a $\vbic$-enriched functor $\wtot:\wtc \to \wfc$ which restricts to an isomorphism onto its image $\wsfc$.
Furthermore, this functor restricts to a $\bimod$-enriched functor
\[\wtot:\wbimod \to \wfm\]
which also restricts to an isomorphism onto its image $\wsfm$.
\end{teo}

\begin{proof}
We show first that this assignment defines a $\vbic$-enriched functor $\wtot$.
By Lemma \ref{lem:tot_lax_over_base} $\Tot$ is a lax functor over $\vbic$.
Thus, it is enough to show that $\wtot$ arises as the extension of $\Tot$ as described in the proof of Proposition \ref{prop:functor_extends_to_enriched}.
  
Let $A,B$ be twisted chain complexes. Let $K$ denote the vertical bicomplex $\wtc(A,B)$. Let $ev_{AB}$ denote the adjoint of the identity through the bijection (\ref{adjoint_external})
\begin{equation*}
\Hom_{\tc}(K\otimes A, B)\cong \Hom_{\vbic}(K,K),
\end{equation*}
of Proposition \ref{external_hom_tc}. Explicitly
\[(ev_{AB})_m(f\otimes a)=(-1)^{m|f|} f_m(a),\ f\in\wtc(A,B), a\in A.\]
The map
$\wtot_{A,B}:  K\rightarrow \wfc(\Tot(A),\Tot(B))$ is obtained as the adjoint through the bijection
\[\Hom_{\fc}(K\ast \Tot(A), \Tot(B)) {\cong} \Hom_{\vbic}(K,\wfc(\Tot(A),\Tot(B))),\]
of Lemma \ref{lem:enrich_tc_fm_by_monoidal_over} of the composite
\[\xymatrix{K\ast \Tot(A)= \Tot_c(K)\otimes \Tot(A)\ar[r]^>>>>{\mu_{K,A}}&\Tot(K\otimes A)\ar[rr]^{\Tot(ev_{AB})} && \Tot(B)}\]
as in Proposition \ref{prop:functor_extends_to_enriched}.

For $f\in\wtc(A,B)_u^v,  a=(a_k)_k\in\Tot^n(A)$ one has
\[\mu_{K,A}(f\otimes a)_k=(-1) ^{un} f\otimes a_{k-u},\]
\begin{align*}
 (\Tot(ev_{AB})\circ\mu_{K,A}(f\otimes a))_{j+u}&=\sum_{m\geq 0} (-1)^{m(n+v-u)} (ev_{AB})_m(\mu_{K,A}(f\otimes a)_{j+u+m})\\
 &=\sum_{m\geq 0} (-1)^{m(n+v-u)+un} (ev_{AB})_m(f\otimes a_{j+m})\\
& =\sum_{m\geq 0} (-1)^{m(n+v-u)+un+m(v-u)} f_m(a_{j+m}).
 \end{align*}
As a consequence
\[  (\Tot(f)(a))_{j+u}=\sum_{m\geq 0} (-1)^{(m+u)n} f_m(a_{j+m}).\]

To see that $\wtot$ restricts to an isomorphism onto its image we construct a $\vbic$-enriched functor which is inverse to the restriction of $\wtot$ onto its image.  
Let $K, L$ be split filtered complexes, then we define $\wtot^{-1}$ to be given by the assignments $\wtot^{-1}(K):=\Tot^{-1}(K)$ and 
\begin{center}
\def\arraystretch{1.5}
\begin{tabular}{cll}
$\wtot^{-1}_{K,L}:\wfc(K,M)$ &$\longrightarrow$ &$\wtc(\Tot^{-1}(K),\Tot^{-1}(L))$\\
$f$ &$\mapsto$ &$\Tot^{-1}(f)$.
\end{tabular}
\end{center}
A computation shows that this is indeed a map of vertical bicomplexes.  
Furthermore, we have that $\Tot(\Tot^{-1}(K))=K$ and $\Tot^{-1}(\Tot(A))=A$.

Now we check that $\Tot^{-1}(\Tot(f))=f$.  Write $\pi_{l}$ for the projection $\prod_{l\leq j+u} B_l^{m+v-u+l}\twoheadrightarrow B_{l}^{m+v-u+l}$ (without a sign).
For $f$ of bidegree $(u,v)$ and $a_j\in A_j^{m+j}$ one has that
\[
(\Tot^{-1}\circ\Tot(f))_i(a_j)=(-1)^{(i+u)m}\pi_{j+u-i}(\sum_{k\geq 0}(-1)^{(k+u)m}f_k(a_{j-i+k}))=f_i(a_j), \]
showing that $\Tot^{-1}(\Tot(f))=f$.

To see that $\Tot(\Tot^{-1}(g))=g$, let $g:(\Tot(A))^n\rightarrow (\Tot(B))^{n+v-u}$ be a map of bidegree $(u,v)$. For $(a)\in (\Tot(A))^n$ we write
$g(a)=(g_k(a))_{k\in \ZZ}$, where $g_k(a)\in B_k^{n+v-u+k}$.
Recall that for all $q$, we have that $g(F_q((\Tot(A))^n)\subset F_{q+u}(\Tot(B))^{n+v-u}$.
Fix $q\in\ZZ$.
If $(a)\in (\Tot(A))^n$ then  write $a=\alpha_{q-1}+\sum_{m\geq 0} a_{q+m}$
where $\alpha_{q-1}\in \prod_{r\leq q-1} A_r^{n+r}= F_{q-1}(\Tot(A))^n$ and  $\sum_{m\geq 0} a_{q+m}$ is finite.
As a consequence
\[ g_{q+u}(a)=\sum_{m\geq 0} g_{q+u}(a_{q+m}), \]

\[\Tot^{-1}(g)_m(a_{j+m})=(-1)^{(m+u)n} g_{j+u}(a_{j+m}).\]
Hence
\[(\Tot\circ\Tot^{-1}(g))(a)_{j+u}=\sum_{m\geq 0} (-1)^{(m+u)n} \Tot^{-1}(g)_m(a_{j+m})=\sum_{m\geq 0} g_{j+u}(a_{j+m})=g_{j+u}(a),\]
showing that $\Tot(\Tot^{-1}(g))=g$. 

It follows that $\wtot^{-1}$ is associative and unital since it is the inverse to  $\wtot$ and thus it defines a $\vbic$-enriched functor $\wtot^{-1}:\wsfc \lra \wtc$ which is inverse to the restriction of $\wtot$ onto its image.
The statement on $\wbimod$ follows by restriction.
\end{proof}

\begin{prop}\label{prop:lax_enriched}
The enriched functors 
\[\wtot:\wbimod\to \wfm\, ,\quad \quad \quad \wtot:\wtc\to \wfc\] 
are lax symmetric monoidal in the enriched sense and when restricted to the bounded case they are strong symmetric monoidal in the enriched sense.
\end{prop}
\begin{proof}
This follows from Propositions \ref{prop:lax}, \ref{prop:functor_extends_to_enriched} and Lemma \ref{lem:tot_lax_over_base}.
\end{proof}

\subsection{Derived $A_\infty$-algebras as $A_\infty$-algebras in twisted complexes}
We reinterpret the category of derived $A_\infty$-algebras as the category of $A_\infty$-algebras in twisted  complexes.  Generally, given an operad $\Pp$ on a symmetric monoidal category one 
studies $\Pp$-algebra structures on objects of the same category. We can extend this to the case of monoidal categories over a base via the following definition due to~\cite{Fresse}.

\begin{defi}
Let $\Cc$ be a monoidal category over $\vcat$ and let $\Pp$ be an operad in $\vcat$.  A $\Pp$-algebra in $\Cc$ consists of an object $A\in \Cc$, together with maps 
\[\Pp(n)\ast A^{\otimes n} \longrightarrow A\]
for which the unit and associativity axioms hold.
\end{defi}

We can give an equivalent definition by means of an enriched endomorphism operad.

\begin{defi}
Let $\ccat$ be a monoidal $\vcat$-enriched category and $A$ an object of $\ccat$.  We define $\End_A$ to be the collection in $\vcat$ given by 
\[\End_A(n):=\ccat(A^{\otimes n}, A)\quad\quad \text{for } n\geq 1.\]
\end{defi}

\begin{lem}
For any $A\in\ccat$, the collection $\End_A$ defines an operad in $\vcat$ with unit 
\[1\stackrel{u_A}{\longrightarrow} \ccat(A,A)=\End_A(1)\]
and composition
\[\End_A(r)\otimes \End_A(n_1) \otimes \End_A(n_2)\otimes \dots \otimes \End_A(n_r) \to
\End_A(r)\otimes \ccat(A^{\otimes n}, A^{\otimes r}) {\longrightarrow}
\End_A(n),
\]
where $n=n_1+\dots +n_r$. The first morphism is given by the monoidal structure of $\ccat$ and the second is the composition of the symmetry morphism of $\vcat$ with the composition morphism of $\ccat$.
\end{lem}
\begin{proof}
The appropriate diagrams commute by associativity of composition in an enriched category.  We also refer the reader to \cite[Definition 3.4.1]{Fresse}.
\end{proof}

\begin{example}
For any twisted  complex $A$ and any filtered complex $K$ we have operads in vertical bicomplexes $\End_A$ and $\End_K$.
\end{example}

The following result gives an equivalent interpretation of $\Pp$-algebras in monoidal categories over a base.

\begin{prop}~\cite[Proposition 3.4.3]{Fresse}
\label{fresse_enriched_end}
Let $\Cc$ be a monoidal category over $\vcat$, let $\Pp$ be an operad in $\vcat$ and $A$ an object in $\Cc$.  Then there is a one-to-one correspondence between $\Pp$-algebra structures on $A$ and morphisms of operads $\Pp\to\End_A$.\qed
\end{prop}

Operad morphisms can be constructed from functors on ordinary categories which behave well with respect to the monoidal structure.  The result below due to Fresse is originally stated for the monoidal case.  However, all his methods extend to the lax monoidal setting as we describe below.

\begin{prop}~\cite[Proposition 3.4.7]{Fresse}
\label{prop:functor_induces_operad_map}
Let $\Cc$ and $\Dd$ be monoidal categories over $\vcat$.  
Let $F:\Cc\to\Dd$ be a lax monoidal functor over $\vcat$.  
Then for any $X\in\Cc$ there is an operad morphism
\[\End_X{\longrightarrow} \End_{F(X)}.\]
\end{prop}
\begin{proof}
By Proposition~\ref{prop:functor_extends_to_enriched}, 
$F$ induces a $\vcat$-enriched functor $\underline{F}$ which is lax monoidal in the enriched sense.
Using this one can construct the operad map for each arity $n$ as the composite 
\begin{align*}
\End_X(n):=\ccat(X^{\otimes n},X)
\to
\ddcat(F(X^{\otimes n}),F(X)) 
\to
\ddcat(F(X)^{\otimes n},F(X))=\End_{F(X)}(n),
\end{align*}
where the first map is given by the $\vcat$-enriched functor $\underline{F}$ and the second map comes from its lax monoidal structure.
Since naturality holds by construction these assemble into a morphism of operads.
\end{proof}

We reinterpret  $dA_\infty$-algebras as $A_\infty$-algebras in twisted  complexes using the structure of $\tc$ as a monoidal category over $\vbic$.

\begin{prop}\label{dA_and_A_tc}
Let $(A, d^A)$ be a twisted  complex, $A$ its underlying bigraded module and consider $A_\infty$ as an operad in $\vbic$ sitting in horizontal degree zero.  There is a one-to-one correspondence between $A_\infty$-algebra structures on $(A, d^A)$ and $dA_\infty$-algebra structures on $A$ which respect the twisted complex structure of $A$.  More precisely, let $\mathrm{End}_A$ be the operad in $\vbic$ corresponding to the bigraded module $A$.  We have a natural bijection
\[\Hom_{\mathrm{vbOp}}(A_\infty, \End_A)\cong 
\Hom_{\mathrm{vbOp},d{^A}}(dA_\infty, \mathrm{End}_A),
\]
where $\mathrm{vbOp}$ denotes the category of operads in vertical bicomplexes and $\Hom_{\mathrm{vbOp},d{^A}}$ denotes the subset of morphisms which send $\mu_{i1}$ to $d_i^A$, $i\geq 1$.
\end{prop}

\begin{proof}
Let $f:A_\infty \to \End_A$ be a map of operads in $\vbic$.  Since $A_\infty$ is quasi-free, this is equivalent to maps in $\vbic$
\[(A_\infty(v),\partial_\infty)\to (\End_A(v),\delta)\]
 for each $v\geq 1$, which are determined by elements   $M_v:=f(\mu_v)\in \End_A(v)$ for $v\geq 2$ of bidegree $(0, 2-v)$ such that
\begin{equation}\label{relation_A_tc}
\delta (M_v) = f(\partial_\infty(\mu_v)).
\end{equation}
Moreover, $M_v:=(  m_{0v},   m_{1v}, \ldots)$ where $  m_{uv}:=(M_v)_u: A^{\otimes v} \to A$ is a map of bidegree $(-u, 2-u-v)$.
We first compute the left-hand side of (\ref{relation_A_tc}). Since $\delta(M_v)=c(d^A, M_v)-(-1)^v c(M_v, d^{A^{\otimes v}})$, we have
\begin{alignat*}{2}
(\delta (M_v))_u & = \sum_{{\substack{
u=i+p
}}}  (-1)^{iv} d_i^A (M_v)_p -(-1)^v \sum_{{\substack{
u=i+p\\
v=1+r+t 
}}} (-1)^i (M_v)_i (1^{\widehat{\otimes} r}\widehat{\otimes} d^A_p\widehat{\otimes} 1^{\widehat{\otimes} t})
\end{alignat*}
and 
\begin{alignat*}{2}
f(\partial_\infty(\mu_v))_u &=
-\sum_{ \substack{v=r+q+t\\ j=r+1+t\\ j,q>1}}(-1)^{rq+t} \left( c(M_j, (1^{\widehat{\otimes} r} \widehat{\otimes} M_q \widehat{\otimes} 1^{\widehat{\otimes} t}))\right)_u\\
& =
-\sum_{ \substack{i+p=u\\v=r+q+t\\ j=r+1+t\\ j,q>1}} (-1)^{rq+t+iq} (M_j)_i(1^{\otimes r} \otimes (M_q)_p \otimes 1^{\otimes t})\\
&= 
-\sum_{ \substack{i+p=u\\v=r+q+t\\ j=r+1+t\\ j,q>1}} (-1)^{rq+t+iq} \widetilde m_{ij}(1^{\otimes r} \otimes \widetilde m_{pq} \otimes 1^{\otimes t}).\\
\end{alignat*}
Then by setting the notation $\widetilde m_{i1}=d_i^A$, the relation 
\[\delta (M_n) - f(\partial_\infty(\mu_n))=0\]
gives us the relations
\begin{equation*}
\sum_
{\substack{
u=i+p\\
v=v+q+t\\
j=1+r+t
}}
(-1)^{rq+t+iq}\widetilde m_{ij}({1}^{\otimes r}\otimes \widetilde m_{pq}\otimes {1}^{\otimes t})=0.
\end{equation*}
By setting $m_{ij}=(-1)^{ij} \widetilde m_{ij}$ one obtains the relation
 
\begin{equation}
\sum_
{\substack{
u=i+p\\
v=v+q+t\\
j=1+r+t
}}
(-1)^{rq+t+pj} m_{ij}({1}^{\otimes r}\otimes  m_{pq}\otimes {1}^{\otimes t})=0 \tag{$A_{uv}$}
\end{equation}
which by \cite{LRW} is equivalent to giving a map 
$dA_\infty \to \mathrm{End}_A$
of operads in $\vbic$.
\end{proof}

\begin{rmk}
\label{Ainftwoways}
Note that as in the case of $A_\infty$-algebras in $\cpx$ we have two equivalent descriptions of $A_\infty$-algebras in $\tc$. 
\begin{enumerate}
\item A twisted  complex $(A,d^A)$ together with a morphism $A_\infty \to \End_A$ of operads in $\vbic$, 
which is determined by a family of elements $M_i^A\in \wtc(A^{\otimes i}, A)_0^{2-i}$ for $i\geq 2$ for which the $(A'_{0i})$ relations hold,
where the composition is the one prescribed by the composition morphisms of $\wtc$.
\item A bigraded module $A$ together with a family of elements $M_i\in \wbimod(A^{\otimes i}, A)_0^{2-i}$ for $i\geq 1$ for which the $(A_{0i})$ relations hold, where the composition is the one prescribed by the composition morphisms of $\wbimod$.
\end{enumerate}
Here $(A'_{0i})$ and $(A_{0i})$ are the $A_\infty$ relations for $i\geq 2$ or $i\geq 1$ respectively.  Since the composition morphism in $\wbimod$ is induced from the one in $\wtc$ by forgetting the differential, these two presentations are equivalent.
\end{rmk}

We now consider infinity morphisms and composition.

\begin{defi}\label{def:aintc}
Let $A, B, C$ be twisted  complexes and $(A,M_i^A)$, $(B,M_i^B)$ and $(C,M_i^C)$ be $A_\infty$-algebra structures on them. An \emph{$A_\infty$-morphism in twisted  complexes $F: (B,M_i^B) \to (C,M_i^C)$ } is a family of elements $F:=\{F_j\in \wtc(B^{\otimes j},  C)_0^{1-j}\}_{j\geq 1}$ for which the $A_\infty$ relations hold, i.e., 
\begin{equation}
\sum\limits_{\substack{v=r+q+t \\ j=1+r+t}} (-1)^{rq+t}
c(F_j, (1^{\widehat{\otimes} r} \widehat{\otimes} M^B_{q} \widehat{\otimes} 1^{\widehat{\otimes} t})) = \\
\sum\limits_{\substack{v=q_1+\cdots+q_j }} (-1)^{\sigma} c(M^{C}_{j},(F_{q_1}\widehat{\otimes}\cdots\widehat{\otimes} F_{q_j})),
\tag{$B_{0v}$}\end{equation}
where $\sigma=\sum\limits_{k=1}^{j-1}q_k(j+k)+q_k(\sum\limits_{s=k+1}^j q_s)$ and $c$ is described in Definition \ref{composition_defi}. 

Let $F: (B,M_i^B) \to (C,M_i^C)$ and $G: (A,M_i^A) \to (B,M_i^B)$  be $A_\infty$-morphisms in twisted complexes.  Their composite is the  $A_\infty$-morphism in twisted complexes $F\circ G:(A,M_i^A) \to (C,M_i^C)$ given by 
\begin{equation}
(F\circ G)_v := \\
\sum\limits_{\substack{v=q_1+\cdots+q_j }} (-1)^{\sigma} 
c(F_{j}, (G_{q_1}\widehat{\otimes}\cdots\widehat{\otimes} G_{q_j})).
\tag{$C_{0v}$}\end{equation}

The \emph{category of $A_\infty$-algebras in twisted complexes}, denoted $\Ainfintc$,
is the category with objects $A_\infty$-algebras in twisted complexes 
and whose morphisms are $A_\infty$-morphisms in twisted complexes.
\end{defi}

\begin{teo} \label{dA_Atc} The construction above extends to a functor 
$\Psi: \Ainfintc \to \dAinf{R}$
which is an isomorphism of categories.

\end{teo}

\begin{proof}
On objects $\Psi(A,M_i^A)=(A, m_{ij}^A)$ takes an $A_\infty$-algebra in twisted  complexes and associates to it its corresponding $dA_\infty$-algebra as described in Proposition \ref{dA_and_A_tc}.  

On morphisms, consider $F:(A,M_i^A) \to (B, M_i^B)$, a morphism in $A_\infty \text{-algs in } \tc$ which is given by $F:=\{F_j\in \wtc(A^{\otimes j},  B)_0^{1-j}\}_{j\geq 1}$, where $F_j:=(\widetilde f_{0j}, \widetilde f_{1j}, \widetilde f_{2j}, \ldots)$. 
The relations $(B_{0v})$ translate to
\begin{equation}
\sum\limits_{\substack{ u=i+p\\ v=r+q+t \\ j=1+r+t}} (-1)^{\sigma_l}
\widetilde f_{ij} (1^{\otimes r} \otimes \widetilde m^A_{pq} \otimes 1^{\otimes t}) = \\
\sum\limits_{\substack{u=i+p_1+\cdots+p_j\\ v=q_1+\cdots+q_j }} (-1)^{\sigma_r} \widetilde m^{B}_{ij} (\widetilde f_{p_1q_1}\otimes\cdots\otimes \widetilde f_{p_jq_j}).
\tag{$\widetilde{B}_{uv}$}\end{equation}
Multiplying the equation by $(-1)^{(i+p)(j+q)}=(-1)^{u(v+1)}$ and setting $f_{ij}=(-1)^{ij}\widetilde f_{ij}$ and  $m_{pq}=(-1)^{pq}\widetilde m_{pq}$
one obtains the equation
\begin{equation*}
\sum\limits_{\substack{ u=i+p\\ v=r+q+t \\ j=1+r+t}} (-1)^{\widetilde\sigma_l}
  f_{ij} (1^{\otimes r} \otimes  m^A_{pq} \otimes 1^{\otimes t}) = \\
\sum\limits_{\substack{u=i+p_1+\cdots+p_j\\ v=q_1+\cdots+q_j }} (-1)^{ \widetilde \sigma_r}  m^{B}_{ij} ( f_{p_1q_1}\otimes\cdots\otimes  f_{p_jq_j}).
\end{equation*}

Let us compute the signs modulo 2:
\[\widetilde\sigma_l=rq+t+(i+p)(j+q)+ij+pq+iq=rq+t+pj,\]
\begin{alignat*}{2}
\widetilde\sigma_r=&\sum\limits_{k=1}^{j-1}q_k(j+k)+q_k(\sum\limits_{s=k+1}^j q_s)+u+uv+ij+\sum\limits_{k=1}^j p_kq_k+i(\sum\limits_{k=1}^j (q_k+1))+
\sum\limits_{k=1}^j(1+q_k)(\sum\limits_{s=1}^{k-1}p_s)\\
=&\sum\limits_{k=1}^{j-1}q_k(j+k)+q_k(\sum\limits_{s=k+1}^j q_s)+u+uv+iv+\sum\limits_{k=1}^j q_k(\sum\limits_{s=1}^{k}p_s)+\sum\limits_{s=1}^{j-1}p_s(\sum\limits_{k=s+1}^j1)\\
=&u+\sum\limits_{k=1}^{j-1}(p_k+q_k)(j+k)+\sum\limits_{k=1}^{j-1}q_k(\sum\limits_{s=k+1}^j p_s+q_s).
\end{alignat*}
This gives exactly the relation defining 
morphisms of $dA_\infty$-algebras:
\begin{equation}
\sum\limits_{\substack{ u=i+p, v=r+q+t \\ j=1+r+t}} (-1)^{rq+t+pj}
f_{ij} (1^{\otimes r} \otimes m^A_{pq} \otimes 1^{\otimes t}) = \\
\sum\limits_{\substack{ u=i+p_1+\cdots+p_j,\\ v=q_1+\cdots+q_j }} (-1)^{\sigma} m^{B}_{ij}(f_{p_1q_1}\otimes\cdots\otimes f_{p_jq_j}),
\tag{$B_{uv}$}\end{equation}
where $\sigma=u+\sum\limits_{k=1}^{j-1}(p_k+q_k)(j+k)+q_k(\sum\limits_{s=k+1}^j p_s+q_s)$.
Moreover, any morphism 
between $dA_\infty$-algebras can be constructed in this way.  
Therefore this construction is a bijection on morphisms. 
Finally, this construction is functorial.
The relations $(C_{0v})$ translate to

\begin{equation}
\widetilde{(F\circ G)}_{uv}=
\sum\limits_{\substack{u=i+p_1+\cdots+p_j\\ v=q_1+\cdots+q_j }} (-1)^{\sigma_r} \widetilde{f}_{ij} (\widetilde g_{p_1q_1}\otimes\cdots\otimes \widetilde g_{p_jq_j}).
\tag{$\widetilde{C}_{uv}$}\end{equation}
Setting $f_{ij}=(-1)^{ij}\widetilde f_{ij}$,  $g_{pq}=(-1)^{pq}\widetilde g_{pq}$, and ${(F\circ G)}_{uv}=(-1)^{uv}\widetilde{(F\circ G)}_{uv}$,
one obtains the equation

\begin{equation*}
{(F\circ G)}_{uv}=
\sum\limits_{\substack{u=i+p_1+\cdots+p_j\\ v=q_1+\cdots+q_j }} (-1)^{\widetilde \sigma_r'}  f_{ij} ( g_{p_1q_1}\otimes\cdots\otimes  g_{p_jq_j}),
\end{equation*}
with $\widetilde\sigma_r'=\widetilde\sigma_r+u+uv+uv=\sum\limits_{k=1}^{j-1}(p_k+q_k)(j+k)+\sum\limits_{k=1}^{j-1}q_k(\sum\limits_{s=k+1}^j p_s+q_s)$,
which is the composition of morphisms of $dA_\infty$-algebras.
\end{proof}

\subsection{Derived $A_\infty$-algebras as filtered $A_\infty$-algebras}

From the fact that $\wbtc$ and $\wbsfc$
are isomorphic $\vbic$-enriched monoidal categories, we now reinterpret $dA_\infty$-algebras, in the bounded case,  in terms of split filtered $A_\infty$-algebras. First we recall the definition of
filtered $A_\infty$-algebras and their morphisms. Filtered $A_\infty$-algebras and their associated 
spectral sequences have been previously 
studied in \cite{Lapin2003}, \cite{Lapin2008} and \cite{Her}.

\begin{defi}\label{def:fa_obj}
A \emph{filtered $A_\infty$-algebra} is an $A_\infty$-algebra $(A,m_i)$ together with a filtration
$\{F_pA^i\}_{p\in\ZZ}$ on each $R$-module $A^i$ such that for all $i\geq 1$ and all $p_1,\ldots,p_i\in\ZZ$ and $n_1,\dots, n_i\geq 0$,
\[m_i(F_{p_1}A^{n_1}\otimes \cdots \otimes F_{p_i}A^{n_i})\subseteq F_{p_1+\cdots+p_i}A^{n_1+\cdots+n_i+2-i}.\]
Such a filtered $A_\infty$-algebra is said to be \textit{split} if $A=\Tot(B)$ 
is the total graded module of a bigraded $R$-module $B=\{B_i^j\}$ and $F$ is
the column filtration of $\Tot(B)$.
\end{defi}

\begin{rmk}
\label{filtered_A_infty_operad}
Consider $A_\infty$ as an operad in filtered complexes with the trivial filtration and let $K$ be a filtered complex.
There is a one-to-one correspondence between filtered $A_\infty$-structures on $K$ and morphisms of operads in filtered complexes
$A_\infty \to \mathrm{End}_K$.
To see this, notice that if one forgets the filtrations such a map of operads gives an $A_\infty$ structure on $K$. 
The fact that this is a map of operads in filtered complexes implies that all the $m_i$s respect the filtrations.
\end{rmk} 

\begin{defi}\label{def:fa_mor}
A \emph{morphism of filtered $A_\infty$-algebras} from $(A,m_i,F)$ to $(B,m_i,F)$
is a morphism $f:(A,m_i)\to (B,m_i)$ of $A_\infty$-algebras
such that each map $f_j:A^{\otimes j}\to A$ is compatible with filtrations:
\[f_j(F_{p_1}A^{n_1}\otimes \cdots \otimes F_{p_j}A^{n_j})\subseteq F_{p_1+\cdots+p_j}B^{n_1+\cdots+n_j+1-j},\]
for all $j\geq 1$, $p_1, \ldots p_j\in \ZZ$ and  $n_1, \ldots, n_j\geq 0$.
\end{defi}
Denote by $\fAinf(R)$ the category of filtered $A_\infty$-algebras.
Composition is given as in the unfiltered case (this respects the filtration).
We consider the following full subcategories of $\fAinf(R)$.
\begin{itemize}
\item $\sfAinf(R)$: the subcategory whose objects are split filtered $A_\infty$-algebras.
\item $\bfAinf(R)$: the subcategory whose objects are non-negatively filtered $A_\infty$-algebras.
\item $\bsfAinf(R)$: the subcategory whose objects are split non-negatively filtered $A_\infty$-algebras.
\end{itemize}
That is, we have full embeddings 
\begin{equation}\label{def:sub_cats_fa}
\begin{tikzpicture}
  \node (a) {$\sfAinf(R)$};
  \node[right=2cm of a] (b) {$\fAinf(R)$};
  \node[right=2cm of b] (c) {$\bfAinf(R)$};
   \node[above=1cm of b] (d) {$\bsfAinf(R)$};
  \draw[->]
    (a) edge [right hook->] (b)
    (c) edge [left hook->](b)
    (d) edge [left hook->] (a)
    (d) edge [right hook->] (c);
\end{tikzpicture}
\end{equation}

\begin{lem}
\label{equivalence_enriched_end}
For any twisted complex $A$ there is a morphism of operads
\[\End_A{\longrightarrow} \End_{\Tot(A)},\]
which is an isomorphism of operads if $A$ is bounded.
\end{lem}
\begin{proof}

The existence of the morphism of operads follows directly from Proposition \ref{prop:functor_induces_operad_map} and in this case it is given in arity $n$ by the composite 
\begin{align*}
\End_A(n):=\wtc(A^{\otimes n},A)
\stackrel{\wtot_{A^{\otimes n},A}}{\xrightarrow{\hspace*{1.5cm}}}\wfc(\Tot(A^{\otimes n}),\Tot(A)) 
\longrightarrow
&\wfc(\Tot(A)^{\otimes n},\Tot(A))\\
&=\End_{\Tot(A)}(n).
\end{align*} 
In the bounded case the first map is an isomorphism by Theorem \ref{prop:weirdtot} and the second is an isomorphism by Proposition \ref{prop:lax_enriched}.
\end{proof}

\begin{prop}\label{dA_and_A_fc}
Let $(A, d^A)\in\tc^{b}$ be an $(\NN, \ZZ)$-graded twisted  complex and $A$ its underlying bigraded module.
There is a one-to-one correspondence between filtered $A_\infty$-algebra structures on $\Tot(A)$ and $dA_\infty$-algebra structures on $A$ which respect the twisted complex structure of $A$.  This bijection is induced by a one-to-one correspondence between filtered $A_\infty$-algebra structures on $\Tot(A)$ and $A_\infty$-algebra structures on $(A, d^A)$.  
More precisely we have natural bijections
	\begin{align*}
		\Hom_{\mathrm{vbOp},d^A}(dA_\infty, \mathrm{End}_A)&\cong
			\Hom_{\mathrm{vbOp}}(A_\infty, \End_A)\\
	&\cong\Hom_{\mathrm{vbOp}}(A_\infty, \End_{\Tot(A)})\\
	&\cong\Hom_{\mathrm{fCOp}}(A_\infty, \mathrm{End}_{\Tot(A)}),
\end{align*}
where $\mathrm{vbOp}$ and $\mathrm{fCOp}$ denote the categories of operads in $\vbic$ and $\fc$ respectively, and $\Hom_{\mathrm{vbOp},d{^A}}$ denotes the subset of morphisms which send $\mu_{i1}$ to $d_i^A$.  We view $A_\infty$ as an operad in $\vbic$ sitting in horizontal degree zero or as an operad in filtered complexes with trivial filtration.  
\end{prop}

\begin{proof}

The first isomorphism holds by Proposition \ref{dA_and_A_tc}.  The second isomorphism follows directly from Lemma \ref{equivalence_enriched_end}.  Finally, to see the third isomorphism let $f:A_\infty \to \End_{\Tot(A)}$ be a map of operads in $\vbic^{b}$.  Again, since $A_\infty$ is quasi-free, this is equivalent to maps in $\vbic^{b}$
\[(A_\infty(n),\partial_\infty)\to (\End_{\Tot(A)}(n),\delta)\] 
which are determined by elements $M_n:=f(\mu_n)\in\End_{\Tot(A)}(n)$ of bidegree $(0,2-n)$ such that
\[\delta (M_n) = f(\partial_\infty(\mu_n)).\]
Since $A$ is $(\NN,\ZZ)$-graded, $\Tot$ is symmetric monoidal, and thus we have that
\[\delta M_n =  c(d^{\Tot(A)}, M_n) +(-1)^{2-n} c(M_n, d ^{\Tot(A)^{\otimes n}}).\]
So these maps give the complex $\Tot(A)$ the structure of an $A_\infty$-algebra.
Moreover, all of the $M_n$s respect the filtration since they have horizontal degree zero. 
Therefore, the map $f$ gives $\Tot(A)$ the structure of a split filtered $A_\infty$-algebra and
it is clear that any filtered $A_\infty$-algebra structure on $\Tot(A)$ can be described by such $f$.
\end{proof}

This construction extends to infinity morphisms. 

\begin{teo}
\label{dA-filtered-bounded-tc}
The totalization functor extends to a functor 
\[\Phi: \Ainfintc \to \fAinf(R) \]
which in the bounded case restricts to an isomorphism between the categories of bounded $A_\infty$-algebras in twisted  complexes and split non-negatively filtered $A_\infty$-algebras.\end{teo}

Before proving this result we make the following remark. 

\begin{rmk}
Let $\Cc$ and $\Dd$ be monoidal categories over $\vcat$, let $\Pp$ be an operad in $\vcat$ and let $F:\Cc\to\Dd$ be a symmetric monoidal functor over $\vcat$.  
In \cite[Observation 3.2.14]{Fresse}, Fresse shows that $F$ extends to a functor 
\[\widetilde{F}:\Pp\mbox{-}\mathrm{Alg}(\Cc)\to \Pp\mbox{-}\mathrm{Alg}(\Dd)\] 
from the category of $\Pp$-algebras in $\Cc$ with $\Pp$-algebra morphisms to the category of $\Pp$-algebras in $\Dd$ with $\Pp$-algebra morphisms.

His methods extend to the case where $F$ is lax monoidal over $\vcat$.  
Let $\vcat=\vbic$, $\Pp=\mathrm{A}_\infty$, $\Cc=\tc$, $\Dd=\fc$ and $F=\Tot$. 
Then, the totalization functor extends to a functor 
\[\widetilde{\Tot}:\mathrm{A}_\infty\mbox{-}\mathrm{Alg}(\tc)\to \mathrm{A}_\infty\mbox{-}\mathrm{Alg}(\fc)\]
between the categories of $\mathrm{A}_\infty$-algebras in $\tc$ with strict morphisms to the category of $\mathrm{A}_\infty$-algebras in $\fc$ with strict morphisms.
Our result implies that this functor extends to their respective categories with infinity morphisms. 
\end{rmk}

\begin{proof}[Proof of Theorem \ref{dA-filtered-bounded-tc}]
The functor on objects is given as in Proposition \ref{equivalence_enriched_end}.  Here we describe this explicitly on elements.  Let $(A,M_i)$ be an $A_\infty$-algebra in $\wtc$, that is we have $A\in\tc$ and $M_i\in \wtc(A^{\otimes i},A)_0^{2-i}$ satisfying the $A_\infty$-relations
\begin{equation}
\sum_
{\substack{
v=v+q+t\\
j=1+r+t
}}
(-1)^{rq+t}c(M_i,{1}^{\widehat{\otimes} r}\widehat{\otimes}  M_q\widehat{\otimes} {1}^{\widehat{\otimes} t})=0. \tag{$A_{0v}$}
\end{equation}
Following the notation of Proposition \ref{prop:lax_enriched}, let 
\[\mu_i:=\mu_{A,\ldots,A}: \Tot(A)^{\otimes i}\rightarrow \Tot(A^{\otimes i}),\]
\[\mu_{r,q,t}:=\mu_{A,\ldots,A,A^{\otimes q},A,\ldots,A}:\Tot(A)^{\otimes r}\otimes\Tot(A^{\otimes q})\otimes\Tot(A)^{\otimes t}\rightarrow \Tot(A^{\otimes r+q+t})\]
and define 
\[m_i:=c(\Tot(M_i),\mu_i):\Tot(A)^{\otimes i}\rightarrow \Tot(A).\]
Note first that for any $i$ the map $m_i$ has horizontal degree $0$ thus it respects the filtrations.  Now, we compute 
\begin{align*}
\sum_
{\substack{
v=v+q+t\\
j=1+r+t
}}
(-1)^{rq+t}&m_i ({1}^{{\otimes} r}{\otimes}  m_q {\otimes} {1}^{{\otimes} t}) \\
&= 
\sum_
{\substack{
v=v+q+t\\
j=1+r+t
}}
(-1)^{rq+t}c(\Tot(M_i),\mu_i)({1}^{\otimes r}{\otimes}  c(\Tot(M_q),\mu_q) {\otimes} {1}^{{\otimes} t})\\
&=
\sum_
{\substack{
v=v+q+t\\
j=1+r+t
}}
(-1)^{rq+t}\Tot(M_i ({1}^{\widehat{\otimes} r}\widehat{\otimes}  M_q\widehat{\otimes} {1}^{\widehat{\otimes} t}))
\mu_{r,q,t}({1}_{\Tot(A)}^{{\widehat{\otimes}} r}{\widehat{\otimes}}\mu_q{\widehat{\otimes}}
{1}_{\Tot(A)}^{{\widehat{\otimes}} t}) 
= 0.
\end{align*}
Here the first equality holds by definition, the second by naturality of $\mu$ and the third because the $M_i$s satisfy the relations $(A_{0v})$.  Thus $(\Tot(A), m_i)$ is a filtered $A_\infty$-algebra.  The same computation gives the result on morphisms and this is stable under the composition of morphisms giving the functor $\Phi$.  Furthermore, when $A\in \btc$ this functor restricts to an isomorphism by Proposition \ref{dA_and_A_fc}.
\end{proof}

\begin{cor}\label{dA-filtered-bounded-sf}
Let $\Tot$ denote the composite
\[\Tot: dA_\infty^b(R)\stackrel{\Psi^{-1}}{\longrightarrow} \bAinfintc(R) \stackrel{\Phi}{\longrightarrow} \bfAinf(R).\]
The functor $\Tot$  restricts to an isomorphism between the category of $(\NN,\ZZ)$-graded $dA_\infty$-algebras and the category of split non-negatively filtered $A_\infty$-algebras.  Furthermore, this functor fits into a commutative diagram of categories
\begin{center}
\begin{tikzpicture}
  \node (a) {$dA_\infty^b(R)$};
  \node[right=2cm of a] (b) {$\btc$};
  \node[below=2cm of a] (c) {$\bfAinf(R)$};
   \node[below=2cm of b] (d) {$\bfc$};
  \draw[->]
    (a) edge node [above]{$U$} (b)
    (b) edge node [right]{$\Tot$}(d)
    (a) edge node [left]{$\Tot$} (c)
    (c) edge node [below]{$U$} (d);
\end{tikzpicture}
\end{center}
where the horizontal arrows are forgetful functors and the vertical arrows are full embeddings.
\end{cor}

\begin{proof}
This follows directly from Theorems~\ref{dA_Atc} and~\ref{dA-filtered-bounded-tc}. 
 \end{proof}

\section{Derived $A_\infty$-algebras and $r$-homotopy}
\label{sec:der}
The main goal of this section is to study different but equivalent interpretations of the notion of $r$-homotopy 
for derived $A_\infty$-algebras.  We first define $r$-homotopy by constructing a functorial $r$-path object. 
Then, we study some of the properties of $r$-homotopy. Most notably, we show that $0$-homotopy defines 
an equivalence relation and we study the localized category $\dAinf{R}[\Ss_r^{-1}]$.
Finally, we give an operadic interpretation of $r$-homotopy and show that the two notions are equivalent.

\subsection{Twisted dgas and tensor product}
In general, the tensor product of two
$dA_\infty$-algebras does not inherit a natural $dA_\infty$-algebra structure giving
rise to a monoidal structure on $\dAinf{R}$.
The construction works if one of the components is a twisted differential graded algebra, as we show next.

\begin{defi}
A \textit{twisted dga} is a $dA_\infty$-algebra $(A,\mu_{ij})$ whose only non-zero structure morphisms 
are $\mu_{i1}$ for $i\geq 0$ and $\mu_{02}$.
\end{defi}

\begin{lem}\label{muj}
For $(A,\mu_{ij})$ a twisted dga, the following hold.
\begin{enumerate}
\item $\mu_{02}$ is associative.
\item \label{verify} $\mu_{i1}(\mu_{02})=\mu_{02}(1\otimes \mu_{i1})+\mu_{02}(\mu_{i1}\otimes 1)$ for all $i\geq 0$.
\item Let $\mu_n:A^{\otimes n}\to A$ be defined iteratively by
$\mu_n=\mu_{02}(\mu_{n-1}\otimes 1)$, with $\mu_2=\mu_{02}$. Then
\[\mu_{i1}(\mu_n)=\sum_{r+t+1=n}\mu_n(1^{\otimes r}\otimes \mu_{i1}\otimes 1^{\otimes t}) \quad\text{ for all }i\geq 0.\]
\end{enumerate}
\end{lem}
\begin{proof}
It suffices to check $(\ref{verify})$. Relation ($A_{i2}$) reads:
\[
-(\mu_{02}(1\otimes \mu_{i1})+\mu_{02}(\mu_{i1}\otimes 1))+\mu_{i1}(\mu_{02})=0.\qedhere
\]
\end{proof}

\begin{prop}\label{P:tensorproduct}
Let $(\Lambda,\mu_{ij})$ be a twisted dga and let $(A,m_{ij})$ be a $dA_\infty$-algebra. 
The bigraded module $\Lambda\otimes A$ is endowed with a $dA_\infty$-algebra structure given by
\[\widehat m_{i1}=\mu_{i1}\otimes 1_A+1_\Lambda\otimes m_{i1}\qquad 
\text{and}\qquad\widehat m_{ij}= (\mu_j\otimes m_{ij})\tau_j \quad\text{ for all }j\geq 2.\]
Here $\tau_j:(\Lambda\otimes A)^{\otimes j}\to \Lambda^{\otimes j}\otimes A^{\otimes j}$
denotes the standard isomorphism given by the symmetric monoidal structure and $\mu_j$ is defined in Lemma $\ref{muj}$.
\end{prop}
\begin{proof}For all $n\geq 0$, we have
\[\sum_{i+j=n} (-1)^j \widehat m_{i1}\widehat m_{j1}=\sum_{i+j=n}(-1)^j 
(\mu_{i1}\mu_{j1}\otimes 1+1\otimes m_{i1}m_{j1}+\mu_{i1}\otimes m_{j1}+(-1)^{ij+(1-i)(1-j)} \mu_{j1}\otimes m_{i1})=0.\]

Note that, for $j,q\geq 2$, we have
\[
\widehat m_{ij}(1^{\otimes r}\otimes \widehat m_{pq}\otimes 1^{\otimes t})
	=(\mu_{r+q+t}\otimes m_{ij}(1^{\otimes r} \otimes m_{pq}\otimes 1^{\otimes t}))\tau_{r+q+t}.
\]
Using this, 
for all $u\geq 0$ and all $v \geq 2$, we have
\begin{align*}
	&\sum_{\substack{u=i+p,\ v=j+q-1\\ j=1+r+t, } } 
		(-1)^{rq+t+pj} \widehat m_{ij}(1^{\otimes^r}_{\Lambda\otimes A}\otimes \widehat m_{pq}\otimes 	1^{\otimes^t}_{\Lambda\otimes A})\\
	&=\sum_{\substack{u=i+p,\ v=j+q-1\\ j=1+r+t,} } 
		(-1)^{rq+t+pj}
		(\mu_v\otimes m_{ij})(1^{\otimes r}\otimes m_{pq}\otimes 1^{\otimes t})\tau_v
+\sum_{\substack{u=i+p} }(-1)^{p}(\mu_{i1}\otimes 1_A)
(\mu_v\otimes m_{pv})\tau_v\\
&\qquad\qquad+\sum_{\substack{u=i+p} }(-1)^{v+1+pv}
(\mu_v\otimes m_{iv})\tau_v(1^{\otimes r}_{\Lambda\otimes A}\otimes(\mu_{p1}\otimes 1_A)\otimes 1^{\otimes t}_{\Lambda\otimes A})
\\
&=\sum_{\substack{u=i+p} }(-1)^{p}(\mu_{i1}\mu_v\otimes m_{pv})\tau_v+
\sum_{\substack{u=i+p\\ r+t+1=v} }(-1)^{v+1+pv+ip+(i+v)(1-p)}
(\mu_v(1^{\otimes r}\otimes\mu_{p1}\otimes 1^{\otimes t})\otimes m_{iv})\tau_v\\
&=\sum_{\substack{u=i+p} }\left((-1)^{p}(\mu_{i1}\mu_v)\otimes m_{pv}+\sum_{\substack{v=r+t+1} }
(-1)^{p+1}
\mu_v(1^{\otimes r}\otimes\mu_{i1}\otimes 1^{\otimes t})\otimes m_{pv}\right)\tau_v=0.\qedhere
\end{align*}
\end{proof}

\begin{prop}\label{tensor_functorial}Let $(\Lambda,\mu_{ij})$ be a twisted dga. The above construction gives rise to a functor
$\Lambda\otimes -:\dAinf{R}\to \dAinf{R}$,
sending a morphism $f:A\to B$ of $dA_\infty$-algebras to the morphism
$\widehat{f}:\Lambda\otimes A\to\Lambda\otimes B$ given by
$\widehat f_{i1}=1_\Lambda\otimes f_{i1}$ and $\widehat f_{ij}= (\mu_j\otimes f_{ij})\tau_j$ for all $j\geq 2.$
Furthermore, for a $dA_\infty$-algebra $(A,m_{ij})$,
the construction above gives rise to a functor $-\otimes A$ sending a
strict morphism $f:\Lambda\rightarrow \Lambda'$ to the strict morphism $f\otimes 1_A: \Lambda\otimes A\rightarrow \Lambda'\otimes A$.
\end{prop}
\begin{proof}
The proof follows exactly the same lines of computation as the preceding proof.
\end{proof}

\subsection{$r$-homotopies and $r$-homotopy equivalences}
We next define a collection of functorial paths
indexed by an integer $r\geq 0$ on the category of $dA_\infty$-algebras, giving rise to the corresponding
notions of $r$-homotopy.

We will use specific twisted dgas defined in the following proposition whose proof is left to the reader.

\begin{prop}\label{comparison_paths}Let $r\geq 0$ be an integer.
Let $\Lambda_r$ be the bigraded module
generated by $e_-$ and $e_+$ in bidegree $(0,0)$ and $u$ in bidegree $(-r,1-r)$. 
The only non-trivial operations given by
\[\mu_{r1}(e_-)=-u,\  \mu_{r1}(e_+)=u,\ \mu_{02}(e_-,e_-)=e_-,\ \mu_{02}(e_+,e_+)=e_+,\ \mu_{02}(e_-,u)=\mu_{02}(u,e_+)=u,\]
make $\Lambda_r$ into a twisted dga.
The morphisms
\[\xymatrix{R\ar[r]^-{\iota}&\Lambda_r \ar@<1ex>[r]^{\partial^+} \ar@<-1ex>[r]_{\partial^-}&R}\,\,\,;\,\,\, \partial^\pm\circ \iota=1_R\]
given by $\partial^-(e_-)=1_R$, $\partial^+(e_+)=1_R$ and $\iota(1)=e_-+e_+$ and $0$ elsewhere are strict morphisms of twisted dgas. \qed
\end{prop}

\begin{defi}\label{functorial r path}Let $r\geq 0$ be an integer. The \textit{functorial $r$-path}
$P_r:\dAinf{R}\to \dAinf{R}$ is defined as $P_r:=\Lambda_r\otimes-$.
\end{defi}

Note that for a $dA_\infty$-algebra $A$, one has 
$P_r(A)_i^j=(Re_-\otimes A_i^j)\oplus (Ru\otimes A_{i+r}^{j+r-1})\oplus (Re_+ \otimes A_i^j)$. Hence we may identify $P_r(A)$
with the bigraded $R$-module given by
$P_r(A)_i^j=A_i^j\oplus A_{i+r}^{j+r-1}\oplus A_i^j$ as in Definition~\ref{def:rpathtc}. More precisely, the triple $(x,y,z)$ is identified with
$e_-\otimes x+ u\otimes y+e_+\otimes z$.

Given bigraded $R$-modules $A$ and $B$, let
\[t_2:P_r(A)\otimes P_r(B)\lra P_r(A\otimes B)\]
be the map given by
\[t_2((x,y,z)\otimes(x',y',z'))=(x\otimes x',\overline x\otimes y'+y\otimes z',z\otimes z'),\]
where $\overline x:=(-1)^{rx_1+(1-r)x_2} x$ and $(x_1,x_2)$ denotes the bidegree of $x$.
Likewise, for $n\geq 2$ we let \[t_n:P_r(A_1)\otimes\cdots \otimes P_r(A_n)\lra P_r(A_1\otimes\cdots\otimes A_n)\]
be the map given by
\[t_n((x_1,y_1,z_1)\otimes\cdots \otimes (x_n,y_n,z_n))=
(x_1\otimes\cdots\otimes x_n,\sum_{1\leq j\leq n}\overline x_1\otimes\cdots \otimes \overline x_{j-1}\otimes y_{j}
\otimes z_{j+1}\otimes\cdots\otimes z_{n},z_1\otimes\cdots\otimes z_n).\]

Note that under the identification above, the map $t_n$ is obtained as the composite 
\[(\mu_n\otimes 1)\circ \tau_n: (\Lambda_r\otimes A_1)\otimes\cdots\otimes (\Lambda_r\otimes A_n)\rightarrow (\Lambda_r)^{\otimes n}\otimes A_1\otimes\cdots\otimes A_n\rightarrow \Lambda_r\otimes A_1\otimes\cdots\otimes A_n.
\]

As a consequence, combining Propositions \ref{P:tensorproduct} and  \ref{tensor_functorial} one obtains the following.

\begin{prop}\label{rpathdAinf}
The \textit{$r$-path} $(P_r(A), M_{ij})$ of a $dA_\infty$-algebra $(A,m_{ij})$ is 
given by the bigraded module $P_r(A)$
together with the morphisms $M_{ij}:P_r(A)^{\otimes j}\to P_r(A)$ of bidegree $(-i,2-i-j)$ given by
\[M_{r1}:=\left(
\begin{matrix}
m_{r1}&0&0\\
{-1}&-m_{r1}&{1}\\
0&0&m_{r1}
\end{matrix}
\right)\text{ and }
M_{i1}:=\left(
\begin{matrix}
m_{i1}&0&0\\
0&{(-1)^{i+r+1}}m_{i1}&0\\
0&0&m_{i1}
\end{matrix}
\right)\text{ for $i\neq r$}\]
 and the morphisms 

\[M_{ij}:=\left(
\begin{matrix}
m_{ij}&0&0\\
0&(-1)^{rj+i+j}m_{ij}&0\\
0&0&m_{ij}
\end{matrix}
\right)\circ t_j,\text{  for $i\geq 0$ and $j\geq 2$.}\]

The $r$-path of a morphism $f:(A,m_{ij}^A)\to (B,m_{ij}^B)$ of $dA_\infty$-algebras is the morphism of $dA_\infty$-algebras
$P_r(f):(P_r(A),M_{ij}^A)\to (P_r(B),M_{ij}^B)$
given by
$P_r(f)_{ij}=(f_{ij},(-1)^{(r+1)(j-1)+i}f_{ij},f_{ij}).$

The structure morphisms of the $r$-path
\[\xymatrix{A\ar[r]^-{\iota_A}&P_r(A) \ar@<1ex>[r]^{\partial^+_A} \ar@<-1ex>[r]_{\partial^-_A}&A}\,\,\,;\,\,\, \partial^\pm_A\circ \iota_A=1_A\]
are
given by $\partial^-_A(x,y,z)=x$, $\partial^+_A(x,y,z)=z$ and $\iota_A(x)=(x,0,x)$.
\qed
\end{prop}

It follows directly from the above proposition that the $r$-path is compatible with the forgetful functor $U:\dAinf{R}\lra \tc$.
Also, if $(A,m_{0j})$ is a $dA_\infty$-algebra concentrated in horizontal degree 0, then 
its $0$-path $P_0(A)$ coincides with its path object as an $A_\infty$-algebra as defined by Grandis in \cite{Grandis}.
Hence the $0$-path is compatible with the inclusion $\Ainf{R}\hookrightarrow \dAinf{R}$.

\begin{defi} 
\label{def:rhomda}
Let $f,g:A\to B$ be two morphisms of $dA_\infty$-algebras.
An \textit{$r$-homotopy from $f$ to $g$} is given by a morphism of $dA_\infty$-algebras $h:A\to P_r(B)$ such that 
$\partial^-_B\circ h=f$ and $\partial^+_B\circ h=g$.
We use the notation $h:f\simr{r} g$.
\end{defi}
We postpone until later giving an explicit version of $r$-homotopy, in terms of
a collection of morphisms $\widehat h_{ij}:A^{\otimes j}\to B$; see Proposition~\ref{equivalent_htp_dAinf}.

In the category of $A_\infty$-algebras, the notion of homotopy defines an equivalence relation
on the sets of $A_\infty$-morphisms (see \cite{Proute}, see also \cite{Grandis}).
We next prove an analogous result in the context of $dA_\infty$-algebras, for $0$-homotopies.
Our proof is an adaptation of the proof given by Grandis for $A_\infty$-algebras.
\begin{prop}
The notion of $r$-homotopy is reflexive and compatible with the composition.
Furthermore, for $r=0$ it defines an equivalence relation on the set of morphisms of $dA_\infty$-algebras from $A$ to $B$,
provided $A$ and $B$ are $(\NN,\ZZ)$-graded.
\end{prop}
\begin{proof}
Since the notion of $r$-homotopy is defined via a functorial path, it is reflexive and compatible with the composition.
To show that 0-homotopy is symmetric we will define a natural reversion morphism of the $r$-path 
$\zeta:P_0(A)\to P_0(A)$ of a $dA_\infty$-algebra $(A,m_{ij})$ such that $\partial^\pm\zeta=\partial^{\mp}$.
Then, given a $0$-homotopy $h:f\simr{0}g$ we will have a $0$-homotopy $\zeta h:g\simr{0}f$.

Consider the filtered $A_\infty$-algebra defined by applying $\Tot$ to $P_0(A)$. This is given by:
\[F_p\Tot(P_0(A))^n=F_p\Tot(A)^n\oplus F_p\Tot(A)^{n-1}\oplus F_p\Tot(A)^n,\]
with structure morphisms
\[M_1=\left(
\begin{matrix}
\Tot(m_{i1})&0&0\\
{-1}&-\Tot(m_{i1})&{1}\\
0&0&\Tot(m_{i1})
\end{matrix}
\right)\text{ and }
M_{j}=\left(
\begin{matrix}
\Tot(m_{ij})&0&0\\
0&(-1)^{j}\Tot(m_{ij})&0\\
0&0&\Tot(m_{ij})
\end{matrix}
\right)\circ t_j,\]
for $j\geq 2$.
We will next define a morphism of filtered $A_\infty$-algebras $\zeta:\Tot(P_0(A))\to \Tot(P_0(A))$.
Note that such a map is determined by its composition with the three maps
$\partial^-$, $\partial^0$ and $\partial^+$ defined by projection to each of the direct summands of $\Tot(P_0(A))$.
We let $\partial^-\zeta_{1}=\partial^+$, 
$\partial^0\zeta_{1}=-\partial^0$ and $\partial^-\zeta_{1}=\partial^+$.
For $j>1$, we let $\partial^{\pm}\zeta_{j}=0$ and define $\partial^0\zeta_{j}$ inductively by
	\[
		\partial^0\zeta_{j}=\sum\limits_{\substack{p+q=n+1\\ x+y+z=n+1}}S_{nqxp}\Tot(m_{ij})			((\partial^+)^{\otimes x-1}\otimes \zeta_{q}\otimes (\partial^-)^{\otimes y-1}\otimes \zeta_{1}\otimes 	(\partial^+)^{\otimes z-1},\]
where all indices in the sum are poisitive integers and $S_{nqxp}$ is a sign coefficient (see~\cite[p56]{Grandis}).
By~\cite[Theorem 7.1]{Grandis}, the family $\{\zeta_{j}\}_{j\geq 1}$ is a morphism of $A_\infty$-algebras.
Since for all $p\in\ZZ$, $\partial^\epsilon(F_p\Tot(P_0(A))\subset F_p P_0(A)$ for $\epsilon\in\{-,0,+\}$, the morphism
$\zeta$  is compatible with filtrations. Therefore by Corollary~\ref{dA-filtered-bounded-sf}
it gives the desired reversion of $dA_\infty$-algebras.

We next prove transitivity. Consider the pull-back of $dA_\infty$-algebras
\[
\xymatrix{
\pb\ar[d]_{\pi^+}
\Qq(A)\ar[r]^{\pi^-}&P_0(A)\ar[d]^{\partial^+}\\
P_0(A)\ar[r]^{\partial^-}&A.
}
\]
To prove transitivity it suffices to define a
morphism $\xi:\Qq(A)\lra P_0(A)$ of $dA_\infty$-algebras such that
$\partial^\pm \xi=\partial^{\pm}\pi^{\pm}$ (see for example~\cite[Proposition I.4.5(b)]{KampsPorter}).

Consider the filtered $A_\infty$-algebra given by  applying $\Tot$ to $\Qq(A)$. 
We will denote by $\partial^{\epsilon\eta}:=\partial^\eta\pi^\epsilon$,
with $\epsilon=\pm$ and $\eta\in\{-,0,+\}$,
the five projections $\Tot(\Qq(A))\to \Tot(A)$, noting that 
$\partial^{+-}=\partial^{-+}$.

We next define $\xi:\Tot(\Qq(A))\lra \Tot(P_0(A))$.
Let $\xi_1$ be defined by $\partial^-\xi_1:=\partial^{--}$, $\partial^0\xi_1:=\partial^{0-}+\partial^{0+}$ and 
$\partial^+\xi_1:=\partial^{++}$.
For $j>1$, we let $\partial^{\pm}\xi_j=0$ and define the central components by letting
\[\partial^0\xi_j=(-1)^j\Tot(m_{ij})\sum\limits_{\substack{x+y+z=j+1\\ x,y,z\geq 1}}
(\partial^{--})^{\otimes x-1}\otimes \partial^{0-}\otimes (\partial^{+-})^{\otimes y-1}
\otimes \partial^{0+}\otimes (\partial^{++})^{\otimes z-1}.
\]
By~\cite[Theorem 6.3]{Grandis}, the family $\xi=\{\xi_j\}_{j\geq 1}$ is a morphism of $A_\infty$-algebras. By construction, it is 
compatible with filtrations. 
Therefore by Corollary~\ref{dA-filtered-bounded-sf}
it gives the desired morphism of $dA_\infty$-algebras.
\end{proof}

\begin{rmk}
The proof of symmetry and transitivity of 0-homotopies given above does not extend to $r$-homotopies, due to the fact that
for $r>0$, the projection $\partial^0:\Tot(P_r(A))\to \Tot(A)$ is not necessarily compatible with filtrations. Note that we have
$\partial^0(F_p\Tot(P_r(A)))\subset F_{p+r}(\Tot(A))$. 
\end{rmk}

Denote by $\simre{r}$ the congruence of $\dAinf{R}$ generated by $r$-homotopies:
$f\simre{r}g$ if and only if there is a chain of $r$-homotopies
$f\simr{r}\cdots\simr{r}g$ from $f$ to $g$ or a chain $g\simr{r}\cdots\simr{r}f$ from $g$ to $f$.
\begin{defi}
A morphism of $dA_\infty$-algebras $f:A\to B$ is called an \textit{$r$-homotopy equivalence} if 
there exists a morphism $g:B\to A$ satisfying $f\circ g\simre{r} 1_B$ and $g\circ f\simre{r} 1_A$.
\end{defi}
Denote by $\Ss_r$ the class of $r$-homotopy equivalences of $\dAinf{R}$. 
This class is closed under composition and contains all isomorphisms.
Since the $r$-path commutes with the forgetful functor $U:\dAinf{R}\lra \tc$, we have $\Ss_r=U^{-1}(\Ss_r^{\tc})$.
Note as well that $\Ss_r\subset \Ss_{r+1}$ and $\Ss_r\subset \Ee_r$ for all $r\geq 0$.

\begin{lem}\label{iotahtpAlg}
Let $(A,m_{ij})$ be a $dA_\infty$-algebra. The strict morphism $\iota_A:(A,m_{ij})\lra (P_r(A),M_{ij})$
given by $\iota_A(x)=(x,0,x)$ is an $r$-homotopy equivalence.
\end{lem}
\begin{proof}
Note first that the category of twisted dgas together with strict morphisms is a monoidal category, hence $\Lambda_r\otimes\Lambda_r$ is
a twisted dga.
Let $\Delta:\Lambda_r\to \Lambda_r\otimes\Lambda_r$ be the map  given by
\[\Delta(e_-)=e_-\otimes (e_-+e_+)+e_+\otimes e_-,\quad  \Delta(e_+)=e_+\otimes e_+,\ \text{ and } \Delta(u)=u\otimes e_++e_+\otimes u.\]
That $\Delta$ is a strict morphism of dgas is a matter of computation. Furthermore one has
\[(\partial^+\otimes 1)\Delta=id \text{ and } (\partial^-\otimes 1)\Delta=\iota\circ \partial^-.\]
Consequently $\Delta\otimes 1_A:P_r(A)\rightarrow P_r(P_r(A))$ is an $r$-homotopy from $\iota_A\circ\partial^-_A$ to the identity.
\end{proof}

\begin{teo}\label{QuotientcatdAinf}
The localized category $\dAinf{R}[\Ss_r^{-1}]$
is canonically isomorphic to the quotient category  $\pi_r(\dAinf{R}):=\dAinf{R}/\simre{r}$.
\end{teo}
\begin{proof}
The proof is analogous to that of Proposition $\ref{QuotientcatTC}$, using Lemma $\ref{iotahtpAlg}$.
\end{proof}

\subsection{Operadic approach}
Since $dA_\infty$-algebras are algebras for the operad $(\dAs)_{\infty}$, one expects to be able
to describe homotopies in terms of structure on cofree  $(\dAs)^{\antishriek}$-coalgebras, where
 $(\dAs)^{\antishriek}$ is the Koszul dual cooperad. We carry this out in this section, giving several
equivalent formulations and showing that they agree with the definition via a path object presented above.

\subsubsection{$(\dAs)^{\antishriek}$-coalgebras and coderivations}\label{section:dascoder}
In this setting, a homotopy between 
morphisms $g$ and $f$ should be a \emph{coderivation homotopy}, that is, it satisfies two conditions,
a usual homotopy relation together with
a condition of compatibility with the comultiplication, called a $(g,f)$-coderivation 
condition. (See~\cite{ALRWZ} for the coderivation notion in an operadic context and~\cite{LefHas}
for $(g,f)$-coderivations in the setting of $A_\infty$-algebras.)

First we recall the $(g,f)$-coderivation condition for coassociative coalgebras.

\begin{defi}
Let $(A,\Delta^A)$, $(B, \Delta^B)$ be coassociative $R$-coalgebras and let $f,g:A\to B$ be coalgebra
morphisms. A \emph{$(g,f)$-coderivation} is an  $R$-linear map $h:A\to B$ such that
\[
(g\otimes h + h\otimes f)\Delta^A=\Delta^B h.
\]
\end{defi}

Next we define $(g,f)$-coderivations in a suitable operadic setting.

\begin{defi}
Let $\Cc$ be a non-symmetric cooperad in vertical bicomplexes. For $X,Y,Z$ vertical bicomplexes,
the vertical bicomplex $\Cc(X;Y;Z)$ is given by 
	\[
	\Cc(X;Y;Z) := \bigoplus_{n\geq 1} \Cc(n) \otimes 
	\bigl(\bigoplus_{a+b+1=n} X^{\otimes a} \otimes Y \otimes Z^{\otimes b}\bigr).
	\]
If $f\colon X \to X'$, $h: Y\to Y'$ and $g\colon Z \to Z'$ are maps of vertical bicomplexes, the map 
\[\Cc(f;h;g)\colon \Cc(X;Y;Z) \to \Cc(X';Y';Z')\]  is defined as the direct sum of the maps $ 1\otimes f^{\otimes a} \otimes h \otimes g^{\otimes b}$.
\end{defi}

\begin{defi}
Let $\Cc$ be a non-symmetric cooperad in vertical bicomplexes and let $A$ and $B$ be vertical bicomplexes. Let $g$ and $f$ be maps of $\Cc$-coalgebras $\Cc(A)\to \Cc(B)$. For $r\geq 0$, an \emph{$r$-$(g,f)$-coderivation} is a map of vertical bicomplexes
$h:\Cc(A)\to \Cc(B)$ of bidegree $(r,r-1)$ such that the following diagram commutes.

\[
\xymatrix{
\Cc(A)=\Cc(A;A;A) \ar[r]^-{\Delta_{\Cc}} 
\ar[d]_-{h}
\ar[r] & \Cc \circ\Cc(A;A;A)\cong
		\Cc(\Cc(A);\Cc(A;A;A) ;\Cc(A)) \ar[d]^-{\Cc(g;h;f)}\\
\Cc(B)=\Cc(B;B;B)\ar[r]^-{\Delta_{\Cc}} &
 \Cc \circ\Cc(B;B;B)\cong
		\Cc(\Cc(B);\Cc(B;B;B) ;\Cc(B))
}
\]

\end{defi}

In order to understand very explicitly what such a thing looks like in the $(\dAs)^{\antishriek}$ case, we first recall
Proposition 3.2 of~\cite{ALRWZ} and then extend it to $(g,f)$-coderivations. As there, it is slightly simpler to use 
cooperadic suspension and work with the suspended cooperad  $\Lambda(\dAs)^{\antishriek}$.

Consider triples $(C,\Delta, f)$ where $(C, \Delta)$ is a conilpotent coassociative coalgebra
and  $f:C\to C$ is a linear map of bidegree $(1,1)$ satisfying   $(f\otimes 1)\Delta=(1\otimes f)\Delta=\Delta f$.
A morphism between two such triples is a morphism of coalgebras commuting with the given linear maps.

\begin{prop}\cite[Proposition 3.2]{ALRWZ}
\label{prop:dAscoalg}
Cooperadic suspension gives rise to an isomorphism of categories between the
category of conilpotent coalgebras over the cooperad  $(\dAs)^{\antishriek}$ and the category
of triples $(C,\Delta, f)$ as above.

 An operadic coderivation  of bidegree $(0,1)$ of a  
$(\dAs)^{\antishriek}$-coalgebra $S^{-1}C$ corresponds
on $(C,\Delta, f)$ to a coderivation of bidegree $(0,1)$ of the coalgebra $C$, anti-commuting with the linear map $f$.\qed\end{prop}

\begin{example} \cite[Example 3.3]{ALRWZ}\label{ex:FreedAsCoalg}
As an example we give the structure corresponding to the cofree
$\Lambda(\dAs)^{\antishriek}$-coalgebra cogenerated by $C$. 
We have $\Lambda(\dAs)^{\antishriek}(C)\cong R[x]\otimes\overline{T}C$, where $\overline{T}C$ denotes the reduced
tensor coalgebra on $C$.

The coalgebra structure is given by
	\[
	\Delta(x^i\otimes a_1\otimes \cdots \otimes a_n)
	=\sum_{k=1}^{n-1}\sum_{r+s=i} (-1)^\epsilon 
		(x^r\otimes  a_1\otimes \cdots \otimes a_k)\otimes (x^s\otimes  a_{k+1}\otimes \cdots \otimes a_n),
	\]
where $\epsilon= rn+ik+(s,s)(|a_1|+\cdots + |a_k|)$.

Let $\pi_0$ denote the projection of $R[x]\otimes\overline{T}C$ onto 
$R x^0\otimes\overline{T}C\cong\overline{T}C$. Then
\[\Delta\pi_0=(\pi_0\otimes\pi_0)\Delta\] where the first $\Delta$ is the usual deconcatenation product defined on $\overline{T}C$.

 The linear map $f$ will be denoted $d_x$ in this cofree case and
	\[
	d_x:R[x]\otimes\overline{T}C\to R[x]\otimes\overline{T}C
	\] 
is  determined by $d_x(x^n\otimes a)=(-1)^{j+1} x^{n-1}\otimes a$, for $a\in C^{\otimes j}$.
\end{example}

\begin{prop}
\label{prop:matchcoders}
Let $g$ and $f$ be maps of $(\dAs)^{\antishriek}$-coalgebras $(\dAs)^{\antishriek}(A)\to (\dAs)^{\antishriek}(B)$
and let $h:(\dAs)^{\antishriek}(A)\to (\dAs)^{\antishriek}(B)$ be an $r$-$(g,f)$-coderivation. If $f$, $g$ correspond under 
 the above isomorphism of categories to coalgebra morphisms $\tilde{f}, \tilde{g}:
R[x]\otimes\overline{T}SA\to R[x]\otimes \overline{T}SB$,
(commuting with $d_x$), then $h$ corresponds to a $(\tilde{g},\tilde{f})$-coderivation of 
coalgebras  $\tilde{h}:R[x]\otimes\overline{T}SA\to R[x]\otimes \overline{T}SB$ of bidegree $(r,r-1)$, 
graded commuting with $d_x$.
\end{prop}

\begin{proof}
After applying cooperadic suspension, the $(g,f)$-coderivation condition gives the following
commutative diagram.
\[
\xymatrix{
R[x]\otimes\overline{T}SA \ar[r]^-{\rho_A} 
\ar[d]_-{\tilde{h}} \ar[r] & R[x]\otimes\overline{T}(R[x]\otimes\overline{T}SA)
 \ar[d]^-{(\dAs)^{\antishriek}(\tilde{g};\tilde{h};\tilde{f})}\\
R[x]\otimes\overline{T}SB\ar[r]^-{\rho_B} &
R[x]\otimes\overline{T}(R[x]\otimes\overline{T}SB)
}
\]
Let $C=SA$ and $D=SB$.

As in~\cite[Example 3.3]{ALRWZ}, the structure map $\rho_C$ of the cofree coalgebra on $C$ is
completely determined by $\Delta=\rho_{0,2}$ and $d_x=\rho_{1,1}$, where
 $\rho_{i,j}: C\to  C^{\otimes j}$ is the following composite
\[
\rho_{i,j} \colon \xymatrix{ C \ar[r]^-{\rho} & R[x]\otimes\overline{T}C \ar@{>>}[r]^{\pi_{i,j}} & 
R x^i \otimes C^{\otimes j} \ar[r]^-{\cong} & C^{\otimes j}}.
\]

Post-composing the horizontal maps with $\pi_{0,2}$ in the commutative diagram above, we find that we obtain a commutative diagram
\[
\xymatrix{
R[x]\otimes\overline{T}C \ar[r]^-{\rho_A} \ar[d]_-{\tilde{h}}  
& R[x]\otimes\overline{T}(R[x]\otimes\overline{T}C)
 \ar[d]^-{(\dAs)^{\antishriek}(\tilde{g};\tilde{h};\tilde{f})}
\ar@{>>}[r]^-{\pi_{0,2}}
& kx^0\otimes (R[x]\otimes\overline{T}C)^{\otimes 2}\cong (R[x]\otimes\overline{T}C)^{\otimes 2}
\ar[d]^{\tilde{g}\otimes \tilde{h}+\tilde{h}\otimes\tilde{f}}\\
R[x]\otimes\overline{T}D\ar[r]^-{\rho_B} &
R[x]\otimes\overline{T}(R[x]\otimes\overline{T}D)\ar@{>>}[r]^-{\pi_{0,2}}
& kx^0\otimes (R[x]\otimes\overline{T}D)^{\otimes 2}\cong (R[x]\otimes\overline{T}D)^{\otimes 2}
}
\]
and the outer commuting square gives the $r$-$(\tilde{g}, \tilde{f})$-coderivation condition for $\tilde{h}$.

Similarly, one may post-compose  the horizontal maps with $\pi_{1,1}$ and check that the resulting condition is 
\[
d_x^D\tilde{h}=(-1)^{|\tilde{h}|}\tilde{h}d_x^C. \qedhere
\]
\phantom\qedhere
\end{proof}

Next we want to pass from morphisms on $R[x]\otimes\overline{T}C$, commuting with $d_x$, to families of morphisms
on $\overline{T}C$.
As in Example 3.3 of~\cite{ALRWZ}, given a morphism of bigraded modules  $f:R[x]\otimes\overline{T}C \to R[x]\otimes\overline{T}C$, write
	\[
	f(x^n\otimes a)=\sum_{i} x^i\otimes f^{n,i}(a),
	\] 
where $f^{n,i}:\overline{T}(C)\to\overline{T}(D)$ and $a\in C^{\otimes j}$. 
Then 
commuting with the map $d_x$ means that $f$ is completely determined by the family of maps $f^{n,0}$.

Define $f_n:\overline{T}C \to\overline{T}C$ 
by $f_n(a)=(-1)^{nj}f^{n,0}(a)=(-1)^{nj}\pi_0f(x^n\otimes a)$,  where $a\in C^{\otimes j}$.

The correspondence gives a bijection between the subset of morphisms in 
$\bimod(R[x]\otimes \overline{T}C, R[x]\otimes \overline{T}D)$ which commute with $d_x$
and $\bimodinf( \overline{T}C, \overline{T}D)$.

Recall from~\cite{ALRWZ} that under this assignment, a square-zero coderivation $\delta$ on
$R[x]\otimes\overline{T}C$ corresponds to a twisted complex structure on $\overline{T}C$.

\begin{prop}\label{codertensor}
Let $\tilde{g},\tilde{f}$ be coalgebra morphisms $R[x]\otimes \overline{T}C\to R[x]\otimes \overline{T}D$, 
commuting with $d_x$.
If  $\tilde{h}$ is an $r$-$(\tilde{g},\tilde{f})$-coderivation of 
coalgebras  $\tilde{h}:R[x]\otimes \overline{T}C\to R[x]\otimes \overline{T}D$, graded commuting with $d_x$,
the corresponding family of morphisms $\tilde{h}_n: \overline{T}C \to  \overline{T}D$
satisfies
\[
	\left(\sum_j (-1)^j \tilde{g}_j\otimes \tilde{h}_{i-j} +\sum \tilde{h}_j\otimes \tilde{f}_{i-j}\right)\Delta=\Delta 	 
	 \tilde{h}_i \quad\text{for all $i\geq 0$.}
\]
\end{prop}

\begin{proof}
First note that for $\alpha$, $\beta: R[x]\otimes\overline{T}C\to R[x]\otimes \overline{T}D$, 
 $a\in C^{\otimes k}$ and $b\in C^{\otimes l}$,
we have
\[
	(\pi_0\otimes \pi_0)(\alpha\otimes\beta)(x^i\otimes a\otimes x^j\otimes b)
=(-1)^{i(u'+v')+j(a_1+a_2)+ik+jl}(\alpha_i\otimes \beta_j)(a\otimes b),
\]

Then we calculate $\Delta 	\tilde{h}_i(a)$ for $a=a_1\otimes \cdots\otimes a_n\in C^{\otimes n}$.
The $\epsilon$ appearing in the sign in the following calculation is given as in Example~\ref{ex:FreedAsCoalg} above. We also use that $|\tilde{h}|=-1$ and $|\tilde{f}|=0$.

\begin{align*}
	\Delta\tilde{h}_i(a)&=
			(-1)^{in}\Delta\pi_0\tilde{h}(x^i\otimes a)\\
	&=(-1)^{in}(\pi_0\otimes \pi_0)\Delta\tilde{h}(x^i\otimes a)\\
	&=(-1)^{in}(\pi_0\otimes \pi_0)(\tilde{g}\otimes\tilde{h}+\tilde{h}\otimes \tilde{f})\Delta(x^i\otimes a)\\
	&=(-1)^{in}(\pi_0\otimes \pi_0)(\tilde{g}\otimes\tilde{h}+\tilde{h}\otimes \tilde{f})
			\sum_{k=1}^{n-1}\sum_{s+t=i} (-1)^\epsilon 
		(x^s\otimes  a_1\otimes \cdots \otimes a_k)\otimes (x^t\otimes  a_{k+1}\otimes \cdots \otimes a_n)\\
	&=(-1)^{in}\sum_{k=1}^{n-1}\sum_{s+t=i} (-1)^{sn+ik+s+sk+t(n-k)}(\tilde{g_s}\otimes \tilde{h}_t)(
		(a_1\otimes\dots \otimes  a_k)\otimes (a_{k+1}\otimes \dots\otimes a_n))\\
	&\qquad + (-1)^{in}\sum_{k=1}^{n-1}\sum_{s+t=i} (-1)^{sn+ik+sk+t(n-k)}
			(\tilde{h}_s\otimes \tilde{f_t})(
		(a_1\otimes\dots \otimes  a_k)\otimes (a_{k+1}\otimes \dots\otimes a_n))\\
	&=\sum_{s+t=i} \left((-1)^s \tilde{g}_s\otimes \tilde{h}_t+\tilde{h}_s\otimes\tilde{f}_t\right)\Delta(a).
\qedhere
\end{align*}
\phantom\qedhere
\end{proof}

We now consider an $r$-shifted version of the usual homotopy relation and
 explain how $r$-homotopy of twisted complexes appears in this context.

\begin{defi}
For $F:\Lambda(\dAs)^{\antishriek}(C)\to \Lambda(\dAs)^{\antishriek}(D)$ a morphism of bigraded modules, let
$\mathbb{S}F:\Lambda(\dAs)^{\antishriek}(C)\to \Lambda(\dAs)^{\antishriek}(D)$ be given by $\mathbb{S}F:=-Fd_x^C$.
\end{defi}

\begin{prop}\label{prop:dAsshift}
The corresponding sequence of maps $\overline{T}C\to  \overline{T}D$ is given by
$(\mathbb{S}F)_i=F_{i-1}$ if $i\geq 1$ and $(\mathbb{S}F)_0=0$.
\end{prop}

\begin{proof}
For $a\in C^{\otimes n}$, and setting $F_{-1}=0$,
\begin{align*}
(\mathbb{S}F)_i(a)&=(-1)^{in}\pi_0(\mathbb{S}F)(x^{i}\otimes a)
				=(-1)^{in+1}\pi_0 Fd_x^C(x^{i}\otimes a)\\
		&=(-1)^{in+n}\pi_0 F(x^{i-1}\otimes a)
			=(-1)^{in+n+(i-1)n}F_{i-1}(a)\\
		&=F_{i-1}(a).\qedhere
\end{align*}
\phantom\qedhere
\end{proof}

We defined a shift $\bbS: \wbimod(A,B)_u^v\to \wbimod(A,B)_{u+1}^{v+1}$ in Definition~\ref{def:shift}.
Proposition~\ref{prop:dAsshift} shows that the one defined here corresponds to that one, hence we use the same notation.

\begin{prop}\label{rhomtensor}
Let $f,g:\Lambda(\dAs)^{\antishriek}(C)\to \Lambda(\dAs)^{\antishriek}(D)$ be morphisms of $\Lambda(\dAs)^{\antishriek}$-coalgebras
and let $h:\Lambda(\dAs)^{\antishriek}(C)\to \Lambda(\dAs)^{\antishriek}(D)$ be a morphism of bigraded modules of bidegree
$(r,r-1)$ satisfying
\[
	(-1)^r \delta^Dh+h\delta^C=\mathbb{S}^r(g-f).
\]
Then the corresponding family of morphisms $\tilde{h}_n:\overline{T}C \to  \overline{T}D$ gives
an $r$-homotopy of twisted complexes.
\end{prop}

\begin{proof}
We extract the $i$-th map in the families corresponding to the two sides of the given equation.
On the right-hand side, by Propostion~\ref{prop:dAsshift}, we obtain
$g_{i-r}-f_{i-r}$ for $i\geq r$ and $0$ for $i<r$.

Recall from~\cite{ALRWZ}, that one can write  $\delta^{i,j}(a)=\sum_l \delta^{i,j,l}(a)$, with $\delta^{i,j,l}(a)\in D^{\otimes l}$, and similarly for $h^{i,j}(a)$.
Anti-commuting with the map $d_x$ implies that
$\delta^{i,j,l}(a)=(-1)^{j(n+l+1)}\delta^{i-j,0,l}(a)$ for $a\in C^{\otimes n}$, and similarly for $h$.

On the left-hand side we calculate to check the signs.
For $a\in C^{\otimes n}$, 
\begin{align*}
\left((-1)^r \delta^Dh+h\delta^C \right)_i(a)&=
 (-1)^{in}\pi_0\left((-1)^r \delta^Dh+h\delta^C \right)(x^i\otimes a)\\
&=(-1)^{in+r}\pi_0\delta\big(\sum_j x^j\otimes h^{i,j}(a)\big)
+(-1)^{in}\pi_0h\big(\sum_j x^j\otimes \delta^{i,j}(a)\big)\\
&=(-1)^{in+r}\pi_0\big(\sum_{j,k} x^k\otimes \delta^{j,k}h^{i,j}(a)\big)
		+(-1)^{in}\pi_0\big(\sum_{j,k} x^k\otimes h^{j,k}\delta^{i,j}(a)\big)\\
&= (-1)^{in+r}\sum_{j} \delta^{j,0}h^{i,j}(a)+(-1)^{in}\sum_{j}  h^{j,0}\delta^{i,j}(a)\\
&=(-1)^{in+r}\big(\sum_{j,l}(-1)^{j(n+l+1)}  \delta^{j,0}h^{i-j,0,l}(a)\big)\\
	&\qquad	+(-1)^{in}\big(\sum_{j,l} (-1)^{j(n+l+1)}  h^{j,0}\delta^{i-j,0,l}(a)\big),\\
\end{align*}
where $\delta^{i,j}(a)=\sum_l \delta^{i,j,l}(a)$, with $\delta^{i,j,l}(a)\in D^{\otimes l}$, and similarly for $h^{i,j}(a)$.
Continuing the above calculation, we obtain
\begin{align*}
(-1)&^{in+r}\big(\sum_{j,l}(-1)^{j(n+1)}  \delta_{j}h^{i-j,0,l}(a)\big)
		+(-1)^{in}\big(\sum_{j,l} (-1)^{j(n+1)}  h_{j}\delta^{i-j,0,l}(a)\big)\\
&=(-1)^{in+r}\big(\sum_{j}(-1)^{j(n+1)+(i-j)n}  \delta_{j}h_{i-j}(a)\big)
		+(-1)^{in}\big(\sum_{j} (-1)^{j(n+1)+(i-j)n}  h_{j}\delta_{i-j}(a)\big)\\
&=\sum_j (-1)^{j+r} \delta_{j}h_{i-j}(a) + (-1)^j h_{j}\delta_{i-j}(a), 
\end{align*}
as required.
\end{proof}

Thus an $r$-shifted version of the operadic notion of coderivation homotopy 
corresponds to
a sequence of maps of bigraded $R$-modules $h_n:\overline{T}C \to  \overline{T}D$,
where $h_n$ has bidegree $(r-n,r-n-1)$,
satisfying both the condition in Proposition~\ref{codertensor} and the condition in
Proposition~\ref{rhomtensor}. 
We are going to show that this is equivalent to 
 an $r$-homotopy of $dA_\infty$-algebras, as defined via the path construction.

As an intermediate step, we reformulate the conditions using the composition in $\wbimod$.

\subsubsection{Tensor coalgebra viewed in $\wbimod$}
First we make two definitions about families of maps on reduced tensor coalgebras.
The first one is a coalgebra-morphism type condition for a family of maps; see also~\cite[Section 4]{Sag10}.

\begin{defi}
\label{def:coalgfamily}
Let $(\tilde{f}_p)\in \wbimod(\overline TSA,  \overline TSB)_0^0$. Write $\tilde{f}_{pq}^j$ for the map
$(SA)^{\otimes q}\to (SB)^{\otimes j}$ coming from $\tilde{f}_p$. We say that $(\tilde{f}_p)$ is a 
\emph{coalgebra-family of morphisms} if
for all $j, p,q$, we have
	\[
	\tilde{f}_{pq}^j=\sum_{\substack{p_1+\cdots +p_j=p\\ q_1+\cdots +q_j=q}} 
				\tilde{f}^1_{p_1q_1}\otimes \tilde{f}^1_{p_2q_2}\otimes \dots \otimes \tilde{f}^1_{p_jq_j}.
	\]
\end{defi}

Now we consider coderivations between such families.

\begin{defi}
\label{def:coderfamily}
Let  $(\tilde{f}_p), (\tilde{g}_p)\in \wbimod(\overline TSA,  \overline TSB)_0^0$ be two coalgebra-families of morphisms.
Let $(\tilde{h}_{p})\in  \wbimod(\overline TSA,  \overline TSB)_r^{r-1}$. Write $\tilde{h}_{pq}^j$ for the map
$(SA)^{\otimes q}, (SB)^{\otimes j}$ coming from $\tilde{h}_p$. We say that 
$(\tilde{h}_p)$ is an $r$-$((\tilde{g}_p),(\tilde{f}_p))$-\emph{coderivation-family of morphisms} if
for all $i,k,l$, we have

\[
	\tilde{h}_{pk}^l=\sum_{\substack{0\leq s\leq l-1\\
			p_1+\cdots+p_l=p\\
			q_1+\cdots + q_l=k}} (-1)^{p_1+\cdots+ p_s}
		 \tilde{g}_{p_1q_1}^1\otimes \cdots \otimes \tilde{g}_{p_sq_s}^1\otimes 
		 \tilde{h}_{p_{s+1}q_{s+1}}^1\otimes \tilde{f}_{p_{s+2}q_{s+2}}^1\otimes\cdots\otimes \tilde{f}_{p_lq_l}^1 . 
	\]

\end{defi}

\begin{notation}
For $A\in \bimod$, we let $\mathpzc{\overline{T}}SA$ denote the object of  $\wbimod$
given by the  underlying bigraded module of the reduced tensor 
coalgebra on $SA$. We let $\underline{\Delta}\in \wbimod(\mathpzc{\overline{T}}SA, \mathpzc{\overline{T}}SA\otimes \mathpzc{\overline{T}}SA)_0^0$ be given by $\underline{\Delta}=(\Delta_0, \Delta_1, \dots)$ 
with $\Delta_0:=\Delta$, the usual deconcatenation comultiplication and $\Delta_i:=0$ for $i>0$.

Similarly, for $f\in \hom_{\bimod}(A,B)$, write $\mathpzc{\overline{T}}Sf$ for 
 $\mathpzc{\overline{T}}Sf:=(\overline{T}Sf, 0, 0, \dots)\in \wbimod(\mathpzc{\overline{T}}SA, \mathpzc{\overline{T}}SB)$.
\end{notation}

In the light of the results of the previous section, and noting Remark~\ref{Ainftwoways}, 
one expects to be able to describe  $A_\infty$-algebras, their morphisms and so on, in terms of structure
in $\wbimod$ 
 on the pairs $(\mathpzc{\overline{T}}SA, \underline{\Delta})$. We give the details next.

 In the following
proposition, we use the natural notions of coderivations, coalgebra morphisms and $(g,f)$-coderivations in $\wbimod$.
For example, a coalgebra morphism $\mathpzc{\overline{T}}SA\to \mathpzc{\overline{T}}SB$
in $\wbimod$
means
a morphism in $\wbimod(\mathpzc{\overline{T}}SA, \mathpzc{\overline{T}}SB)_0^0$ satisfying
			\[c(f\widehat{\otimes} f, \underline{\Delta})=c(\underline{\Delta}, f).\]

\begin{prop}
\begin{enumerate}
\label{prop:wcoder}
\item Let $\delta\in \wbimod(\mathpzc{\overline{T}}SA, \mathpzc{\overline{T}}SA)_0^1$.
If $\delta$ is a coderivation such that $c(\delta, \delta)=0$ then it
 corresponds to
a collection of coderivations $\delta_i: \overline{T}SA\to \overline{T}SA$ in $\bimod$, $i\geq 0$, together
making $\overline{T}SA$ into a twisted complex. Thus,  a $dA_\infty$-algebra structure on a bigraded module $A$ is equivalent to specifying
such a square-zero coderivation $\delta$.
\item Let $f\in \wbimod(\mathpzc{\overline{T}}SA, \mathpzc{\overline{T}}SB)_0^0$. If $f$ is a coalgebra morphism
 commuting with given square-zero
coderivations $d^A$, $d^B$, it corresponds to
a coalgebra-family of morphisms $f_i: \overline{T}SA\to \overline{T}SB$, $i\geq 0$, together making a morphism
of twisted complexes.  Thus,  a morphism of $dA_\infty$-algebras from $A$ to $B$ is equivalent to specifying such
a coalgebra morphism $f$.
\item Let $f, g\in\wbimod(\mathpzc{\overline{T}}SA, \mathpzc{\overline{T}}SB)_0^0$
be coalgebra morphisms.\\
A $(g,f)$-coderivation $h\in\wbimod(\mathpzc{\overline{T}}SA, \mathpzc{\overline{T}}SB)_r^{r-1}$ 
corresponds to  an
$r$-$((g_i), (f_i))$-coderivation-family of
morphisms $h_i:\overline{T}SA\to \overline{T}SB$.
\item Let $f, g\in\wbimod(\mathpzc{\overline{T}}SA, \mathpzc{\overline{T}}SB)_0^0$
be coalgebra morphisms.\\
A morphism $h\in\wbimod(\mathpzc{\overline{T}}SA, \mathpzc{\overline{T}}SB)_r^{r-1}$  satisfying 
	\[
	(-1)^rc(d^{B},h)+c(h,d^A)=\bbS^r(g-f)
	\]
corresponds to an $r$-homotopy of morphisms of twisted complexes $h_i:\overline{T}SA\to \overline{T}SB$, $i\geq 0$.
\end{enumerate}
\end{prop}

\begin{proof}
\begin{enumerate}
\item  We have that $c(\delta, \delta)=0$ 
if and only if the $\delta_i$ satisfy the twisted complex relations. And the coderivation condition
on $\delta$ translates into the coderivation condition on the individual $\delta_i$. In more detail,
\begin{align*}
c(\delta\widehat{\otimes} 1 +1\widehat{\otimes} \delta,\underline{\Delta})&=c(\underline{\Delta},\delta)\\
\quad&\Longleftrightarrow \quad\left(c(\delta\widehat{\otimes} 1 +1\widehat{\otimes}\delta,\underline{\Delta})\right)_i=\left(c(\underline{\Delta},\delta)\right)_i
			\quad\text{for all $i\geq 0$}\\
	&\Longleftrightarrow \quad  (\delta\widehat{\otimes} 1 +1\widehat{\otimes} \delta)_i\Delta=\Delta\delta_i
				\quad\text{for all $i\geq 0$, since $\underline{\Delta}=(\Delta, 0, 0, \dots)$}\\		
	&\Longleftrightarrow \quad   (\delta_i\otimes 1 +1\otimes \delta_i)\Delta=\Delta\delta_i
	\qquad\text{for all $i\geq 0$, since $1=(1_A, 0, 0, \ldots)$.} 
\end{align*}
The final part follows from~\cite[4.1]{Sag10}. 

\item  We have that $c(f,d^A)=c(d^B, f)$
if and only if the $f_i$ satisfy the relations to be a morphism of twisted complexes
$(\overline{T}SA, d^A_i)\to (\overline{T}SB, d^{B}_i)$.  Then $f$ is a coalgebra morphism
means
\begin{align*}
c(f\widehat{\otimes} f,\underline{\Delta})=c(\underline{\Delta},f) \quad&\Longleftrightarrow 
	\quad\left(c(f\widehat{\otimes} f,\underline{\Delta})\right)_i=\left(c(\underline{\Delta},f)\right)_i
		\quad\text{for all $i\geq 0$}\\
		&\Longleftrightarrow \quad (f\widehat{\otimes} f)_i\Delta=\Delta f_i
				\quad\text{for all $i\geq 0$, since $\underline{\Delta}=(\Delta, 0, 0, \dots)$}\\		
	&\Longleftrightarrow \quad \sum_{j} (f_j\otimes f_{i-j})\Delta =\Delta f_i \quad\text{for all $i\geq 0$.}
\end{align*}
The last condition can be recursively reduced to the coalgebra-family of morphisms condition. The final part follows from~\cite[4.3]{Sag10}. 

\item 
\begin{align*}
c(g\widehat{\otimes} h+h\widehat{\otimes} f,\underline{\Delta})&=c(\underline{\Delta},h)\\ \quad&\Longleftrightarrow 
\quad \left(c(g\widehat{\otimes} h+h\widehat{\otimes} f,\underline{\Delta})\right)_i=\left((\underline{\Delta},h)\right)_i
			\quad\text{for all $i\geq 0$}\\
		&\Longleftrightarrow \quad 	 (g\widehat{\otimes} h+h\widehat{\otimes} f)_i \Delta=\Delta h_i
			\quad\text{for all $i\geq 0$, since $\underline{\Delta}=(\Delta_0, 0, 0, \dots)$}\\
		&\Longleftrightarrow \quad \left(\sum_j (-1)^j g_j\otimes h_{i-j} +\sum h_j\otimes f_{i-j}\right)\Delta=\Delta h_i \quad\text{for all $i\geq 0$.}
\end{align*}
The last condition recursively reduces to the coderivation-family one.
\item We calculate:
\begin{align*}
	(-1)^rc(d^{B},h)&+c(h,d^A)=\bbS^r(g-f)\\ \quad&\Longleftrightarrow 
\quad \left((-1)^rc(d^{B},h)+c(h,d^A)\right)_i=\left(\bbS^r(g-f)\right)_i \quad\text{for all $i\geq 0$}\\
&\Longleftrightarrow \quad 
	\sum_{j} (-1)^{j+r}d_j^{B}h_{i-j}+(-1)^j h_jd_{i-j}^A=
				\begin{cases} g_{i-r}-f_{i-r} &\text{if $i\geq r$},\\
				0&\text{if $i<r$.}
				\end{cases} 
\tag{$H_{i1}$}
\end{align*}
\end{enumerate}
\end{proof}

\subsubsection{Comparison with path definition}

We will show that the path definition of $r$-homotopy (Definition~\ref{def:rhomda})
corresponds to imposing conditions (3) and (4)
in Proposition~\ref{prop:wcoder}. In section~\ref{section:dascoder} we have checked that these
match up with the corresponding $(\dAs)^{\antishriek}$-coalgebra notions.

Let  $f,g:A\to B$ be two morphisms of $dA_\infty$-algebras and let  $h:A\to P_r(B)$
be an $r$-homotopy from $f$ to $g$. Recall that this is a morphism of $dA_\infty$-algebras
satisfying $\partial^-_B\circ h=f$ and $\partial^+_B\circ h=g$.

 Let $F,G:\mathpzc{\overline{T}}SA\to\mathpzc{\overline{T}}SB$ be the coalgebra morphisms  in $\wbimod$
 corresponding to $f,g$ and let $H:\mathpzc{\overline{T}}SA\to\mathpzc{\overline{T}}SP_r(B)$ be the coalgebra morphism  in $\wbimod$
corresponding to $h$.

We have three projection maps from  $P_r(B)$ to $B$, on to the left, middle and right copies of $B$, denoted
$\partial_B^-$, $\partial_B^0$ and $\partial_B^+$ respectively. We denote
the corresponding maps  $SP_r(B)$ to $SB$
by $\pi_L$, $\pi_M$ and $\pi_R$ respectively.
Then let $\pi\in\wbimod(\mathpzc{\overline{T}S}P_r(B),\mathpzc{\overline{T}S}B)$ be the strict map 
\[
\sum_{s,t} \pi_L^{\otimes s}\otimes\pi_M\otimes \pi_R^{\otimes t}.
\]

\begin{prop}
In the situation above, $c(\pi, H):\mathpzc{\overline{T}S}A\to\mathpzc{\overline{T}S}B$ is a $(G,F)$-coderivation
in  $\wbimod$.
\end{prop}

\begin{proof}
Consider the diagram
\[
\xymatrix{
\mathpzc{\overline{T}S}A\ar[d]_-{H}
\ar[r]^-{\underline{\Delta}}&\mathpzc{\overline{T}S}A\otimes \mathpzc{\overline{T}S}A
\ar[d]_-{H\widehat{\otimes} H}\\
\mathpzc{\overline{T}S}P_r(B)\ar[d]_-{\pi}
\ar[r]^-{\underline{\Delta}}&\mathpzc{\overline{T}S}P_r(B)\otimes \mathpzc{\overline{T}S}P_r(B)
\ar[d]^-{\overline{T}S(\pi_R){\widehat{\otimes}}\pi+\pi{\widehat{\otimes}} \overline{T}S(\pi_L)}\\
\mathpzc{\overline{T}S}B\ar[r]_{\underline{\Delta}} &\mathpzc{\overline{T}S}B\otimes \mathpzc{\overline{T}S}B\\
}
\]
We claim that this is a commutative diagram in the category $\wbimod$. The top square is such
a commutative diagram since $H$ is a coalgebra morphism in $\wbimod$. In the bottom square all the morphisms are strict,
so composition and tensor agree with the usual ones and we just need to see that this commutes in the usual sense,
which can be easily checked.

Finally, one can check that $c(\mathpzc{\overline{T}S}(\partial_B^+), H)=G$ and $c(\mathpzc{\overline{T}S}(\partial_B^-), H)=F$,
so the commuting of the outer square reads $c(G{\widehat{\otimes}} c(\pi, H)+c(\pi, H){\widehat{\otimes}} F),\underline{\Delta})=c(\underline{\Delta},c(\pi, H))$,
which is the $(G,F)$-coderivation condition for $c(\pi, H)$.
\end{proof}

\begin{lem}\label{lemma:homcalc}
\[
\left(c(\pi,d^{P_rB})+(-1)^r c(d^B,\pi)\right)_i =\begin{cases}
(\mathpzc{\overline{T}}(\pi_R) - \mathpzc{\overline{T}}(\pi_L)) &\text{if $i=r$},\\
0&\text{otherwise}.
\end{cases}
\]
\end{lem}

\begin{proof}
Since $\pi$ is a strict map, 
\[
\left(c(\pi,d^{P_rB})+(-1)^r c(d^B,\pi)\right)_i =\pi d^{P_rB}_i+(-1)^{r+i}d_i^B\pi.
\]
Now 
\[
d_i^B\pi=\sum_{j,s,t}(1_B^{\otimes s}\otimes \widetilde{m}_{ij}\otimes 1_B^{\otimes t})
			(\sum_{u,v}\pi_L^{\otimes u}\otimes \pi_M\otimes \pi_R^{\otimes v}),
\]
where, following Sagave's conventions, $\widetilde{m}_{ij}=\Psi_j(m_{ij}):SB^{\otimes j}\to SB$.

Thus we have
\[ \pi_{L}=S\partial_B^{-1}S^{-1},\ \quad  \pi_{R}=S\partial_B^{-1}S^{-1},\ \quad \pi_M=S\partial_B^0S^{-1}(-1)^{r-1},
\]
with $\pi_{L}$ and $\pi_R$ of bidegree $(0,0)$ and $\pi_M$ of bidegree $(r,1-r)$
and
\[ \widetilde{m}_{ij}=Sm_{ij}(S^{-1})^{\otimes j}(-1)^{1+i},\ \quad  \widetilde{M}_{ij}=SM_{ij}(S^{-1})^{\otimes j}(-1)^{1+i},
\]
both of bidegree $(-i,1-i)$.

When we expand out the composition, there are three types of terms appearing, according to
whether the input to $\widetilde{m}_{ij}$ is of the form $\pi_L^{\otimes j}$,
$\pi_L^{\otimes a}\otimes \pi_M\otimes \pi_R^{\otimes b}$ or $\pi_R^{\otimes j}$.

And 
\[
\pi d^{P_rB}_i=
(\sum_{u,v}\pi_L^{\otimes u}\otimes \pi_M\otimes \pi_R^{\otimes v})
	(\sum_{j,s,t}(1_{P_r(B)}^{\otimes s}\otimes \widetilde{M}_{ij}\otimes 1_{P_r(B)}^{\otimes t}),
\]
where  $\widetilde{M}_{ij}\Psi_j(M_{ij}):SP_r(B)^{\otimes j}\to SP_r(B)$.
Again there are three sorts of terms,  according to
whether they involve $\pi_L\widetilde{M}_{ij}$, $\pi_M\widetilde{M}_{ij}$ or $\pi_R\widetilde{M}_{ij}$.

From the definition of the $M_{ij}$ for $P_r(B)$, we have, for all $(i,j)$,
\[
\partial_B^- M_{ij}=m_{ij}(\partial_B^-)^{\otimes j},\qquad \partial_B^+ M_{ij}=m_{ij}(\partial_B^+)^{\otimes j}.  
\]
And
\begin{align*}
(-1)^{rj+i+j}\partial_B^0 M_{ij}&=\sum_{a+b+1=j}m_{ij}((\partial_B^-)^{\otimes a}\otimes \partial_B^0\otimes (\partial_B^+)^{\otimes b}),
\quad\text{for $(i,j)\neq (r,1)$,}    \\
\partial_B^0 M_{r1}&=-m_{r1}\partial_B^0+\partial_B^+-\partial_B^-.\\
\end{align*}
Sign calculations with the suspension show that these convert to
\[
\pi_L\widetilde{M}_{ij}=\widetilde{m}_{ij}\pi_L^{\otimes j},\qquad \pi_R\widetilde{M}_{ij}=\widetilde{m}_{ij}\pi_R^{\otimes j}, \qquad\text{for all $(i,j)$}  
\]
and 
\begin{align*}
(-1)^{r+i+1}\pi_M \widetilde{M}_{ij}&=\sum_{a+b+1=j}\widetilde{m}_{ij}(\pi_L^{\otimes a}\otimes \pi_M\otimes \pi_R^{\otimes b}),  \\
\pi_M \widetilde{M}_{r1}&=-\widetilde{m}_{r1}\pi_M+\pi_R-\pi_L.
\end{align*}

Using the above, one may now check that in $\left(c(\pi,d^{P_rB})+(-1)^r c(d^B,\pi)\right)_i $
almost
all terms cancel pairwise, with the exception, when $i=r$, of the extra terms in $c(\pi,d^{P_rB})$ coming
from the special form of $M_{r1}$. These contribute
\[
\sum_{s,t} 1^{\otimes s}\otimes (\pi_R-\pi_L)\otimes 1^{\otimes t}
=\sum_j \pi_R^{\otimes j}-\pi_L^{\otimes j}=(\mathpzc{\overline{T}}(\pi_R)-\mathpzc{\overline{T}}(\pi_L)).
\qedhere
\]
\phantom\qedhere
\end{proof}

\begin{prop}
The map $c(\pi, H):\mathpzc{\overline{T}}SA\to\mathpzc{\overline{T}}SB$
satisfies  the $r$-homotopy condition in $\wbimod$
of part (4) of Proposition~\ref{prop:wcoder}.
\end{prop}

\begin{proof}
Using associativity of composition (Lemma~\ref{associative_comp}), Lemma~\ref{lemma:homcalc}
and the relation $c(H,d^A)=c(d^{P_r(B)},H)$ , we have
\begin{align*}
\left(c(c(\pi, H),d^A)+(-1)^rc(d^B,c(\pi, H))\right)_i&= 
\left( c(c(\pi,d^{P_rB})+(-1)^r c(d^B,\pi), H) \right)_i \\
&=
(\mathpzc{\overline{T}}(\pi_R) - \mathpzc{\overline{T}}(\pi_L))H_{i-r}\\
&=G_{i-r}-F_{i-r}.
\end{align*}

So 
\[
c(c(\pi, H),d^A)+(-1)^r c(d^B,c(\pi, H))=
\mathbb{S}^r (G-F),
\]
as required.
\end{proof}

We now collect everything together.

\begin{teo}\label{r-homotopy-agree}
Let  $f,g:A\to B$ be two morphisms of $dA_\infty$-algebras.
Then $h:A\to P_r(B)$
being an $r$-homotopy from $f$ to $g$ in the sense of Definition~\ref{def:rhomda} is equivalent to
 conditions (3) and (4) of Proposition~\ref{prop:wcoder} on $c(\pi, H): \mathpzc{\overline{T}}SA\to \mathpzc{\overline{T}}SB$, where 
$H:\mathpzc{\overline{T}}SA\to\mathpzc{\overline{T}S}P_r(B)$ is the coalgebra morphism  in $\wbimod$
corresponding to $h$.
\end{teo}

\begin{proof}
We have already seen that $c(\pi, H)$ does satisfy conditions (3) and (4) of Proposition~\ref{prop:wcoder}. 
It is also straightforward to see that we can recover $h$ from $c(\pi, H)$.
Indeed, from $c(\pi, H)$ we can extract $G, F$ and $c(\pi_M, H^1)$,
 where $H^1: \mathpzc{\overline{T}}SA\to SP_r(B)$
is the composite of $H$ with the strict projection $\mathpzc{\overline{T}}SP_r(B)\to SP_r(B)$
and
$c(\pi, H)$ is uniquely determined by this data.
Now $H^1$ is also uniquely determined by  $G, F$ and $c(\pi_M, H^1)$ and we have $\tilde{h}_{ij}=H^1_{ij}$.
\end{proof}

Thus the notion of  $r$-homotopy defined via the
path construction  coincides with the operadic one.

\subsubsection{Explicit $r$-homotopy}

\begin{prop}\label{equivalent_htp_dAinf}
Giving an $r$-homotopy $h:A\to P_r(B)$ between morphisms of $dA_\infty$-algebras $f,g:A\to B$ 
is equivalent to giving a collection of morphisms  $h_{ik}:A^{\otimes k}\to B$ of
bidegree $(r-i,r-i-k)$, satisfying, for
all $m$ and $k$,
	\begin{align*}
		(-1)^{m-r}\sum_{i+p=m}\!
		\Big(
 \sum_l  m^B_{il}\!\!\!\sum_{\substack{0\leq s\leq l-1\\
			p_1+\cdots+p_l=p\\
			q_1+\cdots + q_l=k}}
		\!\!\!
(-1)^{p+\alpha+\sum_{u=1}^s p_u}
		  &g_{p_1q_1}\otimes \cdots \otimes g_{p_sq_s}\otimes h_{p_{s+1}q_{s+1}}\otimes 
	f_{p_{s+2}q_{s+2}}\otimes\cdots\otimes f_{p_lq_l}\\
			\qquad\quad+ \sum_l  h_{il}\sum_{\substack{q+s+t=k\\ q+1+t=l}}
	(-1)^{\beta}
			 1^{\otimes s}\otimes m^A_{pq}\otimes 1^{\otimes t}
						\Big)&\\
				&=  \left\{ 			\begin{array}{ll}
								0&\text{ if }m<r,\\
								g_{m-r,k}-f_{m-r,k}&\text{ if }m\geq r.
								\end{array}\right.  
	\tag{$H_{mk}$}
	\end{align*}

Here the signs are given by 
\begin{align*}
\alpha &=\sum_{u=1}^l(p_u+q_u)(l+u)+\sum_{u=1}^lq_u(\sum_{v=u+1}^l p_v+q_v)+(r-1)(l+1+s+\sum_{u=1}^s q_u),
\\
\beta &=sq+t+pl+r.
\end{align*}
\end{prop}

\begin{proof}
We have seen that the path definition of $r$-homotopy for $dA_\infty$-algebras agrees with the 
$r$-coderivation-family and $r$-homotopy of twisted complex conditions on the corresponding families of maps
 $\overline TSA\to  \overline TSB$.

Suppose that the
$dA_\infty$-algebra structures of $A$ and $B$ are encoded in families of coderivations $d_i^A$
and $d_i^B$ making $\overline TSA$ and $\overline TSB$ respectively into twisted chain complexes.
Suppose also that $f$ corresponds to the coalgebra-family of maps $\tilde{f}_i: \overline TSA\to  \overline TSB$, giving
a morphism of twisted complexes; similarly $g$ corresponds to  a coalgebra-family of maps $\tilde{g}_i: \overline TSA\to  \overline TSB$.

The $r$-homotopy condition on $(h_i)$ between maps of twisted complexes is equivalent to, for all $m\geq 0$,
	\begin{equation}\label{homotopy}
	\sum_{i+j=m} (-1)^{i+r} d_i^B\tilde{h}_j+(-1)^i \tilde{h}_i\delta_j^A=\left\{ 
								\begin{array}{ll}
								0&\text{ if }m<r,\\
								\tilde{g}_{m-r}-\tilde{f}_{m-r}&\text{ if }m\geq r.
								\end{array}\right.
								\tag{$H_{m1}$}
	\end{equation}
This is an equality of maps $\overline TSA\to  \overline TSB$. 
Here the bidegree of $\tilde{h}_i$ is $(r-i, r-i-1)$, that of $d^A_i$ and of $d^B_i$ is $(-i, -i+1)$ and that of $\tilde{f}_i$ and of $\tilde{g}_i$ is
$(-i,-i)$.

For each $m\geq 0$ and for $k\geq 1$, we consider the ``component of equation~(\ref{homotopy}) from 
$(SA)^{\otimes k}\to SB$''. That is, we pre-compose with
the inclusion $(SA)^{\otimes k}\to  \overline TSA$ and post-compose with the projection to $SB$. We will show that
this gives, after shifting, the required statement. 

The various conditions on the morphisms mean that everything is determined by components. 

Firstly, the coderivation
condition on each $d^A_i$ means that they are determined by components $(SA)^{\otimes k}\to SB$,
corresponding to $m_{ik}^A:A^{\otimes k}\to A$ and
similarly for $d^B_i$.
Secondly, since the $\tilde{f}_i :\overline TSA\to  \overline TSB$ form a coalgebra-family of morphisms, it is determined by components  
$\tilde{f}_{ik}=\tilde{f}_{ik}^1:(SA)^{\otimes k} \to  SB$ and similarly for $\tilde{g}_i$.  
And finally, recall that the $r$-$((\tilde{g}_i), (\tilde{f}_i))$-coderivation-family condition means that
$\tilde{h}_i$ is determined by components $\tilde{h}_{ik}=\tilde{h}_{ik}^1:(SA)^{\otimes k}\to SB$, where 
$\tilde{h}_{ik}^l:(SA)^{\otimes k}\to (SB)^{\otimes l}$ is given by
	\[
	\tilde{h}_{ik}^l=\sum_{\substack{s+1+t=l\\
			i_1+\cdots+i_s+p+l_1+\cdots+l_t=i\\
			k_1+\cdots + k_r+q+j_1+\cdots +j_t=k}} (-1)^{i_1+\cdots+ i_s}
		 \tilde{g}_{i_1k_1}\otimes \cdots \otimes \tilde{g}_{i_sk_s}\otimes \tilde{h}_{pq}\otimes \tilde{f}_{l_1j_1}\otimes\cdots\otimes \tilde{f}_{l_tj_t} . 
	\]

Let $f_{ik}=\Psi_k^{-1}(\tilde{f}_{ik}):A^{\otimes k} \to B$,  $g_{ik}=\Psi_k^{-1}(\tilde{g}_{ik}):A^{\otimes k} \to B$ and
 $h_{ik}=\Psi_k^{-1}(\tilde{h}_{ik}):A^{\otimes k} \to B$ and note that the bidegree of $h_{ik}$ is $(r-i, r-i-k)$.
We have  $\tilde{m}_{ik}=\Psi_k(m_{ik})=(-1)^{i+1}\sigma^{-1}m_{ik}\sigma^{\otimes k}$,
 $\tilde{f}_{ik}=\Psi_k(f_{ik})=(-1)^{i}\sigma^{-1}f_{ik}\sigma^{\otimes k}$ and similarly for $\tilde{g}_{ik}$ and
$\tilde{h}_{ik}=\Psi_k(h_{ik})=(-1)^{r+i+1}\sigma^{-1}h_{ik}\sigma^{\otimes k}$.

Then it is a matter of direct calculation that the extraction of the relevant component maps from equation~(\ref{homotopy}),
together with removing the shifts using the isomorphism $\Psi_k$,
gives the required result.
\end{proof}

\begin{rmk}
Sagave~\cite[Definition 4.9]{Sag10} defined homotopy of morphisms from $A$ to $B$ of 
$dA_\infty$-algebras only in the special case  where $A$ is minimal and $B$ is a bidga. One may check that
his definition is equivalent to our $0$-homotopy in that case.
\end{rmk}

\appendix
\section{List of notation for categories}

For easy reference, we provide a list of the notation for the main categories appearing  in this paper.

\begin{itemize}
\item $(\cpx,\otimes,R)$ is the category with objects cochain complexes and morphisms degree $0$ chain maps.   The unit $R$ is the cochain complex concentrated in degree zero.
\item $(\bimod,\otimes,R)$ is the category with objects bigraded modules and morphisms degree $(0,0)$ maps of bigraded modules.  The unit $R$ is as above.
\item $(\vbic,\otimes,R)$ is the category with objects vertical bicomplexes and morphisms degree $(0,0)$ maps of vertical bicomplexes.  The unit $R$ is as above.  See Definition \ref{def:vbicx}.
\item $(\tc,\otimes,R)$ is the category with objects twisted  complexes and morphisms degree $(0,0)$ maps of twisted  complexes i.e., infinity morphisms.  The unit $R$ is the twisted  complex concentrated in degree $(0,0)$.  See Definitions \ref{def:tc_objects} and \ref{def:tc_mor}.
\item $(\bimodinf,\otimes,R)$ is the full subcategory of $\tc$ whose objects are twisted complexes with trivial structure i.e., zero differentials.  
\item $(\fm, \otimes, R)$ is the symmetric monoidal category with objects filtered graded modules and morphisms degree $0$ morphisms which respect the filtration.  The unit is the base ring $R$ sitting in degree $0$ with trivial filtration.   See Definition \ref{def:fm_obj}.
\item $\sfm$ is the full subcategory of $(\fm, \otimes, R)$ whose objects are split filtered modules. See Definition \ref{def:sfc}.
\item $(\fc,\otimes, R)$ is the category with objects filtered complexes and morphisms degree $0$ chain maps that respect the filtration, the unit as above. See Definition \ref{def:fc_obj}.
\item $\sfc$ the full subcategory of $(\fc,\otimes, R)$ whose objects are split filtered complexes.  See Definition \ref{def:sfc}.
\item $\btc$, $\bvbic$, $\bbimod$ are the full subcategories  whose objects are $(\NN,\ZZ)$-graded twisted complexes, vertical bicomplexes and bigraded modules respectively. See Definition~\ref{bddcats}.
\item $\bfm$, $\bsfm$, $\bfc$, $\bsfc$ are the full subcategories  whose objects are (split) non-negatively filtered modules respectively complexes,
i.e.~the full subcategories with objects $(K, F)$ such that $F_p K^n=0$ for all $p<0$. See Definition~\ref{bddcats}.
\item $\Ainf{R}$ is the category of $A_\infty$-algebras over $R$.
\item $\dAinf{R}$ is the category of derived $A_\infty$-algebras over $R$.  See Definitions \ref{def:dA_obj} and \ref{def:dA_mor}.
\item $\Ainfintc$ is the category of derived $A_\infty$-algebras over $R$ in the category of twisted chain complexes.  See Definition \ref{def:aintc}.
\item $\fAinf(R)$ is the category of filtered $A_\infty$-algebras over $R$.  See Definitions \ref{def:fa_obj} and \ref{def:fa_mor}.
\item $\sfAinf(R)$, $\bfAinf(R)$ and $\bsfAinf(R)$ are the full subcategories whose objects are
split filtered $A_\infty$-algebras,  non-negatively filtered $A_\infty$-algebras and
split non-negatively filtered $A_\infty$-algebras respectively.  See Diagram (\ref{def:sub_cats_fa}).
\item  $\wbimod$ is the $\bimod$-enriched category of bigraded modules.  See Definition \ref{weird_bimod}.
\item  $\wtc$ is the $\vbic$-enriched category of twisted  complexes.  See Definition \ref{def:wtc}.
\item $\wbtc$  is the full subcategory of	$\wtc$ whose objects are $(\NN,\ZZ)$-graded twisted complexes.
\item $\wfm$ is the $\bimod$-enriched category of filtered graded modules.  See Definition \ref{def:wfm}.
\item  $\wfc$ is the $\vbic$-enriched category of filtered complexes.  See Definition \ref{def:wfc}.
\item $\wsfm$ and $\wbsfm$ are  the full subcategories of $\wfm$ whose objects are  split filtered modules
and split non-negatively filtered modules respectively.
\item $\wsfc$ and $\wbsfc$ are  the full subcategories of $\wfc$ whose objects are  split filtered complexes
and split non-negatively filtered complexes respectively.
\item $\widehat{\otimes}$ denotes the enriched monoidal structure on $\wtc$, $\wbimod$, $\wfm$ and $\wfc$.  See Lemmas \ref{lem: mon structure on wtc} and \ref{lem:monidal_wfc}.
\end{itemize}

\linespread{1}
\bibliographystyle{amsalpha}
\bibliography{bibliografia}

\end{document}